\documentclass[11pt,twoside]{article}

\usepackage{graphicx}
\usepackage{styleset}
\usepackage{macros}

\title	 {\normalfont\headfam\itshape\LARGE
          Knots, sutures and excision}

\author {P. B. Kronheimer and T. S. Mrowka%
            
            \footnote{%
             The work of the first author
            was supported by the National Science Foundation through
            NSF grant number DMS-0405271. The work of
            the second author was supported by NSF grants DMS-0206485,
            DMS-0244663 and DMS-0805841.} } 

\address {Harvard University, Cambridge MA 02138 \\
          Massachusetts Institute of Technology, Cambridge MA 02139}

\begin{document}

\pagestyle{numbered}	

\maketitle

\begin{abstract}
    We develop monopole and instanton Floer homology groups for balanced
    sutured manifolds, in the spirit of \cite{Juhasz-1}. Applications
    include a new proof of Property P for knots.
\end{abstract}

\tableofcontents

\section{Introduction}

Floer homology for sutured manifolds is an invariant $\SFH(M,\gamma)$
of ``balanced
sutured 3-manifolds'' $(M,\gamma)$, introduced by Juh\'asz in
\cite{Juhasz-1,Juhasz-2}. It incorporates the knot Floer homology of
Ozsv\'ath-Szab\'o and Rasmussen
\cite{Ozsvath-Szabo-knotfloer,Rasmussen-thesis} as a special case, and
it provides a framework in which to adapt the arguments of Ghiggini
and Ni \cite{Ghiggini,Ni-A,Ni-B} to reprove, for example, that knot Floer
homology detects fibered knots.

The construction that forms the basis of Juh\'asz's invariant is an
adaptation of Ozsv\'ath and Szab\'o's Heegaard Floer homology for
$3$-manifolds. 
The purpose of the present paper is to show how something very
similar can be done using either monopole Floer
homology \cite{KM-book} or instanton Floer homology
\cite{Donaldson-book} in place of the Heegaard version.
We will define an invariant of balanced sutured manifolds by gluing
them up, with some extra pieces, to form a closed manifold and then
applying ordinary Floer homology, of either variety, to this closed manifold.
Many of the theorems and constructions of Ghiggini, Ni and Juh\'asz can
be repeated in this context. In particular, our construction leads to
candidates for  ``monopole knot homology'' and ``instanton knot
homology'':  the monopole and instanton counterparts
of the Heegaard knot homology  groups. Adapting the arguments of
\cite{Ghiggini} and \cite{Ni-A}, we shall also prove that fibered knots
can be characterized using  either of these invariants.

The definition of instanton knot homology which
arises in this way, motivated by Juh\'asz's sutured manifold
framework, is not new. It turns out
to be exactly the same as an earlier instanton homology for knots, defined by
Floer twenty years ago~\cite{Floer-Durham-paper}. We conjecture that,
over a field of characteristic zero, the knot homology groups of
Ozsv\'ath-Szabo and Rasmussen are isomorphic to Floer's instanton knot
homology.

Monopole Floer homology for balanced sutured
manifolds is defined in section~\ref{sec:Monopole-Floer-homology},
and the definition is
adapted to the instanton case in section~\ref{sec:instantons}. The
same definition could be applied with Heegaard Floer homology: it
is not clear to the authors whether the resulting invariant of sutured
manifolds would be the same as Juh\'asz's invariant, but we would
conjecture that this is the case. It seems, at least, that
the construction  recaptures Heegaard knot homology
\cite{Rasmussen-2}. Some things are missing however.
Our construction leads to knot homology
groups which lack (a priori) the $\Z$ grading as well as the
additional structures that are present in the theory developed in
\cite{Ozsvath-Szabo-knotfloer} and \cite{Rasmussen-thesis}.

In the setting of instanton homology, we obtain new non-vanishing
theorems.  Among other applications, the non-vanishing theorems
lead to a new proof of
Property P for knots. In contrast to the proof in \cite{KM-Witten-P},
the argument presented here is independent of the work of Feehan and
Leness \cite{Feehan-Leness} concerning Witten's conjecture, and does
not require any tools from contact or symplectic topology. As a
related matter, we show that instanton homology captures the Thurston
norm on an irreducible $3$-manifold, answering a question raised in
\cite{KM-asymptotics}.

\subparagraph{Acknowledgments.} The authors would like to thank Andr\'as
Juh\'asz and Jake Rasmussen for helpful comments and corrections to
an earlier version of this paper.

 \section{Background on monopole Floer homology}

\subsection{Monopole Floer homology recalled}
\label{subsec:monopole-recap}
 
We follow the notation of \cite{KM-book} for monopole Floer homology.
Thus, to a closed, connected,  oriented $3$-manifold $Y$ equipped with a \spinc{}
structure $\s$, we associate three varieties of Floer homology
groups with integer coefficients,
\[
            \Hto_{\bullet}(Y,\s) , \quad \Hfrom_{\bullet}(Y,\s),\quad
            \Hred_{\bullet}(Y,\s).
\]
The notation using $\bullet$ in place of the more familiar $*$ was
introduced in \cite{KM-book} to denote that, in general, there is a
completion involved in the definition. In all that follows, the
distinction between $\Hfrom_{\bullet}$ and $\Hfrom_{*}$ does not
arise, but we preserve the former notation as a visual clue. Unless
$c_{1}(\s)$ is torsion, these groups are not $\Z$-graded, but they
always have a canonical $\Z/2$ grading.

The three varieties are related by a long exact sequence
\[
            \dots\to
            \Hred_{\bullet}(Y,\s)\stackrel{i}{\to}\Hto_{\bullet}(Y,\s)
            \stackrel{j}{\to}\Hfrom_{\bullet}(Y,\s)
            \stackrel{p}{\to}\Hred_{\bullet}(Y,\s)\to\cdots.
\]
If $c_{1}(\s)$ is not torsion, then $\Hred_{\bullet}(Y,\s)$ is zero
and $\Hto_{\bullet}(Y,\s)$ and $\Hfrom_{\bullet}(Y,\s)$ are
canonically isomorphic, via $j$. In this case, we simply write
$\HM_{\bullet}(Y,\s)$ for either $\Hto_{\bullet}(Y,\s)$ or
$\Hfrom_{\bullet}(Y,\s)$. All these groups can be non-zero only for
finitely many \spinc{} structures on a given $Y$: we write
\[
                \Hto_{\bullet}(Y) = \bigoplus_{\s}
                \Hto_{\bullet}(Y,\s)
\]
for the total Floer homology, taking the sum over all isomorphism
classes of \spinc{} structure, with similar notation for the $\Hfrom$
and $\Hred$ cases.

\subsection{Local coefficients}

We can also define a version of
Floer homology with a local system of coefficients. The following
definition is adapted from \cite[section 22.6]{KM-book}.
Let $\cR$ denote any commutative ring with $1$ supplied with an
``exponential map'', a group homomorphism
\begin{equation}\label{eq:exponential}
                \exp: \R \to \cR^{\times}.
\end{equation}
We will use polynomial notation for the exponential map, writing
\[
            t = \exp(1)
\]
and so writing $\exp(n)$ as $t^{n}$.
Let $\bonf(Y,\s)$ denote the Seiberg-Witten configuration space for a
\spinc{} structure $\s$ on $Y$; that is, $\bonf(Y,\s)$ is the space of
gauge equivalences classes $[A,\Phi]$ consisting of a \spinc{}
connection $A$ and a section $\Phi$ of the spin bundle. Given a
smooth $1$-cycle $\eta$ in $Y$ with real coefficients, we can
associate to each path $z: [0,1] \to \bonf(Y,\s)$ a real number $r(z)$
by
\[
                        r(z) = \frac{i}{2\pi} \int_{[0,1]\times
                        \eta}\tr F_{A_{z}},
\]
where $A_{z}$ is the $4$-dimensional connection on $[0,1]\times
Y$ arising from the path $z$. Now define a  local system
$\Gamma_{\eta}$ on $\bonf(Y,\s)$ by declaring its fiber at every point
to be $\cR$ and declaring the map $\cR\to\cR$ corresponding to a path $z$
to be multiplication by $t^{r(z)}$. Following \cite[section
22]{KM-book}, we obtain Floer homology groups with coefficients in
$\Gamma_{\eta}$;
they will be $\cR$-modules
 denoted
\[
            \Hto_{\bullet}(Y;\Gamma_{\eta}) ,
            \quad \Hfrom_{\bullet}(Y;\Gamma_{\eta}),\quad
            \Hred_{\bullet}(Y;\Gamma_{\eta}).
\]
These still admit a  direct sum decomposition by isomorphism classes
of \spinc{} structures. The following is essentially
Proposition~32.3.1 of \cite{KM-book}:

\begin{proposition}\label{prop:twist-vanish}
If there is an integer cohomology class that
evaluates as $1$ on $[\eta]$, and if $t-t^{-1}$ is invertible in
$\cR$, then
 $\Hred_{\bullet}(Y;\Gamma_{\eta})$ is zero; thus we again
have an isomorphism $j$ between  $\Hto_{\bullet}(Y;\Gamma_{\eta})$
and $\Hfrom_{\bullet}(Y;\Gamma_{\eta})$.
\end{proposition}

In the situation of the proposition, 
we once more drop the decorations and simply write
\[
                    \HM_{\bullet}(Y;\Gamma_{\eta}) =
                    \bigoplus_{\s}
                \HM_{\bullet}(Y,\s;\Gamma_{\eta})
\]
for this $\cR$-module.

\subsection{Cobordisms}

    Cobordisms between $3$-manifolds give rise to maps between their
Floer homology groups. More precisely, if $W$ is a compact, oriented
cobordism from $Y_{1}$ to $Y_{2}$, equipped with a
homology-orientation in the sense of \cite{KM-book}, then $W$ gives
rise to a map
\[
            \Hto(W) : \Hto_{\bullet}(Y_{1}) \to \Hto_{\bullet}(Y_{2})
\]
with similar maps on $\Hfrom_{\bullet}$ and $\Hred_{\bullet}$. If
$\eta_{1}$ and $\eta_{2}$ are $1$-cycles in $Y_{1}$ and $Y_{2}$
respectively, then to obtain a  maps between the Floer groups with
local coefficients, we need an additional piece of data: a $2$-chain
$\nu$ in $W$ with $\partial\nu = \eta_{2}-\eta_{1}$. In this case, we
obtain a map which we denote by
\[
          \Hto(W;\Gamma_{\nu}) :
          \Hto_{\bullet}(Y_{1};\Gamma_{\eta_{1}}) \to
          \Hto_{\bullet}(Y_{2};\eta_{2}).
\]

The map $\Hto(W)$ and its relatives are defined by taking a sum over
all \spinc{} structures on $W$. In the case of $\Hto(W;\Gamma_{\nu})$,
the \spinc{} contributions are weighted according to the pairing of
the curvature of the connection with the cycle $\nu$. There is a
corresponding invariant for a closed $4$-manifold $X$ with $b_{+}\ge
2$ containing a closed $2$-cycle $\nu$. In \cite{KM-book}, this
invariant of $(X,\nu)$ is denoted by $\m(X,\nu)$ (or $\m(X,[\nu])$,
because only the homology class of $\nu$ matters);
it is an element of $\cR$ defined by
    \begin{equation}\label{eq:SW-generating}
                        \m(X,[\nu]) = \sum_{\s} \m(X,\s)
                        t^{\langle c_{1}(\s),[\nu]\rangle},
    \end{equation}
where  $\m(X,\s)$ denotes the ordinary Seiberg-Witten
invariant for a \spinc{} structure $\s$.

\subsection{Adjunction inequalities and non-vanishing theorems}
\label{subsec:adjunction}

Monopole Floer homology detects the Thurston norm of a $3$-manifold
$Y$. We recall from \cite{KM-book} what lies behind this slogan. Let
$F\subset Y$ be a closed, oriented, connected surface in our closed,
oriented $3$-manifold $Y$. We shall suppose $F$ is not sphere.  Then
we have a vanishing theorem \cite[Corollary 40.1.2]{KM-book},
which states that
\[
                            \HM_{\bullet}(Y, \s) = 0
\]
for all \spinc{} structures $\s$ satisfying
\[
                    \langle c_{1}(\s) , [F] \rangle > 2\genus(F) -
                    2.
\]
(Note that this condition implies that $c_{1}(\s)$ is not torsion.)
This vanishing theorem is usually referred to as the ``adjunction
inequality''. Accompanying this result is a rather deeper
\emph{non}-vanishing
theorem, which we state (for the sake of simplicity) in the case that
the genus of $F$ is at least $2$. In this case, the non-vanishing
theorem asserts that if $F$ is genus-minimizing in its homology
class, then there exists a \spinc{} structure $\s_{c}$ with
\[
                            \HM_{\bullet}(Y, \s_{c}) \ne 0
\]
and
\[
                    \langle c_{1}(\s_{c}) , [F] \rangle = 2\genus(F) -
                    2.
\]
Slightly more specifically, Gabai's theorem from \cite{Gabai} tells us
that $Y$ admits a taut foliation having $F$ as a compact leaf. A
foliation in turn determines a \spinc{} structure on $Y$. The
non-vanishing result holds for any \spinc{} structure $\s_{c}$ arising
in this way. This result appears as Corollary~41.4.2 in
\cite{KM-book}. The techniques of this paper provide an alternative
proof, which we will explain in the context of instanton homology in
section~\ref{subsec:I-non-vanish} below.

It is convenient to introduce the following shorthand.
We denote the set of isomorphism classes of \spinc{} structures on a
closed oriented manifold $Y$ by $\Sp(Y)$. If $F\subset Y$ is a closed,
connected
oriented surface of genus $g\ge 2$, then we write $\Sp(Y|F)$ for the set of
isomorphisms classes of \spinc{} structures $\s$ on $Y$ satisfying the
constraint
\begin{equation}\label{eq:spinc-max}
            \langle c_{1}(\s), [F] \rangle = 2\genus(F)-2,
\end{equation}
and we write
\[
                \HM_{\bullet}(Y|F) \subset \HM_{\bullet}(Y)
\]
for the subgroup
\begin{equation}\label{eq:bar-F-notation}
            \HM_{\bullet}(Y|F) = \bigoplus_{\s \in \Sp(Y|F)}
            \HM_{\bullet}(Y,\s).
\end{equation}
Note again that all the \spinc{} structures in $\Sp(Y|F)$ have
non-torsion first Chern class.
When a local system $\Gamma_{\eta}$ is given, we
define $\HM_{\bullet}(Y|F;\Gamma_{\eta})$ similarly.
If $F$ is a surface with more than one component, we define
\[
            \Sp(Y|F) = \bigcap_{F_{i}\subset F} \Sp(Y|F_{i})
\]
where the $F_{i}$ are the components, and we define
$\HM_{\bullet}(Y|F)$ accordingly.

As a special case, we have

\begin{lemma}\label{lemma:product-rank-1}
    Let $F$ be a closed, connected, oriented surface of genus at least
    $2$, and let $Y=F\times S^{1}$. Regard $F$ as a surface
    $F\times\{p\}$ in $Y$.
    Then we have
    \[
                \HM_{\bullet}(Y|F) = \Z.
    \]
     Indeed, if $F$ is given a
    metric of constant negative curvature and $Y$ is given the product
    metric, then the complex that computes $\HM_{\bullet}(Y|F)$
    has a single generator, corresponding to a single, non-degenerate
    solution of the Seiberg-Witten equations. 
\end{lemma}

\begin{proof}
    This is standard. The \spinc{} that contributes is the product
    \spinc{} structure, which corresponds to the $2$-plane field
    tangent to the fibers of the map $Y\to S^{1}$. The unique
    gauge-equivalence class of solutions to the equations is a pair
    $[A,\Phi]$ with $\Phi$ covariantly constant.    
\end{proof}

\begin{corollary}\label{cor:product-rank-1-twist}
    Let $Y$ be the product $F\times S^{1}$, as in the previous lemma.
    Then for any local coefficient system $\Gamma_{\eta}$, we have
    \[
                \HM_{\bullet}(Y|F;\Gamma_{\eta} ) = \cR,
    \]
    where $\cR$ is the coefficient ring. \qed
\end{corollary}

\subsection{Disconnected 3-manifolds, part I}
\label{subsec:disconnected-1}

So far, following \cite{KM-book}, we have discussed connected
$3$-manifolds and connected $4$-dimensional cobordisms between them.
Because of the special role played by reducible connections, one must
be careful when generalizing; but there are simple situations where
the discussion can be carried over without difficulty to the case of
$3$-manifolds with several components. The analysis of the
Seiberg-Witten equations on a manifold with cylindrical ends is
carried out in \cite{KM-book} for an arbitrary number of ends, and our
task here is just to package the resulting information.

Let $W$ be cobordism from $Y_{\inn}$ to $Y_{\out}$, and suppose that each
of these has components
\[
\begin{aligned}
Y_{\inn} &= Y_{\inn,1} \cup \dots \cup Y_{\inn,r} \\
Y_{\out} &= Y_{\out,1} \cup \dots \cup Y_{\out,s}.
\end{aligned}
\]
Although we label them this way, no ordering of the components need be
chosen at this point.
We may allow either $r$ or $s$ (or both) to be zero, and we do not
require $W$ to be connected. If $W$ has any closed components, we
insist that each such component has $b_{+}\ge 2$.

As a simple way to avoid reducible connections, let us give a closed,
oriented surface $F_{\inn}\subset Y_{\inn}$ and $F_{\out}\subset Y_{\out}$. We
will suppose that each component of $Y_{\inn}$ contains a component of
$F_{\inn}$ and that all components of $F_{\inn}$ have genus $2$ or more.
Thus we have non-empty surfaces
\[
\begin{aligned}
F_{\inn,i} &= F_{\inn} \cap Y_{\inn,i}\\ & \subset Y_{\inn,i}.
\end{aligned}
\]
We
make a similar hypothesis for $F_{\out}$. We can regard the union
$F_{\inn}\cup F_{\out}$ as a subset of $W$, and we suppose that we are given
a surface $F_{W}\subset W$ which contains $F_{\inn}\cup F_{\out}$ in
addition perhaps to other components. The notation we previously used
for \spinc{} structures with constraints on $c_{1}$ can be extended
to this case: we write $\Sp(W|F_{W})$ for the set of \spinc{}
structures $\s$ on $W$ such that \eqref{eq:spinc-max} holds for every
component of $F_{W}$.

We can define $\HM_{\bullet}(Y_{\inn}|F_{\inn})$ by taking a  product
over the components of $Y_{\inn}$. That is, we should define the
configuration space
$\bonf(Y_{\inn})$ as the product of the $\bonf(Y_{\inn,i})$, and
we should construct $\HM_{\bullet}(Y_{\inn}|F_{\inn})$ as the Floer
homology of the Chern-Simons-Dirac functional on the components of
this product space which belong to the appropriate \spinc{}
structures. The only slight twist here is in understanding the
orientations of moduli spaces that are needed to fix the signs.

We therefore digress to consider orientations. For a cobordism such as
$W$ above, perhaps with several components, we define a $2$-element
set $\Lambda(W)$ of homology orientations of $W$ as follows. Attach
cylindrical ends to the incoming and outgoing ends to get a complete
manifold $W^{+}$, and let $t$ be
function which agrees with the cylindrical coordinate on the ends. The
function $t$ tends to $+\infty$ on the outgoing ends and $-\infty$ on
the incoming ends. Consider the linearized anti-self-duality operator
$\delta =d^{*}\oplus d^{+}$ acting on the weighted Sobolev spaces
\[
            \delta : L^{2}_{1,\epsilon} ( i \Lambda^{1} ) \to
            L^{2}_{\epsilon} (i\Lambda^{0} \oplus i\Lambda^{+}),                         
\]
where $L^{2}_{k,\epsilon}=e^{-\epsilon t}L^{2}_{k}$. Fix a \spinc{}
structure on $W$ and let $D^{+}_{A_{0}}$ be the Dirac operator,  for a
\spinc{} connection $A_{0}$ that is constant on the ends. We consider
$D_{A}$ acting on weighted Sobolev spaces of the same sort, and we
write
\[
     P  = \delta + D^{+}_{A_{0}}   .
\]
These are the linearized Seiberg-Witten equations on $W$, with Coulomb
gauge fixing, at a configuration where the spinor is zero.
There exists
$\epsilon_{0}>0$ such that the  operator $P$ is Fredholm for all
$\epsilon$ in the interval $(0,\epsilon_{0})$. We define $\Lambda(W)$
to be the set of orientations of the determinant line of $P$, for any
$\epsilon$ in this range. Using weighted Sobolev spaces here is
equivalent to using ordinary Sobolev spaces and replacing $P$ by a
zeroth-order perturbation which on the ends has the form
\[
                    P - \epsilon \Theta
\]
where $\Theta$ is obtained from applying the symbol of $P$ to the
vector field $\partial/ \partial t$ along the cylinder. The Dirac
operator is irrelevant at this point because it is complex and its
real determinant is therefore canonically oriented; so we could use
the operator $\delta$ instead.

Now suppose that $\alpha_{1}$ and
$\alpha_{2}$ are gauge-equivalence classes corresponding to
non-degenerate critical points of the
Chern-Simons-Dirac functional on $Y_{1}$ and $Y_{2}$ respectively. Let
$\gamma = (A,\Phi)$ be any configuration on $W^{+}$ that is asymptotic
to these gauge-equivalence classes on the ends. Let $P_{\gamma}$ be the
corresponding operator (acting on the Sobolev spaces without weights).
Define $\Lambda(W; \alpha_{1},\alpha_{2})$ to be the set of
orientations of the determinant of $P_{A}$. This is independent of the
choice of $\gamma$, in a canonical manner. 
If $\Lambda_{1}$ and $\Lambda_{2}$ are two $2$-element sets, we use
the notation $\Lambda_{1}\Lambda_{2}$ to denote the $2$-element set
formed by the obvious ``multiplication'' (the set of bijections from
$\Lambda_{1}$ to $\Lambda_{2}$). With this in mind, we define
\[
            \Lambda(\alpha_{1},\alpha_{2}) =
            \Lambda(W)\Lambda(W;\alpha_{1},\alpha_{2}).
\]
An excision argument makes this independent of $W$. Given now a
$3$-manifold $Y$ (with several components) and a non-degenerate
critical point $\alpha$, we choose cobordism $X$ from the empty set to
$Y$ and we define
\[
            \Lambda(\alpha) = \Lambda(\varnothing,\alpha).
\]
We then have
\[
                \Lambda(\alpha_{1},\alpha_{2}) =
                \Lambda(\alpha_{1})\Lambda(\alpha_{2}).
\]

What this last equality means in practice is this. If we are
given a cobordism
$W$ with a choice of homology orientation in $\Lambda(W)$ and a moduli
space $M=M(W;\alpha_{1},\alpha_{2})$, then a choice of orientation of $M$ is
the same as a choice of bijection from $\Lambda(\alpha_{1})$ to
$\Lambda(\alpha_{2})$. In the case that $Y_{1}=Y_{2}$, the cylindrical
cobordism has a canonical homology-orientation because the operator
$P$ is invertible; so in this $\Lambda(\alpha_{1})\Lambda(\alpha_{2})$
orients the moduli spaces. The appropriate definition of
$\HM_{\bullet}(Y_{1},\s)$ for \spinc{} structures $\s$ that are
non-torsion on each component is therefore to take the complex to be
\[
        C_{\bullet}(Y_{1},\s) =   \bigoplus_{\alpha_{1}} \Z\Lambda(\alpha_{1})
\]
and to define the differential using the corresponding orientation of
the moduli spaces.  In this way, we construct
$\HM_{\bullet}(Y_{1}|F_{1})$ and $\HM_{\bullet}(Y_{2}|F_{2})$. If we
supply $W$ with a homology orientation in the above sense, then
$W$ defines a map
\begin{equation}\label{eq:W-map}
                \HM(W|F_{W}) : \HM_{\bullet}(Y_{\inn}|F_{\inn}) \to
                \HM_{\bullet}(Y_{\out}|F_{\out}).
\end{equation}
The notation $\HM(W|F_{W})$ is meant to imply that we use only the
\spinc{} structures from $\Sp(W|F_{W})$. 

The complex $ C_{\bullet}(Y_{1},\s)$ just defined can be considered as
a tensor product over the connected components of $Y_{1}$:
\begin{equation}\label{eq:tensor-complex}
           C_{\bullet}(Y_{1},\s) = \bigotimes_{i}
            C_{\bullet}(Y_{1,i},\s_{i}) ,
\end{equation}
but there are some choices involved. Let us pick an ordering of the
components. For each $i$, let $X_{i}$ be a cobordism from the empty
set to $Y_{1,i}$. Using the standard convention for the orientation of
a direct sum, we can then identify
\[
                    \Lambda(X) =
                    \Lambda(X_{1})\Lambda(X_{2})\cdots\Lambda(X_{r}),
\]
and similarly with $\Lambda(X,\alpha_{1})$. In this way, we can
specify an isomorphism
\[
                \Z\Lambda(\alpha_{1}) \to
                \Z\Lambda(\alpha_{1,1})\otimes \dots \otimes
                \Z\Lambda(\alpha_{1,r}).
\]
This allows us to identify the complexes on the left and right in
\eqref{eq:tensor-complex} as groups. Ordering issues mean
that there will be the expected alternating signs appearing when we compare the
differentials on the left and right.
As usual with products in homology, what results from this is a split short exact
sequence
\begin{equation}\label{eq:split-short-exact}
    0\to    \bigotimes_{i} \HM_{\bullet}(Y_{\inn,i}|F_{\inn,i}) \to
        \HM_{\bullet}(Y_{\inn}|F_{\inn}) \to T               \to 0
\end{equation}
where $T$ is a torsion group. If $W$ is closed and has more than one component,
the invariant is a product of the contributions from each component.

There is another sign issue to discuss. Consider the case of a
$3$-manifold $Y$ with  non-torsion \spinc{} structure $\s$. Let $Z$ be
the $4$-manifold $S^{1}\times Y$. We can pull back the \spinc{}
structure to $Z$, and we still call it $\s$. For clarity, suppose
that
$b_{+}(Z)$ is bigger than $1$, so that $\m(Z,\s)$ is defined. To fix
the sign of $\m(Z,\s)$, we need a homology orientation of $Z$; but  a
product such as $Z$ has a preferred  homology orientation. To define
it, we must specify an orientation for the determinant of $P$ on $Z$.
The operator $P-\epsilon\Theta$ is invertible for small $\epsilon$,
and we use this to to orient the determinant.
Now let $\alpha$ be  non-degenerate critical point for the (possibly
perturbed) Chern-Simons-Dirac functional on $(Y,\s)$. This pulls back
to an isolated, non-degenerate solution on $Z$ to the $4$-dimensional
Seiberg-Witten equations, say $\hat{\alpha}$. This solution
contributes either $+1$ or $-1$ to the invariant $\m(Z,\s)$. We have
the following lemma.

\begin{lemma}
    The solution $\hat{\alpha}$ contributes $+1$ or $-1$ to the
    invariant $\m(Z,\s)$ according as the critical point $\alpha$ has
    odd or even grading in $C_{\bullet}(Y,\s)$, for the canonical
    $Z/2$ grading.
\end{lemma}

\begin{proof}
    We have two operators differing by zeroth-order terms
    \[
                \begin{aligned}
                    P_{0} &= P-\epsilon\Theta \\
                    P_{1} &= P_{\gamma}.
                \end{aligned}
    \]
    Let $P_{s}$ be a homotopy between them. We have a determinant line
    for this family of operators over the interval $[0,1]$, and the
    invertibility of $P_{0}$ and $P_{1}$ at the two ends gives the
    determinant line a canonical orientation at the two ends. The sign
    with which $\hat{\alpha}$ contributes is, by definition, $+1$ or
    $-1$ according as these two orientations at $s=0,1$ are homotopic.

    On the other hand, we can write $P_{s}$ as
    \[
                \frac{d}{dt} + L_{s}
    \]
    on $S^{1}\times Y$, where $L_{s}$ is a self-adjoint elliptic
    operator perturbed by a bounded term, and the canonical mod 2
    grading of $\alpha$ is determined, by definition, by the parity of
    the spectral flow of the family of operators $L_{s}$ from $s=0$ to
    $s=1$.

    So we must see that the parity of the spectral flow of the
    operators $L_{s}$ determines whether the invertible operators
    $P_{0}$ and $P_{1}$ provide the same orientation. This is a
    general fact about families of self-adjoint Fredholm operators.
    What we have here are two non-trivial homomorphisms
    \[
                H_{1}(S) \to \Z/2,
    \]
    where $S$ is a suitable space of self-adjoint operators. One can
    argue as in \cite{KM-book}, following \cite{Atiyah-Singer-skew},
    that one may take $S$ to have the
    homotopy of $U(\infty)/O(\infty)$,
    at which point it is clear that these two
    are the same.
\end{proof}

A consequence of the lemma is that the invariant $\m(Z,\s)$ is the
equal to the Euler characteristic of $\HM_{\bullet}(Y,\s)$, computed
using the canonical mod 2 grading.
From the lemma and excision, we obtain similar results in other
situations of the following sort. Consider again a cobordism $W$ from
$Y_{1}$ to $Y_{2}$ with surfaces $F_{W}$, $F_{1}$ and $F_{2}$ as
before.
Suppose that one of the incoming boundary components is the same
as one of the outgoing ones: say
\[
            Y_{\inn,r} = Y_{\out,s}.
\]
We may form a new $W^{*}$ from $W$ by identifying these boundary
components, so $W^{*}$ has $r-1$ incoming and $s-1$ outgoing boundary
components. The manifolds $Y_{\inn,r}$ and $Y_{\out,s}$ may belong either to
the same or to different components of $W$, but we treat these cases
together. The surface $F_{W}$ gives rise to a homeomorphic surface
$F_{W^{*}}$ in $W^{*}$. (We push $F_{\inn,r}$ and $F_{\out,s}$ away from the
boundary of $W$ before gluing to $Y_{\inn,r}$ to $Y_{\out,r}$, to keep these
surfaces disjoint, if necessary.) If is possible that this process has
created a $W^{*}$ which has one more closed component than $W$. This new
closed component of $W^{*}$ will have $b_{+}$ at least $1$; but we shall
suppose  that, if there is such a component, it has
$b_{+}$ at least $2$. (The case of $b_{+}=1$ will be discussed in a
slightly different context in the next subsection.)

Under this hypothesis on $b_{+}$ for the closed components, we now have a new map
\begin{equation}\label{eq:W-primes-map}
                \HM(W^{*}|F_{W^{*}}) : \HM_{\bullet}(Y^{*}_{\inn}|F^{*}_{\inn}) \to
                \HM_{\bullet}(Y^{*}_{\out}|F^{*}_{\out}),
\end{equation}
where $Y^{*}_{\inn}$ is $Y_{\inn}\setminus Y_{\inn,r}$ and  $Y^{*}_{\out}$ is defined
similarly. The analysis from \cite{KM-book}  provides a
``gluing theorem'' which tells us that the map $\HM(W^{*}|F_{W^{*}})$ is obtained
from $\HM(W|F_{W})$ by a contraction. More precisely, at the chain
level, $(W,F_{W})$ defines a chain map
\[
                    \bigotimes_{i} C_{\bullet}(Y_{\inn,i}|F_{\inn,i})
                    \to \bigotimes_{j} C_{\bullet}(Y_{\out,j}|F_{\out,j}).
\]
This map can be contracted by taking an alternating trace over
\[
C_{\bullet}(Y_{\inn,r}|F_{\inn,r})=C_{\bullet}(Y_{\out,s}|F_{\out,s}),\] and the
result of this contraction is a chain map which is chain-homotopic
to the chain map defined by $(W^{*},F_{W^{*}})$. 

The cobordism $W$ from $Y_{\inn}$ to $Y_{\out}$ can also be regarded as a
cobordism $\tilde W$ from $\tilde Y_{\inn}$ to $\tilde Y_{\out}$, where
\[
            \tilde Y_{\inn} = Y_{\inn} \cup (-Y_{\out,s})
\]
and
\[
                \tilde Y_{\out} = Y_{\out}\setminus Y_{\out,s}.
\]
(That is, we regard the last outgoing component as an incoming component
with the opposite orientation.) The relation between the maps defined
by $W$ and $\tilde W$ can be put in the same context  as the above
gluing theorem. We first add an extra component $Z$ to $W$, where $Z$
is the cylinder $[0,1]\times Y_{\out,s}$, regarded as a cobordism from
$Y_{\out,s}\cup(-Y_{\out,s})$ to the empty set. The map defined by $W\cup Z$
is a tensor product, at the chain level, and the cobordism $\tilde W$
can be obtained by gluing an outgoing component of $W$ to an incoming
component of $Z$. All that is left is to understand the map defined by
$Z$. Discounting torsion, this last map is the Poincar\'e duality
pairing
\[
                    \HM_{\bullet}(-Y_{\out,s} | F_{\out,s})
                    \otimes \HM_{\bullet}(Y_{\out,s} | F_{\out,s}) \to \Z.
\]
As in \cite{KM-book}, this pairing depends on a homology orientation
of $Y_{\out,s}$, which reappears as the need to choose a homology
orientation for the extra component $Z$.

Let us pursue a simple application of this formalism. Let $W$ be 
again a cobordism from $Y_{\inn}$ to $Y_{\out}$ and let $F_{\inn}$ and
$F_{\out}$ be surfaces in these boundary $3$-manifolds as above.
Suppose that $W$ contains in its interior a product $3$-manifold
\[
                    Z = G \times S^{1}
\]
where $G$ is connected of genus at least $2$. Regard $G = G\times
\{p\}$ also as a
submanifold of $W$. Form a new cobordism $W^{\dag}$ from $Y_{\inn}$ to
$Y_{\out}$ by the following process. Cut $W$ open along $Z$ to obtain
a manifold $W'$ with two extra boundary components $G\times S^{1}$,
then attach a copy of $G\times D^{2}$ to each of these boundary
components to obtain $W^{\dag}$. Set
\[
                \begin{aligned}
                    F_{W} &= (F_{\inn}\cup F_{\out} \cup G) \subset W
                    \\
                     F_{W^\dag} &= (F_{\inn}\cup F_{\out} \cup G)
                     \subset W^{\dag}.
                \end{aligned}
\]
Then we have

\begin{proposition}\label{prop:excision-prototype}
    The maps $\HM(W|F_{W})$ and $\HM(W^{\dag}|F_{W^{\dag}})$ are equal,
    up to sign, as maps
    \[
                    \HM(Y_{\inn}| F_{\inn}) \to
                    \HM(Y_{\out}|F_{\out}).
    \]
\end{proposition}

\begin{proof}
    Consider the manifold $W'$ obtained from $W$ by cutting open along
    $Z$. This is a cobordism from $Y_{\inn}\cup Z$ to $Y_{\out} \cup
    Z$. The manifold $W$ or $W^{\dag}$ can be obtained from $W'$ by
    gluing with  $[0,1]\times Z$ or with $(D^{2}\amalg D^{2}) \times
    G$ respectively. We can regard $[0,1]\times Z$ and $(D^{2}\amalg
    D^{2})\times G$ as two different cobordisms from $Z$ to $Z$, and
    they both induce maps
    \[
                        \HM_{\bullet}(Z|G) \to \HM_{\bullet}(Z|G).
    \]
    The result follows from the glueing formalism as long as we know
    that these two maps on $\HM_{\bullet}(Z|G)$ are the same.
    Lemma~\ref{lemma:product-rank-1} tells us that
    $\HM_{\bullet}(Z|G)$ is simply $\Z$. The product $[0,1]\times Z$
    of course induces the identity map on this copy of $Z$. So it only
    remains to show that the invariant of manifold $D^{2}\times G$ in
    $\HM_{\bullet}(Z|G)$ is $\pm 1$. This can be seen directly by
    examining the solutions of the Seiberg-Witten equations; or one
    can see indirectly that this must be so, on the grounds that there
    exist closed $4$-manifolds containing $(Z|G)$ for which an
    appropriate
    Seiberg-Witten invariant is $1$. 
\end{proof}

\subsection{Disconnected 3-manifolds, part II}
\label{subsec:disconnected-2}

In the previous subsection we discussed gluing results in a context
where the boundary components of the cobordisms carried \spinc{}
structures that had non-torsion first Chern classes. The non-torsion
condition ensures that reducible solutions on the $3$-manifolds play
no role. A situation that is algebraically similar is when the
 boundary components $Y$ carry $1$-cycles $\eta$ and we
use local coefficients for which the vanishing theorem
Proposition~\ref{prop:twist-vanish}
applies. We can think of $\Hred_{\bullet}(Y;\Gamma_{\eta})$ as
measuring the contribution of the reducible solutions; so in a
situation where this group is zero, as in the Proposition, we can
expect simple gluing results.
This expectation is confirmed in the case of connected
$3$-manifolds by the results of \cite[section 32]{KM-book}. We will
deal here with the simplest situation, in which the boundary
components are $3$-tori and local coefficients are used.

Let $W$ be a compact oriented $4$-manifold with boundary, and
suppose the oriented boundary consists of a collection of $3$-tori,
\[
            \partial W = T_{1} \cup \dots \cup T_{r}.
\]
We do not need to suppose that $W$ is connected, but we do require
that every closed component of $W$ has $b_{+}$ at least $2$.
Let $\nu\subset W$ be a $2$-chain with
\[
            \partial \nu = \eta_{1} + \dots + \eta_{r}.
\]
We suppose that each $\eta_{i}$ is a $1$-cycle in $T_{i}$ satisfying
the hypotheses of Proposition~\ref{prop:twist-vanish} and that our
coefficient ring  $\cR$ has $t-t^{-1}$ invertible. We may take it that
each $\eta_{i}$ is a standard circle. For each $i$, the map
\[
    j:\Hto_{\bullet}(T_{i};\Gamma_{\eta_{i}}) \to
    \Hfrom_{\bullet}(T_{i};\Gamma_{\eta_{i}})
\]    
is an isomorphism according to the
proposition, so we again just write \[ HM_{\bullet}(T_{i};\Gamma_{i})
\]
for this group, using $j$ to identify the two. According to
\cite[section 37]{KM-book}, this group is a free $\cR$-module of rank
$1$,
\[
            \HM_{\bullet}(T_{i};\Gamma_{\eta_{i}}) \cong \cR.
\]
(The proof in \cite{KM-book} was done in the case that $\cR=\R$, but
only the invertibility of $t-t^{-1}$ is needed.) After choosing a
basis element in $\HM_{\bullet}(T_{i};\Gamma_{\eta_{i}})$, we should
expect $W$ to have an invariant living in
\[
\bigotimes_{i}
\HM_{\bullet}(T_{i};\Gamma_{\eta_{i}}) 
                    = \cR.
\]
However, there is a short-cut to defining an $\cR$-valued invariant of
$W$, used in \cite{Fintushel-Stern-knot} and \cite[section
38]{KM-book}. We now describe this short-cut. In the remainder of this
subsection, we will leave aside the question of choosing
homology-orientations to fix the  sign of the invariants that
arise. So a $4$-manifold or a cobordism will have an invariant that is
ambiguous in its overall sign.

Let $E(1)$ be a rational elliptic surface and let $\widehat{E(1)}$ be
the complement of the neighborhood of a regular fiber, so that
$\partial\widehat{E(1)}=T^{3}$. Let $\nu_{1}$ be a $2$-cycle in $E(1)$
arising from a section meeting the neighborhood of the fiber
transversely in a disk, and let $\hat\nu_{1}$ be the corresponding
$2$-chain in $\widehat{E(1)}$. Let $\bar{W}$ be the closed
$4$-manifold obtained by attaching $r$ copies of $\widehat{E(1)}$ to
$W$, making the attachments in such a way that the $1$-cycles in the
boundary tori match up: thus the manifold
\[
            \bar W = W \cup_{T_{1}} \widehat{E(1)} \dots
            \cup_{T_{r}}\widehat{E(1)}
\]
contains a $2$-cycle
\[
             \bar\nu = \nu \cup_{\eta_{1}} \hat\nu_{1} \dots
             \cup_{\eta_{r}} \hat\nu_{1}.
\]
We can now compute a Seiberg-Witten invariant of the closed pair
$(\bar{W},\bar\nu)$, and the result depends only on $(W,\nu)$, not on
the choice of gluing. Thus we may make a definition:

\begin{definition}\label{def:W-def-close}
Let $W$ have boundary a collection of $3$-tori, as above, let $\nu$ be
a $2$-chain in $W$, and let $(\bar W, \bar \nu)$ be the closed
manifold obtained by attaching copies of $\widehat{E(1)}$. Suppose
that every component of $\bar W$ has $b_{+}\ge 2$. Then we
write
\[
   \m(W,\nu)\in \cR                 
\]
for the invariant
$
            \m(\bar{W},\bar{\nu})$
of the closed manifold, as defined at \eqref{eq:SW-generating}.             
    \CloseDef
\end{definition}

There is a formal device that can be used to extend this definition to
include the case that $\bar{W}$ has closed components with $b_{+}=1$.
 Let $E(n)$ denote the 
elliptic surface without multiple fibers and having Euler
number $n$, and let $\widehat{E(n)}$ be the complement of a
fiber. There is a $2$-chain $\nu_{n}$ just as in the case $n=1$.
Instead of attaching $\widehat{E(1)}$ to each $T_{i}$ to form $\bar{W}$, we can similarly
attach $\widehat{E(n_{i})}$ to $T_{i}$, for any $n_{i}\ge 1$. We still
refer to the resulting closed manifold as $\bar{W}$. It contains a
$2$-cycle $\bar{\nu}$ as before. By choosing $n_{i}$ larger than $1$
when needed, we can ensure that all components of $\bar{W}$ have
$b_{+}$ least $2$. We then define $\m(W,\nu)$ by
\begin{equation}\label{eq:n-i-correction}
            \m(W,\nu) = (t-t^{-1})^{-\sum (n_i -1)}
            \m(\bar{W},\bar{\nu}).
\end{equation}
By the results of \cite[section 38]{KM-book}, the quantity on the
right is independent of the choice of the $n_{i}$.

Suppose next that $W$ contains in its interior another $3$-torus $T$
which intersects $\nu$ transversely in a single circle $\eta$
representing a primitive element of $H_{1}(T)$. We can then cut $W$
open along $T$ to obtain $W'$, a manifold whose boundary consists of
$(r+2)$ tori. We can denote the two new boundary components by $T_{r+1}$
and $T_{r+2}$. By cutting $\nu$ also, we obtain a $2$-chain $\nu'$ in
$W'$ whose boundary has two new circles $\eta_{r+1}$ and $\eta_{r+2}$
in the new  boundary components. We have the following glueing
theorem. (The hypothesis that $t-t^{-1}$ is invertible in $\cR$ remains in
place.)

\begin{proposition}\label{prop:cut-W-1}
    In the above situation, the invariants of $(W,\nu)$ and
    $(W',\nu')$ are equal: thus
    \[
                    \m(W,\nu) = \m(W',\nu')
    \]
    in the ring $\cR$.
\end{proposition}

\begin{proof}
    There are two cases, according as $T$ is separating or not. The
    separating case is treated in \cite[section 38]{KM-book}. We deal
    here with the non-separating case.
    The definitions mean that both sides are to be interpreted as
    invariants of suitable closed manifolds. Restating it in such
    terms, and throwing out the components that do not contain $T$,
    we arrive at the following. Let $X$ be a closed, connected
    $4$-manifold with $b_{+}\ge 2$, and let $T\subset X$ be a
    non-separating
    $3$-torus. Let $\nu$ be a $2$-cycle in $X$ meeting $T$
    transversely in a standard circle $\eta$ with multiplicity $1$.
    Let $X'$ be cobordism from $T$ to $T$ obtained by
    cutting $X$ open, and let $\nu'$ be the resulting $2$-chain in
    $X'$. Because of what we already know about the separating case,
    the proposition is equivalent to the following lemma, which we
    shall prove.
\end{proof}

\begin{lemma}
    In the above situation, the map induced by the cobordism,
    \[
                        \Hfrom_{\bullet}(X';\Gamma_{\nu'}) :
                        \Hfrom_{\bullet}(T;\Gamma_{\eta}) \to
                        \Hfrom_{\bullet}(T;\Gamma_{\eta})
    \]
    is given by multiplication by the element
    $\m(X, \nu)\in \cR$.
\end{lemma}

\begin{proof}
    It is convenient to arrange first that $X'$ has $b_{+}$ at least
    $1$. We can do this by choosing a standard $2$-torus $F$ near $T$
    intersecting $\nu$ transversely and forming a fiber sum at $F$
    with an elliptic surface $E(n)$. From what we know about
    separating $3$-tori, we can conclude that this modification
    multiplies
    both $ \Hfrom_{\bullet}(X';\Gamma_{\nu'})$ and $\m(X,\nu)$  by
    $(t-t^{-1})^{n-1}$.
    
    We now perturb the Chern-Simons-Dirac functional on $T$, as in
    \cite[section 37]{KM-book}, so that there are only reducible
    critical points, and we stretch $X$ at $T$, inserting a cylinder
    $[-R,R]\times T$ and letting $R$ increase to infinity as usual.
    We consider what happens to the zero-dimensional moduli spaces on
    $X$ in the limit.
    Because $b_{+}(X')$ is at least $1$, we obtain in the limit only
    irreducible solutions on the cylindrical-end manifold obtained
    from $X'$. Furthermore, these irreducible solutions run from
    boundary-unstable critical points at the incoming end to
    boundary-stable critical points at the outgoing end. The weighted
    count of such solutions defines the map
    \[
                \Hft(X';\Gamma_{\nu'}) :
                \Hfrom_{\bullet}(T;\Gamma_{\eta})  \to
                \Hto_{\bullet}(T;\Gamma_{\eta}) 
    \]
    in the notation of \cite[subsection 3.5]{KM-book}.
    We must also obtain in the limit  some (possibly broken) trajectories on
    the cylindrical part, running from  boundary-stable critical
    points to boundary-unstable critical points. For
    dimension-counting reasons, these trajectories must actually be
    unbroken and must be boundary-obstructed. The weighted count of
    such trajectories defines the map 
    \[
                        j : \Hto_{\bullet}(T;\Gamma_{\eta})  \to
                \Hfrom_{\bullet}(T;\Gamma_{\eta}).
    \]
    Thus $\m(X,\nu)$ is equal to the contraction by the Kronecker
    pairing of two chain maps which on homology define the composite
    \[
                        j \circ \Hft_{\bullet}(X; \Gamma_{\nu'})
                        : \Hfrom_{\bullet}(T;\Gamma_{\eta})  \to
                \Hfrom_{\bullet}(T;\Gamma_{\eta}) .
    \]
    It follows that $\m(X,\nu)$ is the trace of this composite map.
    The composite is equal to $\Hfrom_{\bullet}(X';\Gamma_{\nu'})$,
    and the Floer group here is a free $\cR$-module of rank $1$, so
    the result follows.
\end{proof}

There is a straightforward modification of the above results in the
case that $W$ has some additional boundary components which are not
$3$-tori but contain surfaces $F$ of genus $2$ or more, as in the
previous subsection. That is, we suppose that the boundary of $W$ is a
union of $3$-tori $T_{1},\dots,T_{r}$ together with a pair of
$3$-manifolds $-Y_{\inn}$ and $Y_{\out}$, each of which may have several
components. We suppose also that $Y_{\inn}$ and $Y_{\out}$ contain surfaces
$F_{\inn}$ and $F_{\out}$ all of whose components have genus $2$ or more. We
also ask that each component of $Y_{i}$ contains a
component of $F_{i}$. We shall suppose 
that there is a $2$-chain $\nu$ in $W$ whose boundary we write as
\[
            \partial \nu = - \zeta_{1} + \zeta_{2} + \eta_{1} + \dots
            + \eta_{r}.
\]
The $\eta_{i}$ are to be standard circles, one in each torus $T_{i}$
as before. The $1$-cycles $\zeta_{1}$ and $\zeta_{2}$ will be in
$Y_{1}$ and $Y_{2}$, but we can allow these to be arbitrary (zero for
example). We take $F_{W}$ to be any closed surface
in $W$ consisting of $F_{\inn}\cup F_{\out}$ together perhaps with
additional components. We again suppose that any closed component
of $W$ has $b_{+}\ge 2$. Then $W$ should give rise to a map
\begin{equation}\label{eq:W-with-Y}
                \HM_{\bullet}(W | F_{W };\Gamma_{\nu})
                 : \HM_{\bullet}(Y_{\inn}|F_{\inn}; \Gamma_{\zeta_{1}})
                 \to \HM_{\bullet}(Y_{\out}|F_{\out};\Gamma_{\zeta_{2}}).
\end{equation}
To define this map, we can again attach $(\widehat{E(1)},
\hat\nu_{1})$ to each of the $3$-tori, to obtain $(\bar{W}, \bar\nu)$
a cobordism from $Y_{\inn}$ to $Y_{\out}$ containing a $2$-chain $\bar\nu$
and a surface $F_{W}$. The boundary of $\bar\nu$ is just
$-\zeta_{1}+\zeta_{2}$.
As in Definition~\ref{def:W-def-close}, we take
$\HM_{\bullet}(W:F_{W};\Gamma_{\nu})$ to be \emph{defined} by the map
given by the cobordism $\bar{W}$. In the event that $\bar{W}$ has any
closed components with $b_{+}=1$, we modify the construction by using
elliptic surfaces $E(n_{i})$ as in \eqref{eq:n-i-correction}.
Proposition~\ref{prop:cut-W-1} then has the following variant.

\begin{proposition}\label{prop:cut-W-2}
    Let $W$ be as above, and let $T\subset W$ be a $3$-torus meeting
    $\nu$ transversely in a standard circle with multiplicity $1$. Let
    $W'$ and $\nu'$ be obtained from $W$ and $\nu$ by cutting along
    $T$. Suppose that $F_{W}$ is disjoint from $T$, so that it becomes
    also a surface $F_{W'}$ in $W'$. Assume as always that $t-t^{-1}$
    is invertible in $\cR$. Then the maps
\begin{equation}\label{eq:W-with-Y-2}
\begin{aligned}
\HM_{\bullet}(W | F_{W };\Gamma_{\nu})
                 &: \HM_{\bullet}(Y_{\inn}|F_{\inn}; \Gamma_{\zeta_{1}})
                 \to \HM_{\bullet}(Y_{\out}|F_{\out};\Gamma_{\zeta_{2}}) \\
\HM_{\bullet}(W' | F_{W' };\Gamma_{\nu'})
                 &: \HM_{\bullet}(Y_{\inn}|F_{\inn}; \Gamma_{\zeta_{1}})
                 \to \HM_{\bullet}(Y_{\out}|F_{\out};\Gamma_{\zeta_{2}})                 
                 \end{aligned}
\end{equation}
are equal up to sign.\qed
\end{proposition}

A particular application of this setup will be used in the sequel,
a version of Proposition~\ref{prop:excision-prototype}. We formulate
the result as the following corollary:

\begin{corollary}\label{cor:excision-prototype-twist}
    Let $W$ be a cobordism from $Y_{1}$ to $Y_{2}$ containing a
    $2$-chain $\eta$ with boundary $-\zeta_{1} \cup \zeta_{2}$. Let
    $F_{1}$, $F_{2}$ and $F_{W}$ be surfaces as above. Let $T\subset
    W$ be a $3$-torus disjoint from $F_{W}$ and cutting $\nu$ in a
    standard circle $\eta\subset T$. Form $W^{\dag}$ by cutting $W$ along
    $T$ and attaching two copies of $D^{2}\times T^{2}$ in such a way
    that $\partial D^{2}\times\{p\}$ is glued to $\eta$ in both copies. Let
    $\eta^{\dag}$ be the $2$-chain in $W^{\dag}$ obtained by attaching
    $2$-disks $D^{2}\times\{p\}$. Then, as maps from
    $\HM_{\bullet}(Y_{\inn}|F_{\inn}; \Gamma_{\zeta_{1}})$ to
    $\HM_{\bullet}(Y_{\out}|F_{\out}; \Gamma_{\zeta_{2}})$, we have
    \[
                \HM(W|F_{W};\Gamma_{\nu}) = (t-t^{-1})
                \HM(W^{\dag}|F_{W^{\dag}}; \Gamma_{\nu^{\dag}}),
    \]
    to within an overall sign.
\end{corollary}

\begin{proof}
    Using Proposition~\ref{prop:cut-W-2}, this can be proved with the
    same strategy that we applied to
    Proposition~\ref{prop:excision-prototype}.  That is, we consider
    two different cobordisms from $T$ to $T$: first, the product
    cobordism, and second the (disconnected) cobordism formed from two
    copies of $D^{2}\times T^{2}$. In each case, there is an obvious
    $2$-chain whose boundary is the difference of the two copies of
    $\eta$. Each of these cobordisms has an invariant which lives in
    $\cR$, according to Definition~\ref{def:W-def-close}, or more
    accurately its correction at \eqref{eq:n-i-correction}.
    In this sense, the product cobordism has invariant $1\in \cR$.
    The invariant of the other cobordism is
    $(t-t^{-1})^{-1}$, as can be deduced from the invariants of the
    elliptic surfaces.
\end{proof}

\section{Floer's excision theorem}

\subsection{The setup}

We shall need to understand
how monopole Floer homology behaves under certain cutting and gluing
operations on the underlying $3$-manifold. A formula of the type
that we need was first proved by Floer in the context of instanton
homology. Floer's ``excision formula'', as he called it, applied only
to cutting along tori; but in the monopole homology context one can
equally well cut along surfaces of higher genus, as long as one
restricts to \spinc{} structures that are of top degree on the surface
where the cut is made. We give the proof in the monopole Floer
homology context in this section: it is almost identical to Floer's
argument, as presented in \cite{Braam-Donaldson}. Similar formulae have
been proved in Heegaard Floer theory, by Ghiggini, Ni and Juh\'asz
\cite{Ghiggini, Ni-A, Ni-B, Juhasz-1, Juhasz-2}.

The setup is the following. Let $Y$ be a closed, oriented
$3$-manifold, of either one or two components. In the case of two
components, we call the components $Y_{1}$ and $Y_{2}$.  Let
$\Sigma_{1}$ and $\Sigma_{2}$ be closed oriented surfaces in $Y$, both
of them connected and of equal genus.  If $Y$ has two components, then
we suppose that $\Sigma_{i}$ is a non-separating surface in $Y_{i}$
for $i=1,2$. If $Y$ is connected, then we suppose that $\Sigma_{1}$
and $\Sigma_{2}$ represent independent homology classes. In either
case, we write $\Sigma$ for $\Sigma_{1}\cup \Sigma_{2}$.
Fix an
orientation-preserving diffeomorphism $h : \Sigma_{1} \to \Sigma_{2}$.
From this data, we construct a new manifold $\tilde Y$ as follows. Cut each
$Y$ along $\Sigma$ to obtain a manifold $Y'$ with four
boundary components: with orientations, we can write
\[
            \partial Y'  = \Sigma_{1} \cup (-\Sigma_{1}) \cup
            \Sigma_{2} \cup (-\Sigma_{2}) 
\]
If $Y$ has two components, then so does $Y'$, and we can write
$Y'=Y'_{1}\cup Y'_{2}$.
Now form $\tilde Y$ by gluing the boundary component
$\Sigma_{1}$ to the boundary component $-\Sigma_{2}$  and gluing
$\Sigma_{2}$ to $-\Sigma_{1}$, using the chosen diffeomorphism of $h$
both times. See Figure~\ref{fig:Excision-setup} for a picture in the
case that $Y$ has two components. In either case, $\tilde Y$ is
connected. We write $\tilde \Sigma_{1}$ for
the image of $\Sigma_{1} = -\Sigma_{2}$ in $\tilde Y$ and $\tilde
\Sigma_{2}$ for the image of $\Sigma_{2}=-\Sigma_{1}$. So $\tilde Y$
contains a surface $\tilde\Sigma= \tilde\Sigma_{1}\cup
\tilde\Sigma_{2}$.

\begin{figure}
    \begin{center}
        \includegraphics[scale=0.7]{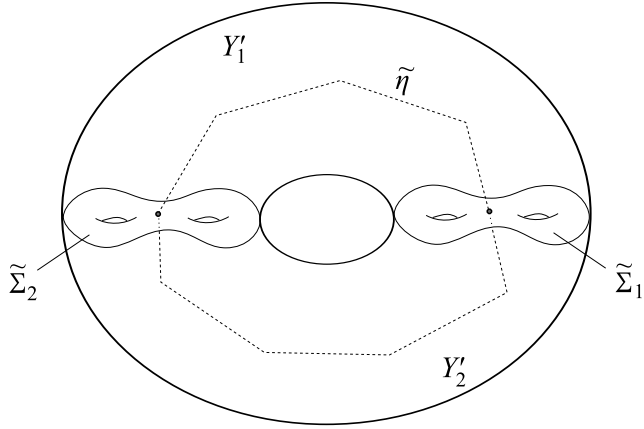}
    \end{center}
    \caption{\label{fig:Excision-setup}
    Forming a manifold $\tilde Y$ from $Y_{1}$ and $Y_{2}$, for the
    excision theorem.}
\end{figure}

If we wish  to use local coefficients in Floer homology, we will
need to augment this excision picture with $1$-cycles $\eta$.
Specifically, we  take a $1$-cycle $\eta$ in $Y$ that intersects each
$\Sigma_{i}$  transversely in a single point
$p_{i}$ ($i=1,2$) with
positive orientation. If $Y$ has two components, then we may write
$\eta=\eta_{1}+\eta_{2}$ for its two parts.  We suppose that the
diffeomorphism $h$ is chosen so that $h(p_{1})=p_{2}$. When this is
done, the $1$-cycle gives to a $1$-cycle
$\tilde \eta$ in
the new manifold $\tilde Y$, as shown, by cutting and gluing.

We begin with a statement of the excision theorem with integer
coefficients, when the genus of $\Sigma$ is two or more.

\begin{theorem}\label{thm:Floer-excision}
If $\tilde Y$ is obtained from $Y$ as above and the genus of
$\Sigma_{1}$ and $\Sigma_{2}$ is at least two, then there is an isomorphism of Floer groups
with integer coefficients,
\[            
            \HM_{\bullet}(Y|\Sigma)
         \to   \HM_{\bullet}(\tilde Y|\tilde \Sigma).
\]
\end{theorem}

\begin{remark}
    In the case that $Y$ has two components, the left-hand side is 
    the homology of a tensor product of complexes. In this case, the
    statement of the theorem implies that there is a split short exact
    sequence
\begin{multline}\label{eq:tensor-product-1}
       \HM_{\bullet}(Y_{1}|\Sigma_{1})
            \otimes
            \HM_{\bullet}(Y_{2}|\Sigma_{2})
         \to   \\
         \HM_{\bullet}(\tilde Y|\tilde \Sigma) \to \mathrm{Tor}\bigl(  \HM_{\bullet}(Y_{1}|\Sigma_{1})
            ,
            \HM_{\bullet}(Y_{2}|\Sigma_{2})\bigr).
\end{multline}
\end{remark}

Floer's version of this theorem has $\Sigma_{1}$  and $\Sigma_{2}$ of
genus $1$, with $Y=Y_{1}\cup Y_{2}$.  It uses
instanton Floer homology associated to an $\SO(3)$ bundle with
non-zero Stiefel-Whitney class on $\Sigma$. To obtain a version in
monopole Floer homology when $\Sigma$ has genus $1$, we need to use
local coefficients. We present a version that is tailored to our
later needs. We recall that $\Gamma_{\eta}$ denotes a system of local
coefficients with fiber $\cR$, a commutative ring as in
section~\ref{subsec:monopole-recap}. We suppose, as just discussed,
that $\eta$ meets $\Sigma_{1}$ and $\Sigma_{2}$ each in a single point
so that we may form $\tilde\eta$ as shown. Under these hypotheses, we
expect there to be an isomorphism
\[
            \HM_{\bullet}(Y;\Gamma_{\eta})
                    \to    \HM_{\bullet}(\tilde Y;\Gamma_{\tilde \eta}) .
\]
We shall not endeavor to prove this variant of Floer's excision
theorem here, because it involves considering reducible solutions on
multiple boundary components. Instead, as in
section~\ref{subsec:disconnected-2}, we introduce some auxiliary
surfaces $F$ and corresponding constraints on the \spinc{} structures,
just to avoid reducibles.

Thus we suppose in addition that $Y$ contains an oriented surface $F$ meeting
$\Sigma=\Sigma_{1}\cup\Sigma_{2}$ transversely, and that the
diffeomorphism $h:\Sigma_{1}\to\Sigma_{2}$ carries the oriented
intersection $\Sigma_{1}\cap F$ to $\Sigma_{2}\cap F$. In this case,
we can form an oriented surface $\tilde F$ in the new $3$-manifold
$\tilde Y$, by cutting $F$ and regluing. We suppose that neither $F$
nor $\tilde F$ contains a $2$-sphere, and that every component of $Y$
contains a component of $F$ whose genus is at least $2$.

\begin{theorem}\label{thm:Floer-excision-genus-1}
Suppose $\tilde Y$ and $\tilde F$ are obtained from $Y$ and $F$ as above, with
$\Sigma_{1}$ and $\Sigma_{2}$ both of genus $1$. Let $\tilde \eta$ be
the
$1$-cycle in $\tilde Y$ formed from the cycle $\eta$ in
$Y$ as shown in Figure~\ref{fig:Excision-setup}.
Assume as usual that $t-t^{-1}$ is invertible in the ring $\cR$.
Then there is an isomorphism:
\[
            \HM_{\bullet}(Y|F;\Gamma_{\eta})
                    \to    \HM_{\bullet}(\tilde Y|\tilde F;\Gamma_{\tilde \eta}) .
\]
\end{theorem}

\begin{remark}
    Note again that if $Y$ has two components and $\cR$ is a field,
    then the left-hand-side is the tensor product
    \[
 \HM_{\bullet}(Y_{1}|F_{1};\Gamma_{\eta_{1}}) \otimes_{\cR}
 \HM_{\bullet}(Y_{2}|F_{2};\Gamma_{\eta_{2}}).
    \]
\end{remark}

There is also a simpler way in which local coefficients can enter into
the excision theorem, when the cycle $\eta$ does not intersect
$\Sigma$. We state an adaptation of
Theorem~\ref{thm:Floer-excision} of this sort.

\begin{theorem}\label{thm:Floer-excision-extra-eta}
Let $\tilde Y$ be obtained from $Y$ as in
Theorem~\ref{thm:Floer-excision},
with $\Sigma$ of genus at least two. Let $\eta_{0}$ be a  1-cycle in
$Y$, disjoint from $\Sigma$. This becomes a cycle also in $\tilde Y$,
which we   denote by $\tilde\eta_{0}$. Then we have an isomorphism
of $\cR$-modules:
\[            
            \HM_{\bullet}(Y|\Sigma;\Gamma_{\eta_{0}})
         \to   \HM_{\bullet}(\tilde Y|\tilde \Sigma;\Gamma_{\tilde\eta_{0}}).
\]
\end{theorem}

In  Theorem~\ref{thm:Floer-excision-extra-eta}, consider the case
that $Y=Y_{1}\cup Y_{2}$ and $\eta_{0}$ is contained in $Y_{1}$. In
this case, the chain complex that computes the group
$\HM_{\bullet}(Y|\Sigma;\Gamma_{\eta_{0}})$ on the left is
\[            
            C_{\bullet}(Y_{1}|\Sigma_{1};\Gamma_{\eta_{0}})
            \otimes_{\Z}
            C_{\bullet}(Y_{2}|\Sigma_{2}),
\]
(the tensor product of a complex of free $\cR$-modules and a complex
of free abelian groups, both finitely generated). By the K\"unneth
theorem, if $\cR$ has no $\Z$-torsion and
$\HM_{\bullet}(Y_{1}|\Sigma_{1};\Gamma_{\eta_{0}})$ is a free
$\cR$-module, then the theorem provides an isomorphism
\begin{equation}\label{eq:excision-extra-twist}            
            \HM_{\bullet}(Y_{1}|\Sigma_{1};\Gamma_{\eta_{0}})
            \otimes
            \HM_{\bullet}(Y_{2}|\Sigma_{2})
         \to   \HM_{\bullet}(\tilde Y|\tilde\Sigma;\Gamma_{\tilde\eta_{0}}).
\end{equation}    
As a particular application of this result, we have:

\begin{corollary}\label{cor:eta-no-change}
    Let $\Sigma\subset Y$ be a closed, oriented surface whose
    components have genus at least $2$ and let $\eta$ be a $1$-cycle
    in $Y$ whose support lies in $\Sigma$. Suppose that
    $\cR$ has no $\Z$-torsion. Then
    \[
                \HM_{\bullet}(Y|\Sigma;\Gamma_{\eta}) \cong
                \HM_{\bullet}(Y|\Sigma)\otimes \cR.
    \]
\end{corollary}

\begin{proof}
   Apply the isomorphism of \eqref{eq:excision-extra-twist} with
   $(Y_{2},\Sigma_{2})=(Y,\Sigma)$ and
   $(Y_{1},\Sigma_{1})=(\Sigma\times S^{1}, \Sigma\times\{p\})$. Take
   $\eta_{0}$ in $\Sigma\times S^{1}$ to be the cycle corresponding to
   $\eta$. By
   Proposition~\ref{cor:product-rank-1-twist} we have
   \[
           \HM_{\bullet}(Y_{1}|\Sigma_{1};\Gamma_{\eta_{0}}) = \cR.
   \]
   The manifold $\tilde Y$ is another copy of the original $Y$ and
   $\tilde \Sigma$ is two parallel copies of $\Sigma$. The cycle
   $\eta_{0}$ becomes now the original $1$-cycle $\eta$, so
   \[
            \HM_{\bullet}(\tilde Y|\tilde\Sigma;\Gamma_{\tilde
            \eta_{0}}) = \HM_{\bullet}(Y|\Sigma;\Gamma_{\eta}) .
   \]
   Thus \eqref{eq:excision-extra-twist} gives an isomorphism
   \[
            \cR \otimes\HM_{\bullet}(Y|\Sigma) \to
            \HM_{\bullet}(Y|\Sigma;\Gamma_{\eta}).
    \]
\end{proof}

\subsection{Proof of the excision theorems}

The proof of Theorem~\ref{thm:Floer-excision}
is very much the same as Floer's proof of his original excision
theorem, as described in \cite{Braam-Donaldson}. 
The first step (which is common to both
Theorem~\ref{thm:Floer-excision} and
Theorem~\ref{thm:Floer-excision-genus-1}) is to construct a cobordism
$W$ from $\tilde Y$ to $Y$. In the case that $Y$ is 
disjoint union $Y_{1} \cup
Y_{2}$, the cobordism $W$ admits a map $\pi: W\to P$, where $P$
is a $2$-dimensional pair-of-pants cobordism.
This is shown schematically in Figure~\ref{fig:Excision-cobordism-1}.
The $4$-dimensional
cobordism is the union of two pieces. The first piece is the product
 $[0,1]\times Y'$, where $Y'$ as before is obtained from $Y$ by
 cutting open along $\Sigma_{1}$ and $\Sigma_{2}$. (In the
Figure, this appears as the union of two pieces, corresponding to the
decomposition of $Y'$ as $Y'_{1}\cup Y'_{2}$.) The
second piece is the product of the closed surface $\Sigma_{1}$ with a $2$-manifold $U$ with
corners: $U$ corresponds to the gray-shaded area in the figure.
The two pieces are fitted
together as shown, using the diffeomorphism $h$. If $Y$ is connected,
then the picture looks just the same in the neighborhood of the
shaded region, but the product region $[0,1]\times Y'$ is connected;
the cobordism $W$ in this case does not admit a map to the pair of
pants.

\begin{figure}
    \begin{center}
        \includegraphics[scale=0.7]{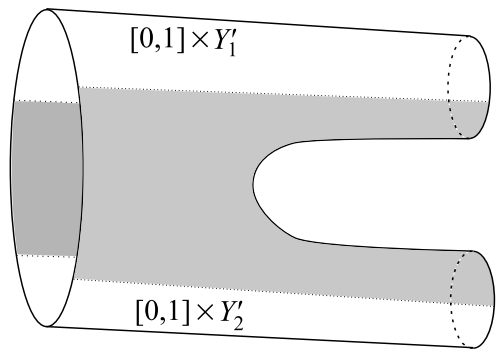}
    \end{center}
    \caption{\label{fig:Excision-cobordism-1}
    A cobordism $W$ from $\tilde Y$ to $Y= Y_{1} \cup Y_{2}$.}
\end{figure}

There is a very similar cobordism $\bar{W}$ which goes the other way:
Theorem~\ref{thm:Floer-excision} arises because the cobordisms $W$ and
$\bar W$ give rise to mutually inverse maps (in the case of genus at
least $2$)
\[
\begin{aligned}
\HM(W) :
            \HM_{\bullet}(\tilde Y|\tilde \Sigma) &\to
            \HM_{\bullet}(Y|\Sigma)
  \\
            \HM(\bar W) :
             \HM_{\bullet}(Y|\Sigma)
               &\to           
             \HM_{\bullet}(\tilde Y|\tilde \Sigma) .
             \end{aligned}
 \]
when the coefficients are a field.

To show that the cobordisms induce mutually inverse maps,
let $X$ be the cobordism from $\tilde Y$ to
$\tilde Y$ formed as the union of $W$ and $\bar W$. We must show that
$X$ gives rise to the identity map on $\HM_{\bullet}(\tilde Y|\tilde \Sigma)$.
This will show that $\HM(\bar{W})\circ \HM(W) =1$,
and there will be a similar argument for the other composite.
Note that $\Sigma$ and $\tilde\Sigma$
are homologous in $X$, so the map induced by $X$ really does factor
through $\HM_{\bullet}(Y|\Sigma)$, not just $\HM_{\bullet}(Y)$.

The manifold $X$ is shown
schematically in Figure~\ref{fig:Excision-cobordism-2} for the case
that $Y$ has two components, $Y_{1}\cup Y_{2}$, in which case it
admits a map $\pi$ to the twice-punctured genus-1
surface, as drawn.
\begin{figure}
    \begin{center}
        \includegraphics[scale=0.7]{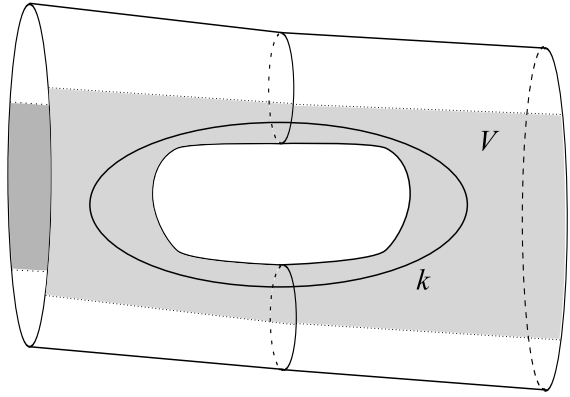}
    \end{center}
    \caption{\label{fig:Excision-cobordism-2}
   The composite cobordism $X$ from $\tilde Y$ to $\tilde Y$.}
\end{figure}
Over the shaded region $V$ it is a product,
\[
\begin{aligned}
\pi^{-1}(V) &= \Sigma_{1 } \times V \\
            &= \Sigma_{2} \times V.
\end{aligned}
\]
If $Y$ is connected, the picture is essentially the same in the
neighborhood of $\pi^{-1}(V)$.
Let $k$ be the closed curve in $V$ that is shown, and let $K$ be the
inverse image
\[
\begin{aligned}
                K &= \pi^{-1}(k) \\
                  &= \Sigma_{1}\times k.
\end{aligned}
\]
(We continue to identify $\Sigma_{1}$ with $\Sigma_{2}$ via $h$ in
what follows.)
Let $X'$ be the manifold-with-boundary formed by cutting along $K$.
Its boundary is two copies of $K$. Let $X^{*}$ be the new cobordism
from $\tilde Y$ to $\tilde Y$ obtained by attaching two copies of
$\Sigma_{1}\times
D^{2}$, with $\partial D^{2}$ being identified with $k$:
\[
                    X^{*} = X' \cup (\Sigma_{1}\times D^{2}) \cup
                    (\Sigma_{1}\times D^{2}).
\]

Floer's proof hinges on  the fact that the manifold $X^{*}$ is  just the
product cobordism from $\tilde Y$ to $\tilde Y$. This means that we
only need show that $X^{*}$ gives rise to the same map as $X$. This
desired equality can be deduced from the formalism of
section~\ref{subsec:disconnected-1}, for it is precisely
Proposition~\ref{prop:excision-prototype}. This concludes the proof
that $\HM(\bar{W})\circ \HM(W)=1$. The picture for the composite of
the two cobordisms in the other order is shown in
Figure~\ref{fig:Excision-cobordism-2b}. The proof that this composite
gives the identity is essentially the same: the relationship between
$Y$ and $\tilde Y$ is a symmetric one, except that we have allowed
only $Y$ to have two components.
Figure~\ref{fig:Excision-cobordism-2b} shows the corresponding
curve $\tilde k$ in this case, along which one must cut, just as we
cut along $k$ in the previous case. This
completes the proof of
Theorem~\ref{thm:Floer-excision}.

\begin{figure}
    \begin{center}
        \includegraphics[scale=0.7]{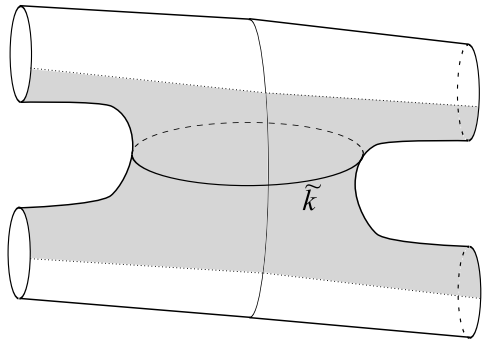}
    \end{center}
    \caption{\label{fig:Excision-cobordism-2b}
   The composite cobordism in the opposite order, from $Y$ to $Y$, in
   the case that $Y$ has two components.}
\end{figure}

The proof of Theorem~\ref{thm:Floer-excision-genus-1} is very similar.
The same cobordisms $W$ and $\bar{W}$ are used. In the cobordism $W$,
there is a $2$-chain $\nu_{W}$ whose boundary is $\eta - \tilde\eta$.
It consists of the product chain $[0,1]\times \eta'$ in part of $W$
obtained from $[0,1]\times Y$; while over the shaded region $U$ in
Figure~\ref{fig:Excision-cobordism-1}, the cycle $\nu_{W}$ is a
section $\{p\}\times U$ of $\Sigma_{1}\times U$. There is a similar
$2$-chain $\nu_{\bar{W}}$ in $\bar{W}$, and these fit together to give
a $2$-chain $\nu_{X}$ in the composite cobordism $X$
(Figure~\ref{fig:Excision-cobordism-2}). The $3$-manifold $K\subset X$
lying over the curve $k$ is now a $3$-torus, and $K$ meets $\nu_{X}$
transversely in a standard circle. The proof now proceeds as before,
but using Corollary~\ref{cor:excision-prototype-twist} in place of
Proposition~\ref{prop:excision-prototype}. We learn that the composite
cobordism $X$ gives a map which is $(t-t^{-1})$ times the map arising
from the trivial product cobordism $X^{\dag}$. That is,
\[
               \HM(\bar W | F_{\bar W}; \Gamma_{\nu_{\bar W}}) \circ
               \HM(W | F_{W}; \Gamma_{\nu_{W}})  = (t-t^{-1}).
\]
The same holds for the composite in the opposite order.
Since $t-t^{-1}$ is a unit in $\cR$, this means that
$\HM(W|F_{W};\Gamma_{\nu_{W}})$ is an isomorphism, as required.

\section{Monopole Floer homology for sutured manifolds}
\label{sec:Monopole-Floer-homology}

In this section, we give the definition of the monopole homology
groups for balanced sutured manifolds, which are the main object of
study in this paper.

\subsection{Closing up sutured manifolds}
\label{subsec:closing}

We recall Juh\'asz's definition of a balanced sutured manifold
\cite{Juhasz-1}, a
restricted version of Gabai's notion of a sutured manifold
\cite{Gabai}:

\begin{definition}
    A \emph{balanced sutured manifold} $(M,\gamma)$ is a compact,
    oriented
    $3$-manifold $M$ with boundary, equipped with the following data:
    \begin{enumerate}
    \item
        a closed, oriented $1$-manifold $s(\gamma)$ in $\partial M$,
        i.e.~a collection of disjoint oriented circles in the
        boundary, called the sutures;
    \item a union $A(\gamma)$ of annuli, which comprise a tubular
    neighborhood of $s(\gamma)$ in $\partial M$; the closure of
    $\partial M\setminus A(\gamma)$ is called $R(\gamma)$.
    \end{enumerate}
    These are required to satisfy the following conditions:
    \begin{enumerate}
        \item $M$ has no closed components;
        \item if the components of $\partial A(\gamma)$ are oriented in
        the same sense as the sutures, then it should be possible to orient
        $R(\gamma)$ so that its oriented boundary coincides with this
        given orientation of $A(\gamma)$;
        \item $R(\gamma)$ has no closed components (which implies that
        the orientation in the previous item is unique); we call it
        the \emph{canonical orientation};
        \item if we define $R_{+}(\gamma)$ (and $R_{-}(\gamma)$ also) as the
        subset of $R(\gamma)$ where the canonical orientation coincides
        with the boundary orientation (or its opposite, respectively),
        then $\chi(R_{+}(\gamma)) = \chi(R_{-}(\gamma))$. \CloseDef
    \end{enumerate}
\end{definition}

It is often helpful to consider sutured manifolds as manifolds with
corners: the corners run along the circles $\partial A(\gamma)$ and
separate the flat annuli from the rest of the boundary.
Note that $M$ need not be connected. A model example is a
\emph{product sutured manifold} \[([-1,1] \times \T, \delta).\] Here $\T$
is an oriented surface with non-empty boundary and no closed
components, and the sutures are
\[
                s(\delta) = \{0\} \times\partial\T
\]
with the boundary orientation. The annuli  $A(\delta)$ are $[-1,1]
\times\partial\T$, and we have
\[
\begin{aligned}
    R_{+}(\delta) &= \{1\}\times\T \\
    R_{-}(\delta) &= \{-1\}\times\T. \\
\end{aligned}
\]

Given a balanced sutured manifold $(M,\gamma)$, we form a closed,
oriented manifold $Y = Y(M,\gamma)$ as follows. The closed manifold is
dependent on some choices, as we shall see. First, we choose an
oriented
connected surface $\T$  whose boundary
components  are in one-to-one correspondence with the components of
$s(\gamma)$. We call $T$ the \emph{auxiliary surface}.
From $\T$ we form the product sutured manifold $([-1,1]\times
T,\delta)$ as
just described. We then glue the annuli $A(\delta)$ to the annuli
$A(\gamma)$: this is done by a map
\[
                A(\delta)\to A(\gamma)
\]
which is orientation-reversing with respect to the  boundary
orientations and which maps $\partial R_{+}(\delta)$ to $\partial
R_{+}(\gamma)$. The result of this step is a 3-manifold with exactly
two boundary components, $\bar R_{+}$ and $\bar R_{-}$, which are
closed orientable surfaces of equal genus:
\[
\begin{aligned}
\bar R_{+} &= R_{+}(\gamma) \cup \{1\}\times T \\
            \bar R_{-} &= R_{-}(\gamma) \cup \{-1\}\times T \\
            \end{aligned}
\]
We require $T$ to be of
sufficiently large genus (genus zero may suffice, and genus
two  always will) so that two conditions hold:
\begin{enumerate}
\renewcommand{\theenumi}{(C\arabic{enumi})}%
    \item \label{item:genus-2} the genus of $\bar R_{\pm}$ is at least two;
    \item \label{item:closed-curve} the surface $T$ contains a simple
    closed curve  $c$ such that $\{1\}\times c$ and $\{-1\}\times c$
    are non-separating curves in $\bar{R}_{+}$ and $\bar{R}_{-}$
    respectively.
\end{enumerate}
Finally, form
$Y(M,\gamma)$ by identifying $\bar R_{+}$ with $\bar R_{-}$ using any
diffeomorphism which reverses the boundary orientations
(i.e.~preserves the canonical orientations),
\[
    h : \bar R_{+} \to \bar R_{-}.
\]

Inside $Y$ is a closed, connected, non-separating surface $\bar R$, obtained from the
identification of $\bar R_{+}$ with $\bar R_{-}$. We can orient $\bar
R$ using the canonical orientation of $R_{+}(\gamma)$. As an oriented
pair, $(Y,\bar R)$ depends only on two things, beyond $(M,\gamma)$
itself: first, the choice of genus for $\T$, and second the choice
of diffeomorphism $h$ used in the last step.

\begin{definition}\label{def:closure}
We call $(Y,\bar{R})$ a
\emph{closure} of the balanced sutured manifold $(M,\gamma)$ if it is
obtained in this way, by attaching to $(M,\gamma)$ a product region
$[-1,1]\times T$ satisfying the above conditions and then attaching
$\bar{R}_{+}$ to $\bar{R}_{-}$ by some $h$. \CloseDef
\end{definition}

\subsection{The definition}

Let $Y=Y(M,\gamma)$ be formed from a sutured manifold $(M,\gamma)$ as
described in the previous subsection. Recall
that $Y$ contains a connected, oriented closed surface $\bar{R}$, by
construction, whose genus is at least two. We make the
following definition:

\begin{definition}\label{def:SHM}
    We define the monopole Floer homology of the sutured manifold
    $(M,\gamma)$ to be the finitely-generated abelian group
    \[
               \SHM(M,\gamma) :=    \HM_{\bullet}(Y|\bar{R}),
    \]
    where $Y=Y(M,\gamma)$ is a closure of $(M,\gamma)$ as described in
    Definition~\ref{def:closure}, and the notation on the right follows
    \eqref{eq:bar-F-notation}. \CloseDef
\end{definition}

    As it stands, this definition appears to depend on the choice of
    genus, $g$, for the auxiliary surface $T$, as well as on the
    choice of gluing diffeomorphism $h$. In
    section~\ref{sec:independence} we shall prove:

\begin{theorem}\label{thm:independence}
    The group $\SHM(M,\gamma)$ defined in \ref{def:SHM}
    depends only on $(M,\gamma)$, not
    on the choice of genus $g$ for the auxiliary surface $T$ or the
    diffeomorphism $h$.
\end{theorem}

There is a version of $\SHM$ with local coefficients that we shall use
at some points along the way. Recall that $T$ is required to contain a
curve $c$ that yields non-separating curves $\{\pm 1\}\times c$ on
$\bar{R}_{\pm}$. Let us choose the diffeomorphism $h$ so that $h$ maps
$\{1\}\times c$ to $\{-1\}\times c$, preserving orientation. Thus the
surface $\bar{R}$ in $Y(M,\gamma)$ now contains a closed curve
$\bar{c}$, the image of $\{\pm 1\}\times c$. Let $c'$ be any dual
curve on $\bar{R}$: a curve $c'$ with $\bar{c}\cdot c'=1$ on
$\bar{R}$.

\begin{definition}\label{def:SHM-Gamma}
    We define the monopole Floer homology of the sutured manifold
    $(M,\gamma)$ \emph{with local coefficients}
    to be the $\cR$-module
    \[
               \SHM(M,\gamma;\Gamma_{\eta}) :=
               \HM_{\bullet}(Y|\bar{R};\Gamma_{\eta}),
    \]
    where the closure $Y=Y(M,\gamma)$ is constructed using a
    diffeomorphism $h$ satisfying the constraint just described, and
    $\eta$ is the $1$-cycle in $Y$ carried by the curve $c'$
    dual to $\bar{c}$ as above. 
    \CloseDef
\end{definition}

We shall see  that this is independent of the choice of $\eta$. When
using local coefficients in this way, we can relax the requirement that
$\bar{R}$ has genus $2$ or more (condition \ref{item:genus-2} above)
and allow closures in which $\bar{R}$ has genus $1$:

\begin{proposition}\label{prop:SHM-Gamma-invariant}
    As long as $t-t^{-1}$ is invertible in $\cR$, 
    the $\cR$-module $\SHM(M,\gamma;\Gamma_{\eta})$ defined in
    \ref{def:SHM-Gamma}
    depends only on $(M,\gamma)$ and $\cR$, not
    on the remaining choices. Furthermore, subject to the same
    condition on $\cR$, one can relax the condition \ref{item:genus-2}
    above and allow $\bar{R}$ to have genus $1$ when using local
    coefficients.
\end{proposition}

In the case that
$\bar{R}$ does have genus $2$ or more, we shall also see that we
can take $\eta$ to be any
non-separating curve on $\bar{R}$, rather than a curve dual to
$\bar{c}$.

\subsection{Proof of independence}

\label{sec:independence}

We now  prove Theorem~\ref{thm:independence}:
our definition of the monopole Floer
homology of a balanced
sutured manifold $(M,\gamma)$ is independent of the choices made in
its definition. The proof consists of several applications of Floer's
excision theorem. We begin with an observation about mapping tori:

\begin{lemma}\label{lem:fibered}
    Let $Y\to S^{1}$ be a fibered $3$-manifold whose fiber $R$ is a
    closed surface of genus at least $2$. Then $\HM(Y|R)\cong\Z$.
\end{lemma}

\begin{proof}
    In the case of the product fibration, we have already seen this in
    the previous section. If $Y_{h}$ denotes the mapping torus of a
    diffeomorphism $h:R\to R$, then the excision theorem,
    Theorem~\ref{thm:Floer-excision}, in the guise of
    \eqref{eq:tensor-product-1}, gives us an injective map
    \[
                \HM(Y_{h}|R) \otimes \HM(Y_{g}|R) \to
                  \HM(Y_{gh}|R)
    \]
    with cokernel the Tor term. When $g=h^{-1}$, the mapping tori
    $Y_{h}$ and $Y_{g}$ are orientation-reversing diffeomorphic; and
    $\HM(Y_{g}|R)$ is therefore isomorphic to $\HM(Y_{h}|R)$ as an abelian
    group. (This is for the same reason that the homology and cohomology
    of a finitely-generated complex of free $\Z$-modules are
    isomorphic, as abelian groups.) So we obtain an injective map
    \[
                    \HM(Y_{h}|R) \otimes \HM(Y_{h}|R) \to
                  \Z
    \]
    whose cokernel is torsion. This forces
    $\HM(Y_{h}|R)$ to be $\Z$.
\end{proof}

\begin{corollary}
    Let $Y_{1}$ be a closed oriented $3$-manifold containing a
    non-separat\-ing oriented surface $\bar{R}$ of genus two or more.
    Let $\tilde Y$ be obtained from $Y_{1}$ by cutting along $\bar{R}$ and
    re-gluing by an orientation-preserving diffeomorphism $h$. Then
    $\HM(Y_{1}|\bar{R})$ and $\HM(\tilde Y|\bar{R})$ are isomorphic.
\end{corollary}

\begin{proof}
    Apply the excision theorem, Theorem~\ref{thm:Floer-excision},
    with $Y=Y_{1}\cup Y_{2}$,
    taking $Y_{2}$ to be the mapping torus of $h$ and
    $\Sigma_{1}=\Sigma_{2}=\bar{R}$. Lemma~\ref{lem:fibered}
    tells us that $\HM(Y_{2}|\bar{R})\cong\Z$, so
    $\HM(Y_{1}|\bar{R}) \cong \HM(\tilde Y|\bar{R})$ by the excision
    theorem.
\end{proof}

Consider now the situation of Theorem~\ref{thm:independence}. We have
a closed $3$-manifold $Y=Y(M,\gamma)$ whose construction depends on a
choice of genus $g$ for $T$ and a choice of diffeomorphism $h$. We are
always supposing that $Y$ has been constructed using an auxiliary
surface $T$ subject to the conditions \ref{item:genus-2} and
\ref{item:closed-curve}. The
above corollary tells us that $\HM(Y|\bar{R})$ is independent of the
choice of $h$. So the group $\SHM(M,\gamma)$, as we have defined it,
depends only on the choice of $g$. Let us temporarily write it as
\begin{equation}\label{eq:SHM-g}
                    \SHM^{g}(M,\gamma).
\end{equation}
We can apply the same arguments with local coefficients:
Theorem~\ref{thm:Floer-excision-extra-eta} can be used in place of
Theorem~\ref{thm:Floer-excision} to see that
\begin{equation}\label{eq:SHM-Gamma-g}
                \SHM^{g}(M,\gamma;\Gamma_{\eta})
\end{equation}
(as defined in Definition~\ref{def:SHM-Gamma}) depends at most on the
choice of $g$, not on $h$ (as long as conditions \ref{item:genus-2}
and \ref{item:closed-curve} hold). However, we can also relate
\eqref{eq:SHM-g} to \eqref{eq:SHM-Gamma-g} directly:

\begin{lemma}\label{lem:straight-tensor}
    If the coefficient ring $\cR$ has no $\Z$-torsion, then
    we have \[ \SHM^{g}(M,\gamma;\Gamma_{\eta}) =
    \SHM^{g}(M,\gamma)\otimes \cR. \]
\end{lemma}

\begin{proof}
    In the definition of the local system $\Gamma_{\eta}$, the
    $1$-cycle $\eta$ is parallel to a curve lying on $\bar{R}$.
    The result therefore follows from the definitions and
    Corollary~\ref{cor:eta-no-change}.
\end{proof}

Because we already know that $\SHM^{g}(M,\gamma)$ is independent of
$h$, the above lemma establishes that $
\SHM^{g}(M,\gamma;\Gamma_{\eta})$ is also independent of $h$, and that
it is also independent of the choice of $\eta$.
Next we prove:

\begin{proposition}\label{prop:independence-local}
    If $t-t^{-1}$ is invertible in the coefficient ring $\cR$ and
    $\cR$ has no $\Z$-torsion,
    then the Floer group with local coefficients,
    $\SHM^{g}(M,\gamma;\Gamma_{\eta})$, is independent of $g$.
\end{proposition}

\begin{proof}
Fix $g_{1}$ and let $T$ be a surface of genus $g_{1}$. Let $Y_{1}$ be
the resulting closure of $(M,\gamma)$, and  write $\bar{R}_{1}$ for
the surface it
contains.
Recall that we required $T$ to contain a
simple closed curve $c$ such that $\{1\}\times c$ and $\{-1\}\times c$
are non-separating in $\bar{R}_{\pm}$. We can form a surface
$\tilde T$ of genus $g_{1}+1$ by the following process. We
take a closed surface $S$ of genus $2$ containing a non-separating
closed curve $d$. We then cut $T$ along $c$ and cut $S$ along $d$, and
we reglue to form $\tilde T$ as shown in Figure~\ref{fig:Genus-add}.
\begin{figure}
    \begin{center}
        \includegraphics[scale=0.7]{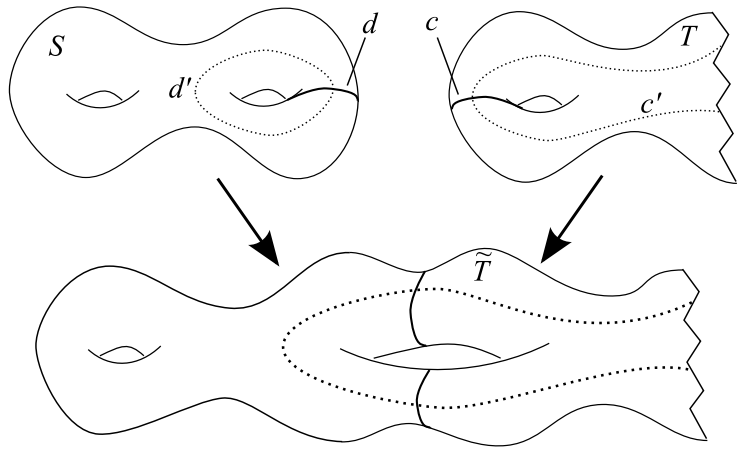}
    \end{center}
    \caption{\label{fig:Genus-add}
   Increasing the genus of $T$ by $1$.}
\end{figure}
The figure also shows curves $c'$ and $d'$ dual to $c$ and $d$. The
curve $c'$ is supposed to be extended (out of the picture) to become a
simple closed curve dual to $c$ in the larger surface
$\bar{R}_{+}=R_{+}(\gamma) \cup \{1\}\times T$.

In forming the closure $Y_{1}$ using $T$, we can arrange that
the diffeomorphism $h : \bar R_{+} \to \bar R_{-}$ carries $\{1\}
\times c$ to $\{-1\}\times c$ by the identity map on $c$.
This is because any
non-separating curve is equivalent to any other in an oriented
surface.  This implies that $Y_{1}$ can be identified with the product
$S^{1}\times T$ over some neighborhood of $c$ in $T$.
So $Y_{1}$ contains a  torus, $S^{1}\times c$. The dual curve $c'$ on
$\bar{R}_{+}$ becomes a curve (also called $c'$) in $Y_{1}$ which intersects the
torus $S^{1}\times c$ once.

We will
apply the second version of the excision theorem,
Theorem~\ref{thm:Floer-excision-genus-1}, as follows. We take
$Y=Y_{1}\cup Y_{2}$ with $Y_{1}$
as given, and  $Y_{2} = S^{1}\times S$.
We take $\Sigma_{1}$ to be the torus $S^{1}\times c$ inside
the product region of $Y_{1}$ and
$\Sigma_{2}$ to be $S^{1}\times d$. We take $\eta$ in $Y$ to be
$\eta_{1}+\eta_{2}$, where the cycle $\eta_{1}$ is
$c'$ and $\eta_{2}$ is
$\{\mathrm{point}\}\times d'$. These $1$-cycles intersect the
respective tori once each; and $\eta_{1}$ is of the sort required for the
definition of $\SHM^{g_{1}}(M,\gamma;\Gamma_{\eta_{1}})$ in
Definition~\ref{def:SHM-Gamma}. To play the role of the surface
$F=F_{1}\cup F_{2}$ in Theorem~\ref{thm:Floer-excision-genus-1} we
take $\bar{R}_{1}\cup \bar{R}_{2}$, where $\bar{R}_{2}$ is the genus-2
surface $\{\mathrm{point}\}\times S$.

The manifold $\tilde Y$ obtained from $Y= Y_{1} \cup Y_{2}$ in the excision
theorem  is
another closure of the original $(M,\gamma)$, using the auxiliary
surface $\tilde T$ of genus one larger than $T$ and a diffeomorphism
$\tilde h$ obtained by extending $h$ trivially over the extra handle.
It contains a closed surface $\tilde{R}$ whose genus is one larger
than the genus of $\bar{R}_{1}$. This is the surface obtained from
$\bar{R}_{1}$ and $\bar{R}_{2}$ by cutting and gluing.
So to prove the proposition, we
must prove
\begin{equation}\label{eq:for-propn}
                \HM(Y_{1}|\bar{R}_{1};\Gamma_{\eta_{1}})\cong
                \HM(\tilde Y|\tilde {R};\Gamma_{\tilde \eta}).
\end{equation}
Theorem~\ref{thm:Floer-excision-genus-1}, provides an isomorphism
\[
                \HM((Y_{1}\cup Y_{2})|(\bar{R}_{1}\cup\bar{R}_{2});\Gamma_{\eta})\to 
                \HM(\tilde Y | \tilde{R};\Gamma_{\tilde \eta}).
\]
But $\HM(Y_{2}|\bar{R}_{2};\Gamma_{\eta_{2}})$ is just $\cR$ by
Corollary~\ref{cor:product-rank-1-twist},
because this manifold is a product, so \eqref{eq:for-propn} follows
from the K\"unneth theorem. This completes the proof of the proposition.
\end{proof}

\begin{remark}
    Although Figure~\ref{fig:Genus-add} is drawn so as to make clear
    that the excision theorem is applicable, the topology can be
    described more simply.
    Let $G$ denotes the genus-one surface with
    one boundary component,
    obtained by cutting $S$ open along $d$ and then removing a
    neighborhood of $d'$. Then the operation of forming $\tilde T$ as
    shown is the same as removing a neighborhood of the point $x = c \cap
    c'$ and attaching $G$ to the boundary so created: a connected sum
    in other words. The $3$-manifold picture is obtained from this
    connected-sum picture by multiplying with by $S^{1}$. That is, we
    drill out a neighborhood of $S^{1}\times\{x\}$ and glue in
    $S^{1}\times G$.
\end{remark}

Now we can complete the proof of the theorem:

\begin{proof}[Proof of Theorem~\ref{thm:independence}]
    We have seen that there is no dependence on the choice of
    diffeomorphism $h$, and we have been considering the dependence on
    the genus $g$: we wish to show that $\SHM^{g}(M,\gamma)$ is
    independent of $g$. From Lemma~\ref{lem:straight-tensor} and
    Proposition~\ref{prop:independence-local}, we learn that the
    $\cR$-module
    \[
                \SHM^{g}(M,\gamma)\otimes \cR
    \]
    is independent of $g$ whenever $\cR$ has no $\Z$-torsion
    and $t-t^{-1}$ is
    invertible. But if $A$ and $B$ are finitely-generated abelian
    groups and $A\otimes \cR\cong B\otimes \cR$  as $\cR$-modules for
    all such $\cR$, then we must have $A\cong B$. For this one can
    take a universal example for $\cR$, namely the ring obtained by
    inverting $t-t^{-1}$ in the $\Z[\R]$, the group ring of $\R$.
\end{proof}

Finally, we turn to Proposition~\ref{prop:SHM-Gamma-invariant}. Up
until this point we have been assuming that $\bar{R}$ has genus $2$ or
more. But the proof of Proposition~\ref{prop:independence-local} works
just as well in the genus $1$ case. Thus if $Y_{1}$ is a closure
formed with $\bar{R}_{1}$ of genus $1$ and $(\tilde{Y}, \tilde{R})$ is
formed as in the proof of Proposition~\ref{prop:independence-local}
with $\bar{R}$ of genus $2$, then
\[
                \HM_{\bullet}(Y_{1}|\bar{R}_{1};\Gamma_{\eta_{1}})
                \cong \HM_{\bullet}(\tilde{Y} | \tilde{R}l
                \Gamma_{\tilde\eta} ).
\]
The group on the right is something we already know to be
independent of other choices: we have therefore
\begin{equation}\label{eq:genus-1-clarify}
          \HM_{\bullet}(Y_{1}|\bar{R}_{1};\Gamma_{\eta_{1}})
          = \SHM(M,\gamma)\otimes \cR.
\end{equation}
This verifies Proposition~\ref{prop:SHM-Gamma-invariant}. \qed

\section{Knot homology}
\label{sec:knot-homology}

Juh\'asz showed in \cite{Juhasz-1} that knot homology could be
obtained as a special case of his (Heegaard) Floer homology of a
sutured manifold. Specifically, given a knot $K$ in a closed $3$-manifold
$Z$, one can form a sutured manifold $(M,\gamma)$ by taking $M$ to be
the knot complement (with a torus boundary) and taking the sutures to
be two oppositely-oriented meridians. In the monopole case  we have
at present no a priori notion of knot homology; but we are free to
take Juh\'asz's prescription as  a \emph{definition} of knot homology
and pursue the consequences.  Thus:

\begin{definition}
    For a knot $K$ in a closed, oriented $3$-manifold $Z$, we define
    the monopole knot homology $\KHM(Z,K)$ to be the monopole homology
    of the sutured manifold $(M,\gamma)$ associated to $(Z,K)$ by
    Juh\'asz's construction. That is,
    \[
                \KHM(Z,K) := \SHM(M,\gamma)
    \]
    where $M=Z\setminus N^{\circ}(K)$ is the knot complement and
    $s(\gamma)$ consists of two oppositely-oriented meridians.
    \CloseDef
\end{definition}

To understand what this definition leads to, we must construct a
suitable closure of the sutured manifold.

\subsection{Closures of knot complements}

So let $K$ be a knot in a closed manifold $Z$, and let $(M,\gamma)$ be
the knot complement, with two sutures as just described. We can
describe a particularly simple closure of $(M,\gamma)$ as follows, if
we temporarily relax the rules and allow the auxiliary surface $T$ to
be an annulus. (The reason this is not a valid closure of $(M,\gamma)$
for our purposes is that the resulting surfaces $\bar{R}_{\pm}$ will
have genus $1$. We will correct this shortly, replacing the annulus by
a surface of genus $1$.) Let $N$ be a closed tubular neighborhood of
$K$, and let $N'\subset N$ be a smaller one.
Let $m$ be a meridian of $K$, lying outside $N'$ but inside $N$. 
We will consider $M$ to
be $\overline {Z\setminus N'}$, and we take two meridional sutures
$s(\gamma)$ on
the boundary of $N'$. If we take $T$ to be an annulus and attach
$[-1,1]\times T$ by gluing the two annuli
$[-1,1]\times \partial T$ to the sutures $A(\gamma)$, then what
results is a $3$-manifold $L$ with two tori as boundary components: we can
identify it with the complement of a tubular neighborhood of $m$ in
$M$.

Figure~\ref{fig:Knot-homology-2} shows the part of $L$ that lies
inside the tubular neighborhood $N$ of $K$. (The top and bottom are
identified.) The figure shows a vertical solid torus $N$ with a
smaller vertical solid torus $N'$ drilled out of it, as well as a
neighborhood $U$ of the meridian $m$, which has also been removed. The
boundary of $L$ consists of the inner vertical boundary (the boundary
of $N'$) and the boundary of the horizontal solid torus (the boundary
of $U$).  These boundary components are $\bar R_{+}$ and $\bar R_{-}$.

\begin{figure}
    \begin{center}
        \includegraphics[scale=0.7]{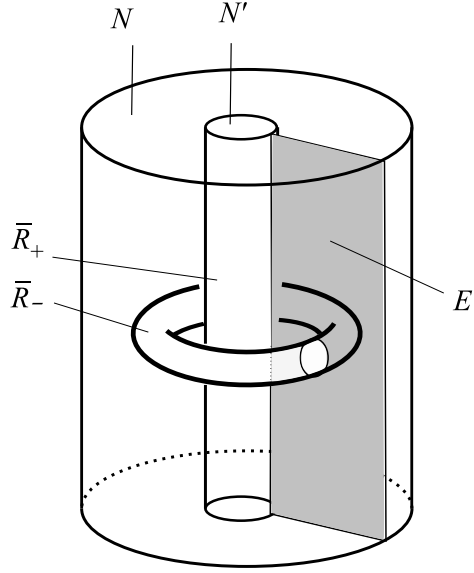}
    \end{center}
    \caption{\label{fig:Knot-homology-2}
   The part of $L$ lying inside the larger tubular neighborhood $N$.}
\end{figure}

If we choose a  framing of $K$, then we
obtain a fibration of $L \cap N$ by punctured annuli $E$ (one of which is
shown gray in the figure).
We now form the closure $Y_{1}=Y(M,\gamma)$ using $T$ as the auxiliary
surface by gluing $\bar R_{+}$ to $\bar R_{-}$: on each punctured
annulus $E$,  we
glue the circle $E \cap\bar R_{+}$ to $E \cap\bar R_{-}$. This turns
each annulus $E$ into a genus-1 surface $F$ with one boundary
component. (The remaining boundary component of $F$ lies on the outer
torus, $\partial N$.) Thus we have seen:

\begin{lemma}\label{lem:picture-F}
    Using an annulus $T$ as the auxiliary surface, a closure of the
    sutured manifold $(M,\gamma)$ associated to a knot $K$ in $Z$ can
    be described by taking a surface $F$ of genus one, with one
    boundary component, and gluing $F \times S^{1}$ to the knot
    complement $\overline{Z\setminus N}$. The gluing is done so that
    $\{p\}\times S^{1}$ is is attached to the meridian of $K$ on
    $\partial N$ and $\partial F \times \{q\}$ is glued to any chosen
    longitude of $K$ on $\partial N$. \qed
\end{lemma}

A shorter way to say what we have done is to that we have glued
together two knot complements: for the knot $K$ in $Z$ and the
standard circle ``knot'' in the $3$-torus, using any chosen framing of
the former and the standard framing of the latter, attaching
longitudes to meridians and meridians to longitudes. We give a name to
this closed manifold:

\begin{definition}\label{def:Closure-1}
    We write $Y_{1}(Z,K)$ for the closed $3$-manifold obtained from the
    framed knot $K$ in $Z$ by the construction just described. \CloseDef
\end{definition}

\begin{figure}
    \begin{center}
        \includegraphics[scale=0.7]{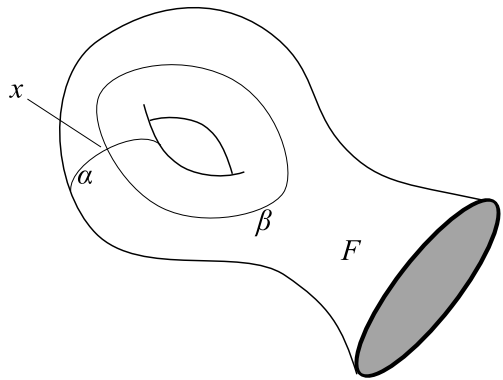}
    \end{center}
    \caption{\label{fig:F-alone}
   The closed surface $F$ obtained by gluing two boundary components
   of $E$: a genus-one surface with one boundary component. The curve
   $\alpha$ is the gluing locus.}
\end{figure}

As we pointed out at the beginning of this subsection, we have
described a closure of the sutured manifold $(M,\gamma)$ that is
illegitimate, because $T$ is an annulus and $\bar{R}$ has genus one. We now
described how $Y_{1}(Z,K)$ gets modified if we use a surface $\tilde{T}$
of genus one (still with two boundary components) instead of $T$.
Figure~\ref{fig:F-alone} shows the surface $F$. The curve $\alpha$ on
$F$ is the intersection of $F$ with the torus $\bar R \subset Y(Z,K)$
where $\bar{R}_{+}$ and $\bar R_{-}$ are glued. Thus $\bar R$
is the torus
\[
\begin{aligned}
\bar {R} &= \alpha\times S^{1} \\
                    &\subset F\times S^{1} \\
                    &\subset Y(Z,K)
                    \end{aligned}
\]
The image of $[-1,1]\times T$ in $Y(Z,K)$ is a copy of $S^{1}\times T$
and can be identified with the neighborhood of $\beta \times S^{1}$:
\[
            S^{1}\times T =\mathrm{nbd}( \beta) \times S^{1} \subset
            F\times S^{1},
\]
where $\mathrm{nbd}(\beta)\subset F$ is an annular neighborhood of
$\beta$.
The identification of the various factors is as indicated: the $S^{1}$
factor in $S^{1}\times T$ becomes the $\beta$ factor on the right, and
the core of the annulus $T$ becomes the $S^{1}$ factor on the right.
Recalling the remark made at the end of
section~\ref{sec:independence}, we see that to effectively increase
the genus of the auxiliary surface by $1$, we should:
\begin{enumerate}
    \item drill out a tubular neighborhood $\beta\times D^{2}$ of the circle
    $\beta\times\{q\}\subset F\times S^{1}$;
    \item attach $S^{1}\times G$, where $G$ is a genus-one surface
    with one boundary component, by a diffeomorphism
    \[
                    S^{1}\times \partial G \to \beta\times \partial
                    D^{2}
    \]
    which preserves the order of the factors.
\end{enumerate}
(In the second step, a framing of $\{q\}\times\beta$ is needed, but we have a
preferred one because $\beta$ lies on $\{q\}\times F$.)

\begin{definition}\label{def:Closure-2}
    We write $\tilde{Y}(Z,K)$ for the manifold obtained from $Y_{1}(Z,K)$
    by the two steps just described. It is a closure of the sutured
    manifold $(M,\gamma)$ associated to the knot $K$ in $Z$ obtained
    using a genus-one auxiliary surface $\tilde T$; and it
    depends only on a choice of framing for $K$. \CloseDef
\end{definition}

While the closure $Y_{1}(Z,K)$ has a genus-one surface $\bar{R}$, the
closure $\tilde Y(Z,K)$ has a genus-two surface $\tilde{R}$. The
latter is obtained from $\bar{R}=\alpha\times S^{1}$ by removing a
neighborhood of the point $(x,q)$ in $\alpha\times S^{1}$ and
adding the genus-one surface $\{x\}\times G$. To summarize this
discussion, we have the following, essentially by definition now:

\begin{corollary}\label{cor:KHM-from-Y}
    The monopole knot homology $\KHM(Z,K)$ can be computed as the
    ordinary monopole homology $\HM(\tilde Y| \tilde R)$, where
    $\tilde Y = \tilde Y(Z,K)$ is as above. Any framing of $K$ can be
    used in the construction of $\tilde Y$. \qed
\end{corollary}

\begin{remark}
    Both $Y_{1}(Z,K)$ and $\tilde Y(Z,K)$ can be described alternatively
    as follows. Let $S$ be a closed surface of genus $l$, let $c$ be a
    non-separating simple closed curve on $S$, and let $\hat{c}$ be
    the curve $\{p\} \times c$ in the $3$-manifold $S^{1}\times S$.
    Let $N(\hat{c})$ be a tubular neighborhood. Let $Y_{l}$ be the
    result of gluing the complement of $\hat{c}$ to the complement of
    $K$:
    \begin{equation}\label{eq:Y-l-picture}
                Y_{l} =  (S^{1}\times S) \setminus
                N^{\circ}(\hat{c})  \cup_{\phi} Z\setminus N^{\circ}(K)
    \end{equation}
    where $\phi$ identifies the meridian curves of $\hat{c}$ to the
    longitudes of $K$ and vice versa. (We give $\hat{c}$ the obvious
    framing, and we recall that  a framing of $K$ has been chosen.)
    Then the manifold $Y_{1}$ is $Y_{1}(Z,K)$, and $Y_{2}$ is $\tilde
    Y(Z,K)$.
\end{remark}

We can also use the simpler manifold $Y_{1}$ to compute monopole knot
homology, as long as we switch to local coefficients. This is the
content of the next lemma.

\begin{lemma}\label{lem:use-Y1}
    If $t-t^{-1}$ is invertible in the coefficient ring $\cR$ and
    $\cR$ has no $\Z$-torsion, then
    the knot homology $\KHM(Z,K)\otimes \cR$ can be computed as
    $\HM_{\bullet}(Y_{1};\Gamma_{\hat\alpha})$, where $Y_{1}$ is the
    manifold described in Definition~\ref{def:Closure-1} and
    $\hat\alpha$ is the curve $\alpha\times \{p\}$ in $F\times
    S^{1}\subset Y_{1}$, regarded as a $1$-cycle.
\end{lemma}

\begin{proof}
    According to Proposition~\ref{prop:SHM-Gamma-invariant}, we can
    use the closure $Y_{1}$ to compute $\SHM(M,\gamma;\Gamma_{\eta})$.
    Together with Lemma~\ref{lem:straight-tensor}, this tells us that
    \[
                    \HM_{\bullet}(Y_{1};\Gamma_{\hat\alpha}) \cong
                    \SHM(M,\gamma)\otimes\cR,
    \]
    where $(M,\gamma)$ is the sutured manifold obtained from the knot
    complement by Juh\'asz's prescription. 
\end{proof}

\subsection{Properties of monopole knot homology}

Suppose that the knot $K\subset Z$ is null-homologous, and let
$\Sigma$ be a Seifert surface for $K$: an oriented embedded surface in
$Z\setminus N^{\circ}$ with boundary a simple closed curve on
$\partial N$. We can frame the knot $K$ so that $N$ is identified with
$K\times D^{2}$ and $\partial \Sigma$ is $K\times \{q'\}$ for some
$q' \in S^{1}$. We can also regard $\Sigma$ as a surface in the
manifold $Y_{1}(Z,K)$ (Definition~\ref{def:Closure-1}). The union of
$\Sigma$ and $F\times \{q'\}$ in $Y_{1}(Z,K)$ is a closed oriented surface
\[
                \bar\Sigma = \Sigma \cup (F\times\{q'\}) \subset
                Y(Z,K).
\]
Its genus is one more than the genus of $\Sigma$. The surface $F\times
\{q'\} \subset Y_{1}(Z,K)$ remains intact in the manifold $\tilde Y(Z,K)$
for $q'\ne q$
(Definition~\ref{def:Closure-2}), so we can regard $\bar\Sigma$ also
as a closed surface in $\tilde Y=\tilde Y(Z,K)$. Using the surface
$\bar\Sigma$, we can decompose $\KHM(Z,K)$ according to the first
Chern class of the \spinc{} structure. We write
\[
                \KHM(Z,K) = \bigoplus_{i\in \Z}\KHM(Z,K,i)
\]
where
\[
                \KHM(Z,K,i) = \bigoplus_{\substack{\s \in \Sp(\tilde Y |
                \tilde R) \\ \langle c_{1}(\s), [\bar\Sigma]
                \rangle=2i} }\mskip -20 mu \HM_{\bullet}(\tilde Y,\s).
\]
If $Z$ is not a homology sphere, then the decomposition by \spinc{}
structures may depend on the choice of the relative homology class for
the Seifert surface $\Sigma$, in which case one should write
\begin{equation}\label{eq:KHM-summands}
                \KHM(Z,K,[\Sigma],i)
\end{equation}
for the summands.

Some familiar properties of the (Heegaard) knot homology of
Ozsv\'ath-Szab\'o and Rasmussen carry over to this
monopole version.

\begin{lemma}
    The groups $\KHM(Z,K,i)$ and $\KHM(Z,K,-i)$ are isomorphic.
\end{lemma}

\begin{proof}
    The isomorphism arises from the isomorphism between
    $\HM_{\bullet}(Y,\s)$ and $\HM_{\bullet}(Y,\bar{s})$, where
    $\bar{s}$ is the conjugate \spinc{} structure.
\end{proof}

\begin{lemma}
    The group $\KHM(Z,K,i)$ is zero for $|i|$ larger than the genus of
    $\Sigma$.
\end{lemma}

\begin{proof}
    The adjunction inequality tells us that $\HM_{\bullet}(Y,\s)$ is
    zero for \spinc{} structures $\s$ with $c_{1}(\s)[\bar\Sigma]$
    greater than $2g(\bar\Sigma)-2$. The genus of $\bar\Sigma$ is one
    larger than the genus of $\Sigma$.
\end{proof}

% \begin{lemma}\label{lem:use-Y}
%     For $i\ne 0$, the group $\KHM(Z,K,i)$ can also be calculated using
%     $Y_{1}=Y_{1}(Z,K)$ in place of $\tilde Y$:
%     \[
%                 \KHM(Z,K) = \bigoplus_{\substack{\s \in \Sp(Y_{1})
%                 \\ \langle c_{1}(\s), [\bar\Sigma]
%                 \rangle=2i} } \mskip -20 mu \HM_{\bullet}( Y_{1},\s).
% \]
% \end{lemma}
% 
% \begin{proof}
%     The proof that the monopole homology of a sutured manifold is
%     independent of the choice of genus for the auxiliary surface $T$
%     carries over in this case, as long as we are dealing with
%     non-torsion \spinc{} structures.
% \end{proof}
% 
\begin{lemma}
    For a classical knot $K$ in $S^{3}$ of genus $g$, the monopole
    knot homology group
    $\KHM(S^{3},K,g)$ is non-zero.
\end{lemma}

\begin{proof}
    We use the description of $\tilde{Y}=\tilde{Y}(K)$ as the the
    manifold $Y_{2}$, where $Y_{l}$ is the manifold described by
    \eqref{eq:Y-l-picture}. Let $S$ be the genus-2 surface used there,
    let $c$ be the closed curve on $S$, and let $c'$ be a dual curve
    on $S$ meeting $c$ once.
     According to Gabai's results \cite{Gabai, Gabai-knots}, a Seifert
    surface $\Sigma$ of $K$ of genus $g$
    arises as a compact leaf of a taut foliation $\cF_{K}$ of
    $S^{3}\setminus N^{\circ}(K)$, and we can ask that the leaves of
    $\cF_{K}$ meet $\partial N(K)$ in
    parallel circles. On the other hand, $S^{1}\times S$ has a taut
    foliation $\cF_{S}$ which is transverse to the curve $\hat{c} = \{p\}\times
    c$. This foliation is obtained from the trivial product foliation
    by cutting alont the torus $S^{1}\times c'$ and regluing with a
    small rotation of the $S^{1}$ factor. Together, the foliations
    $\cF_{K}$ and $\cF_{S}$ define a foliation $\cF$ of
    $\tilde{Y}=Y_{2}$.
    The surface $\bar{\Sigma}$ sits inside $Y_{2}$ as the union of the
    Seifert surface $\Sigma$ and the punctured torus $(S^{1}\times
    c')\setminus D^{2}$. The \spinc{} structure $\s_{c}$ determined by
    $\cF$ has first Chern class of degree $2g$ on $\bar{\Sigma}$ and
    degree $2$ on the genus-2 surface, so
    $\HM_{\bullet}(\tilde{Y},\s_{c})$ is a summand of $\KHM(S^{3}, K,
    g)$ by definition. The non-vanishing theorem from
    section~\ref{subsec:adjunction} tells us that this group is
    non-zero.
\end{proof}

\begin{lemma}\label{lem:Alexander}
    Let $K$ be a  classical knot and let $\chi(K,i)$ denote the Euler
    characteristic of $\KHM(S^{3},K,i)$, computed using the canonical
    $\Z/2$ grading on monopole Floer homology \cite{KM-book}. Then the
    finite Laurent series
    \[
                \sum_{i} \chi(K,i) T^{i}
    \]
    is the symmetrized Alexander polynomial, $\Delta_{K}(T)$,
    for the knot $K$, up to an
    overall sign.
\end{lemma}

\begin{proof}
    In different guise, this is essentially the same result as
    that of Fintushel-Stern \cite{Fintushel-Stern-knot} and
    Meng-Taubes \cite{Meng-Taubes}. Let $\tilde Y = \tilde Y(S^{3},K)$
    be the usual closure of the sutured manifold associated to $(S^{3},
    K)$ as in Definition~\ref{def:Closure-2}, let $\tilde{R}$ be the
    genus-2 surface in $\tilde{Y}$ and let $\bar{\Sigma}\subset
    \tilde{Y}$ be the surface of genus $g+1$ formed from a Seifert
    surface $\Sigma$ for $K$ and the genus-1 surface $F$.

    The Euler characteristic can be computed from the Seiberg-Witten
    invariants of the manifold $S^{1}\times \tilde Y$. Specifically, regard
    both $\bar{\Sigma}$ and $\tilde{R}$ as surfaces in \[X_{K} =
    S^{1}\times \tilde{Y}.\] Take $\nu$ to be the $2$-cycle in $X_{K}$
    defined by $\bar{\Sigma}$ and consider the generating function
    $\m(X_{K},[\nu])$ as in
    \eqref{eq:SW-generating}, but modified to use only \spinc{}
    structures that are of top degree on $\tilde {R}$. We introduce
    the notation
    \begin{equation*}
                        \m'(X_{K},[\nu]) =
                        \sum_{\s\in\Sp(X_{K}|\tilde{R})} \m(X_{K},\s)
                        t^{\langle c_{1}(\s),[\nu]\rangle}.
    \end{equation*}
    We then have
    \[
                    \sum_{i}\chi(K,i) t^{2i} =  \m'(X_{K},[\nu]) .
    \]

    Let $X_{0}$ be the same type of $4$-manifold as $X_{K}$, but formed using the
    unknot in place of $K$. The corresponding $3$-manifold
    $\tilde{Y}_{0}$ is $S^{1}\times S$, where $S$ has genus $2$; so
    $X_{0}$ is $T^{2}\times S$. The remark following
    Corollary~\ref{cor:KHM-from-Y} explains that $\tilde{Y}$ is formed
    from $\tilde{Y}_{0}$ by drilling out a neighborhood of a curve
    $\hat{c}$ and gluing in the knot complement $M_{K}=
    S^{2}\setminus N^{\circ}(K)$. If follows that $X_{K}$ is formed
    from $X_{0}$ by a ``knot surgery'' in the sense of
    \cite{Fintushel-Stern-knot}. This, one drills out a neighborhood
    of the torus $S^{1}\times \hat{c}$ and glues in $S^{1}\times
    M_{K}$. In the formalism of section~\ref{subsec:disconnected-2},
    we can therefore compute the ratio
    \[
                    \m'(X_{K},[\nu]) / \m'(X_{0},[\nu])
    \]
    as the ratio of the invariants associated to $(S^{1}\times M_{K},
    \nu_{K})$ and $(S^{1}\times M_{0}, \nu_{0})$. Here $M_{0}$ is the
    knot complement for the unknot, and $\nu_{K}$ and $\nu_{0}$ are
    the $2$-chains defined by Seifert surfaces for $K$ and the unknot
    respectively. This ratio is precisely what is calculated in
    \cite{Fintushel-Stern-knot} (see also \cite[section
    42.5]{KM-book}), and it is equal to $\Delta_{K}(t^{2})$. The lemma
    follows.
\end{proof}

Given a null-homologous knot $K$ in a $3$-manifold $Z$, there is a
rather more straightforward way to arrive at a sutured manifold than
the one that leads to knot homology. We can simply choose a Seifert
surface $\Sigma$ for $K$ and cut the knot complement $Z\setminus
N^{\circ}(K)$ open along $\Sigma$. The result is a sutured manifold
$(M_{\Sigma}, \delta)$ with a single suture and having
$R_{+}(\delta)=R_{-}(\delta)=\Sigma$. The monopole Floer homology of
this sutured manifold captures the top-degree part of the monopole
knot homology:

\begin{proposition}\label{prop:sutured-vs-knot}
    In the above situation, let $g$ be the genus of the Seifert
    surface $\Sigma$, and suppose $g\ne 0$. Then
    $\SHM(M_{\Sigma}, \delta)$ is isomorphic to
    $\KHM(Z,K,[\Sigma], g)$.
\end{proposition}

\begin{proof}
    It is sufficient to prove that
    \[
                \SHM(M_{\Sigma}, \delta) \otimes \cR
                \cong \KHM(Z,K,[\Sigma], g) \otimes \cR
    \]
    when the coefficient ring $\cR$ has $t-t^{-1}$ invertible and no
    $\Z$-torsion.
    Lemma~\ref{lem:use-Y1} tells us that we can compute the right-hand
    side using the manifold $Y_{1}$ as
       \[
                \KHM(Z,K,[\Sigma],g) \otimes \cR  =
                            \HM_{\bullet}(Y_{1}|\bar{\Sigma};\Gamma_{\hat\alpha})
    \]
    where $\bar{\Sigma}$ is the surface of genus $g+1$ in $Y_{1}$. On
    the other hand, the same manifold $Y_{1}$ arises as a closure of
    $(M_{\Sigma},\delta)$ in the sense of
    section~\ref{subsec:closing}, so we
    also have
    \[
    \begin{aligned}
    \SHM(M_{\Sigma},\delta) \otimes \cR &=
                 \SHM(M_{\Sigma},\delta;\Gamma_{\eta})  \\
                 &=
                \HM_{\bullet}(Y_{1}|\bar{\Sigma};\Gamma_{\hat\alpha}).
                \end{aligned}
    \]
    This proves the proposition.
\end{proof}

\section{Fibered knots}

\subsection{Statement of the result}

In this section, we adapt the material from \cite{Ni-A} to show that
the monopole version of knot homology detects fibered knots. For the
most part, the arguments of \cite{Ni-A} carry over with little
modification.

A balanced sutured manifold $(M,\gamma)$ is a \emph{homology product}
if the inclusions $R_{+}(\gamma) \to M$ and $R_{-}(\gamma)\to M$ are
both isomorphisms on integer homology groups. The main target is the
following theorem.

\begin{theorem}\label{thm:fibered-M}
    Suppose that the balanced sutured manifold $(M,\gamma)$ is
    taut and a homology
    product.
    Then $(M,\gamma)$ is a product sutured manifold if and only if
    $\SHM(M,\gamma)=\Z$.
\end{theorem}

The application to fibered knots is a corollary:

\begin{corollary}\label{cor:fibered-knots}
    If $K\subset S^{3}$ is a knot of genus $g$, then $K$ is fibered if
    and only if $\KHM(S^{3},K,g)=\Z$.
\end{corollary}

\begin{proof}[Proof of the corollary]
    The ``only if'' direction is a straightforward matter: it
    follows from 
    Lemma~\ref{lem:fibered} and
    Proposition~\ref{prop:sutured-vs-knot}.  The interesting
    direction is the ``if'' direction, and this can be deduced from
    Theorem~\ref{thm:fibered-M} as follows.

    Suppose that $\KHM(S^{3},K,g)=\Z$. From Lemma~\ref{lem:Alexander}
    we learn
    that the Alexander polynomial of $K$ is monic and
    that its degree is $g$.
    Let $\Sigma$ be a Seifert surface for $K$ of genus $g$, and let
    $(M_{\Sigma},\delta)$ be the balanced sutured manifold obtained by
    cutting open the knot complement along $\Sigma$. As Ni observes in
    \cite[section 3]{Ni-A}, the fact that the Alexander polynomial is
    monic tells us that $(M_{\Sigma},\delta)$ is a homology product.
    The group $\SHM(M_{\Sigma},\delta)$ is isomorphic to
    $\KHM(S^{3},K,g)$ by Proposition~\ref{prop:sutured-vs-knot}, so
    $\SHM(M_{\Sigma},\delta)=\Z$. Theorem~\ref{thm:fibered-M} implies
    that $(M_{\Sigma},\delta)$ is a product
    sutured manifold, from which it follows
    that the knot complement is fibered.
\end{proof}

We will prove Theorem~\ref{thm:fibered-M} after some preliminary
material on further properties of $\SHM$.

\subsection{\Spinc{} structures}
\label{subsec:spinc-decomp}

The following definition of relative \spinc{} structures on sutured manifolds
coincides with that of Juh\'asz \cite{Juhasz-2}, in slightly different
notation.
If we regard $(M,\gamma)$ as  a manifold with corners, then it carries
a preferred $2$-plane field $\xi_{\partial}$ on its boundary: on
$R_{+}(\gamma)$ and $R_{-}(\gamma)$, we take $\xi_{\partial}$ to be the
tangent planes to the boundary, with the canonical orientation; and on
each component of
$A(\gamma)$ we take $\xi_{\partial}$ to have, as oriented basis, first
the outward normal to $M$ and second the direction parallel to the
oriented suture. On a $3$-manifold an oriented $2$-plane field defines
a \spinc{} structure; so $\xi_{\partial}$ gives a \spinc{} structure in
a neighborhood of the boundary. We define $\Sp(M,\gamma)$ to be the
set of extensions of $\s_{\partial}$ to a \spinc{} structure on all of
$M$, up to isomorphisms which are $1$ on $\partial M$.
We refer to elements of $\Sp(M,\gamma)$ as relative \spinc{}
structures.

Consider the process of forming the closure $Y=Y(M,\gamma)$. When we
attach $[-1,1]\times T$ to the annuli in $\partial M$, the $2$-plane
field $\xi_{\partial}$ extends in the obvious way, as the tangents to
$\{p\}\times T$. When we the attach $\bar{R}_{+}$ to $\bar{R}_{-}$
using $h$, we obtain a $2$-plane field on all of $Y(M,\gamma)$ except
the interior of the original $M$.  On the surface $\bar{R}\subset Y$,
this $2$-plane field is the tangent plane field. So we obtain a natural map
\begin{equation}\label{eq:spin-map}
   \epsilon:                \Sp(M,\gamma) \to \Sp(Y|\bar{R}).
\end{equation}

\begin{lemma}\label{lem:spinc}
    Let $\s_{1}, \s_{2} \in \Sp(M,\gamma)$ be relative \spinc{}
    structures whose difference element in $H^{2}(M,\partial M)$ is
    not torsion. Then we can choose
    $T$ and the diffeomorphism $h$ so that $\epsilon(\s_{1})$ and
    $\epsilon(\s_{2})$ are \spinc{} structures in $\Sp(Y|\bar{R})$
    whose difference is still non-torsion.
\end{lemma}

\begin{proof}
    The statement only concerns the difference elements. The dual of
    $H^{2}(M,\partial M;\Q)$ is $H^{1}(M;\Q)$, and what we must show
    is that given a non-zero element  $\alpha\in H^{1}(M)$, we can
    choose $T$ and $h$ so that $\alpha$ is in the image of the map
    \[
                    H^{1}(Y) \to H^{1}(M).
    \]
    To do this, consider as an intermediate step the manifold $Y'$
    with boundary $\bar{R}_{+}\cup\bar{R}_{-}$ formed from $M$ by
    attaching $[-1,1]\times T$. The map $H^{1}(Y') \to H^{1}(M)$ is
    surjective. Let $\beta$ be a class in $H^{1}(Y')$ which
    restricts to $\alpha$. Represent the dual of $\beta$ by a closed
    surface $(B,\partial B)$ in $(Y',\partial Y')$. By adding to $B$
    an annulus contained in the product region $[-1,1]\times T$ if
    necessary, we can be assured that $\partial B$ intersects both
    $\bar{R}_{+}$ and $\bar{R}_{-}$ in a collection of curves
    representing a  primitive, non-zero homology class. We can then
    modify $B$ without changing its class so that $\partial{B}$
    consists of two circles: a non-separating curve in each of
    $\bar{R}_{+}$ and $\bar{R}_{-}$. Finally, we choose the
    diffeomorphism $h:
    \bar{R}_{+}\to \bar{R}_{-}$ so as to match up these curves. In
    this way we obtain a closed surface $\bar{B}$ in $Y$ whose dual
    class in $H^{1}(Y)$ maps to $\alpha$ in $H^{1}(M)$.
\end{proof}

The following corollary is the tool used by Ghiggini
\cite{Ghiggini} in his proof of the original version of
Corollary~\ref{cor:fibered-knots} for genus-1 knots.

\begin{corollary}\label{cor:rank-2}
    Suppose that $(M,\gamma)$ admits two taut foliations $\cF_{1}$ and
    $\cF_{2}$ such that the corresponding \spinc{} structures $\s_{1}$
    and $\s_{2}$ have non-torsion difference element in
    $H^{2}(M,\partial M)$. Then $\SHM(M,\gamma)$ has rank at least
    $2$.
\end{corollary}

\begin{proof}
    Choose the closure $Y=Y(M,\gamma)$ so that $\epsilon(\s_{1})$ and
    $\epsilon(\s_{2})$ are different \spinc{} structures on $Y$, as
    Lemma~\ref{lem:spinc} allows. The foliations $\cF_{1}$ and
    $\cF_{2}$ extend in an obvious way to foliations of $Y$ belonging
    to the \spinc structures $\epsilon(\s_{1})$ and
    $\epsilon(\s_{2})$. By the non-vanishing theorem described in
    section~\ref{subsec:adjunction}, the Floer groups
    $\HM_{\bullet}(Y,\epsilon(\s_{1}))$ and
    $\HM_{\bullet}(Y,\epsilon(\s_{2}))$ both have non-zero rank. Both
    of these Floer groups contribute to $\HM_{\bullet}(Y|\bar{R})
    =\SHM(M,\gamma)$, because the \spinc{} structures
    $\epsilon(\s_{i})$ belong to $\Sp(Y|\bar{R})$. So $\SHM(M,\gamma)$
    has rank at least $2$.
\end{proof}

\subsection{Decomposition theorems}
\label{subsec:decomp}

The excision theorems, in addition to their role in showing that
$\SHM(M,\gamma)$ is well-defined, can be used in a straightforward way
to establish some decomposition which related the Floer homology of a
sutured manifold $(M,\gamma)$ to that of $(M',\gamma')$, obtained from
$(M,\gamma)$ by cutting along a surface. We record a few types of such
decomposition theorem here. To avoid various circumlocutions involving
tensor products and the K\"unneth theorem, we shall work over $\Q$
instead of $\Z$ here; and when using local coefficients we shall take
$\cR$ to be a field of characteristic zero: either $\R$ with the usual
exponential map, or the field of fractions of the group ring $\Q[\R]$.

\begin{proposition}\label{prop:disjoint-union}
    Suppose $(M,\gamma)$ is a disjoint union $(M_{1},\gamma_{1}) \cup
    (M_{2},\gamma_{2})$ and that both pieces are balanced. Then
    \[
                \SHM(M,\gamma;\Q) \cong \SHM(M_{1},\gamma_{1};\Q)
                \otimes \SHM(M_{2},\gamma_{2};\Q).
    \]
\end{proposition}

\begin{proof}
     It will be sufficient to prove this for the local coefficient
     versions, $\SHM(M,\gamma;\Gamma_{\eta})$, because of
     Lemma~\ref{lem:straight-tensor}.
     Form the closures $(Y_{1},\bar{R}_{1})$ and
     $(Y_{2},\bar{R}_{2})$  of $(M_{1},\gamma_{1})$ and
     $(M_{2},\gamma_{2})$ by attaching product regions $[-1,1]\times
     T_{1}$ and $[-1,1]\times T_{2}$ respectively. Let $c_{1}$ and
     $c_{2}$ be non-separating curves on $T_{1}$ and $T_{2}$. When
     forming the closures $Y_{1}$ and $Y_{2}$, choose the
     diffeomorphisms $h_{1}$ and $h_{2}$ so that $h_{i}$ maps
     $\{1\}\times c_{i}$ to $\{-1\}\times c_{i}$, as in the proof of
     Proposition~\ref{prop:independence-local}. Let $\tilde T$ be the
     connected closed surface obtained from $T_{1}$ and $T_{2}$ by
     cutting open along $c_{1}$ and $c_{2}$ and reattaching, similarly
     to Figure~\ref{fig:Genus-add}. Let $\tilde h$ be the diffeomorphism of
     $\tilde T$ that arises from $h_{1}$ and $h_{2}$, and let $\tilde Y$ be the
     closure of $(M,\gamma)$ that is obtained by attaching $\tilde T$ to
     $(M,\gamma)$ and gluing up using $\tilde h$. We now have a
     connected closure $\tilde Y$ that is related to $Y=Y_{1}\cup
     Y_{2}$ by cutting and gluing along $2$-tori $S^{1}\times c_{i}$.
     So the excision theorem,
     Theorem~\ref{thm:Floer-excision-genus-1}, provides an isomorphism
     \[
                    \HM_{\bullet}(Y|\bar{R};\Gamma_{\eta}) \to
                    \HM_{\bullet}(\tilde{Y} | \tilde{R};
                    \Gamma_{\tilde \eta}),
     \]
     and hence and isomorphism
     \[
\HM_{\bullet}(Y|\bar{R};\cR) \to
                    \HM_{\bullet}(\tilde{Y} | \tilde{R};
                    \cR).     \]
     Since $\cR$ is a field and $Y$ is a disjoint union, the left-hand
     side is a tensor product, and the proposition follows.
\end{proof}

Next we prove a version of Ni's ``horizontal decomposition'' formula.
A \emph{horizontal surface} in $(M,\gamma)$ is a surface $S$ with
$\chi(S) = \chi(R_{+}(\gamma))$ such that $\partial S$ consists of one
circle in each of the annuli comprising $A(\gamma)$; it is required to
represent the same relative homology class as $R_{\pm}(\gamma)$ in
$H_{2}(M,A(\gamma))$ and should have
$[\partial S] = [s(\gamma)]$ in $H_{1}(A(\gamma))$. Cutting along a
horizontal surface creates a new sutured manifold
\[
            (M',\gamma') = (M_{1},\gamma_{1}) \cup (M_{2},
            \gamma_{2}).
\]

\begin{proposition}[{\cite[Proposition
4.1]{Ni-A}}]\label{prop:horizontal}
     If $(M',\gamma')$ is obtained from $(M,\gamma)$ by cutting along
     a horizontal surface, then
     \[
                  \SHM(M,\gamma;\Q) = \SHM(M',\gamma';\Q).
     \]
\end{proposition}

\begin{proof}
    This follows directly from Theorem~\ref{thm:Floer-excision} and
    Proposition~\ref{prop:disjoint-union}.
\end{proof}

We shall also need to decompose sutured manifolds by cutting along
vertical surfaces. We prove a result along the lines of  \cite{Ni-A}
and \cite{Juhasz-2}. A \emph{product annulus} in $(M,\gamma)$ is an
embedded annulus $A=[-1,1]\times d$ in $(M,\gamma)$ such that the
circle $d_{+}=\{1\}\times d$ lies in the interior of $R_{+}(\gamma)$
and $d_{-}=\{-1\}\times d$ lies in the interior of $R_{-}(\gamma)$. 

\begin{proposition}\label{prop:vertical}
    Let $(M',\gamma')$ be obtained from $(M,\gamma)$ by cutting along
    a product annulus $A$. Then
    \[
                    \SHM(M,\gamma;\Q) = \SHM(M',\gamma';\Q)
    \]
    if we are in either of the following two situations:
    \begin{enumerate}
        \item \label{item:non-zero-d} the curves $d_{+}$ and $d_{-}$ represent non-zero
        classes in the first homology of $R_{+}(\gamma)$ and $R_{-}(\gamma)$
        respectively; or
        
        \item \label{item:zero-d} the curves $d_{+}$ and $d_{-}$ represent the zero
        class in $H_{1}(R_{+}(\gamma))$ and $H_{1}(R_{-}(\gamma))$
        respectively, at least one of them does not bound a disk,
        and the annulus $A$ separates $M$ into two parts, $M_{1}\cup
        M_{2}$, one of which is disjoint from the annuli $A(\gamma)$.
    \end{enumerate}
\end{proposition}

\begin{proof}
    We begin with case \ref{item:non-zero-d} of the proposition.
    We shall construct
    closures $(Y,\bar{R})$ and $(\tilde{Y},\tilde{R})$ for
    $(M,\gamma)$ and $(M',\gamma')$ which are related to each other as
    described in the excision theorem,
    Theorem~\ref{thm:Floer-excision-genus-1}, and the result will
    follow.

    When we attach the product $[-1,1]\times T$ to $(M,\gamma)$, the
    curves $d_{+}$ and $d_{-}$ remain non-separating in the closed
    surfaces $\bar{R}_{\pm}$, because $T$ is connected. By taking $T$
    to have non-zero genus, we can also ensure that there is a curve
    $c$ in the interior of $T$ which is non-separating in $T$. So
    after attaching the product region, we have two product annuli
    $[-1,1]\times d$ and $[-1,1]\times c$, with independent
    non-separating curves $d_{+},c_{+}$ in $\bar{R}_{+}$ in
    $d_{-},c_{-}$ in $\bar{R}_{-}$. We can close up the manifold using
    a diffeomorphism $h:\bar{R}_{+}\to\bar{R}_{-}$ such that
    $h(d_{+})=d_{-}$ and $h(c_{+})=c_{-}$. The closure $(Y,\bar{R})$
    of $(M,\gamma)$ that we arrive at in this way contains two tori,
\[
    \begin{aligned}
        \Sigma_{1}&= S^{1}\times c \\
        \Sigma_{2}&= S^{1}\times d 
    \end{aligned}
\]
    There is a $1$-cycle $\eta$ lying on $\bar{R}$ that is transverse
    to both of these tori, so Theorem~\ref{thm:Floer-excision-genus-1}
    is applicable. (This is an instance of that theorem where the
    manifold $Y'$ obtained by cutting along $\Sigma_{1}$ and
    $\Sigma_{2}$ is connected.) The manifold $(\tilde{Y},\tilde{R})$
    obtained from $(Y,\bar{R})$ by cutting along $\Sigma_{1}\cup
    \Sigma_{2}$ and regluing is a closure of the $(M',\gamma')$, so
    we are done with case \ref{item:non-zero-d}.

    We turn to case \ref{item:zero-d}. Without loss of generality, we
    suppose that $M_{1}$ does not meet $A(\gamma)$ and $d_{+}$ does
    not bound a disk. Let $R_{+,1}$
    denote $R_{+}(\gamma)\cap M_{1}$ and let $R_{-,2}$ denote
    $R_{-}(\gamma)\cap M_{2}$. The surface $R_{+,1}$ has genus at
    least $1$ and its only boundary component is $d_{+}$. In
    \cite{Ni-A}, Ni uses the following observation. The union
    \[
                    R_{+,1}\cup A \cup R_{-,2}
    \]
    is isotopic to a horizontal surface in $(M,\gamma)$ to which
    Proposition~\ref{prop:horizontal} applies. By cutting along this
    horizontal surface, the pieces we get from $(M,\gamma)$ are (up to
    diffeomorphism)
    \[
                      \bigl(  [-1,1]\times R_{+,1}\bigr)
                      \cup_{[-1,1]\times d} M_{2}
    \]
    and
    \[
                    M_{1} \cup_{[-1,1]\times d} \bigl([-1,1] \times
                    R_{-,2}\bigr).
    \]
    In this way, case \ref{item:zero-d} is reduced to the case that
    either $M_{1}$ or $M_{2}$ is a product.

    If $M_{2}$ is a product, $[-1,1]\times R_{-,2}$, then the result
    is entirely straightforward: the surface $R_{-,2}$ contains all
    the annuli $A(\gamma)$. A closure $Y$ of $(M,\gamma)$ using an
    auxiliary surface $T$ can also be regarded as a closure of
    $(M_{1},\gamma_{1})$ using the auxiliary surface $R_{-,2}\cup T$.
    So we have
    \[
                        \SHM(M,\gamma) = \SHM(M_{1},\gamma_{1}).
    \]
    On the other hand, because $M_{2}$ is a product, we have
    $\SHM(M_{1},\gamma_{1}) = \SHM(M',\gamma')$ by
    Proposition~\ref{prop:disjoint-union}. Finally, if $M_{1}$ is a
    product, then we can cut $M_{1}$ open along a non-separating
    annulus because $R_{+,1}$ has positive genus, and this does not
    change $\SHM$, by part \ref{item:non-zero-d} of the proposition.
    After cutting open $M_{1}$ in this way, we arrive at a situation
    in which proposition \ref{item:non-zero-d} applies again, and the
    proof is complete.
\end{proof}

\subsection{Proof of Theorem~\ref{thm:fibered-M}}
\label{subsec:proof-fibered}

Those ingredients of Ni's proof from \cite{Ni-A} which involve
Heegaard Floer homology have all been replicated here in the context
of monopole Floer homology, so the proof carries through with little
change. We outline the argument, adapted from \cite{Ni-A}.
Let $(M,\gamma)$ be a balanced sutured manifold satisfying the
hypotheses of the theorem, and suppose $(M,\gamma)$ is not a product
sutured manifold. We shall show that $\SHM(M,\gamma)$ has rank at
least $2$.

Because of Proposition~\ref{prop:disjoint-union}, it is sufficient to
treat the case that $M$ is connected. Similarly, because of
Proposition~\ref{prop:horizontal}, we may assume that $(M,\gamma)$ is
``vertically prime'': that is, every horizontal surface in $(M,\gamma)$
is a parallel copy of either $R_{+}(\gamma)$ or $R_{-}(\gamma)$.  By
attaching product regions to $(M,\gamma)$ and appealing to
Proposition~\ref{prop:vertical}, we are also free to suppose that
$(M,\gamma)$ has only one suture. We now consider a maximal product
pair $i : [-1,1]\times E\hookrightarrow (M,\gamma)$ as in
\cite{Ni-A,Ni-Corr} and the induced map
\[
            i_{*} : H_{1}([-1,1]\times E) \to H_{1}(M).
\]
There are two cases.

\subparagraph{Case 1: $i_{*}$ is not surjective.} In this case, Ni
establishes that $(M,\gamma)$ admits two taut foliations $\cF_{1}$ and
$\cF_{2}$ whose difference element is non-torsion in $H^{2}(M,\partial
M)$. It then follows from Corollary~\ref{cor:rank-2} that
$\SHM(M,\gamma)$ has rank $2$ or more, as required.

\subparagraph{Case 2: $i_{*}$ is surjective.} In this case, let
$(M',\gamma')$ be the complement of the maximal product pair.
This is non-empty, because $(M,\gamma)$ is not a product sutured
manifold.
Proposition~\ref{prop:vertical} tells us that $\SHM(M,\gamma)$ and
$\SHM(M',\gamma')$ have the same rank. Ni observes that the
vertically-prime condition on $(M,\gamma)$ implies that $M'$ is
connected. Furthermore, $(M',\gamma')$ is a homology product, and its
top and bottom surfaces $R_{\pm}(\gamma')$ are planar, because of the
surjectivity of $i_{*}$. The
(connected) surfaces $R_{\pm}(\gamma')$ are not disks, so
$(M',\gamma')$ has at least two sutures. Let $r \ge 2$ be the number of
sutures in $(M',\gamma')$. Let $S$ be a planar surface with $r+1$
boundary components, so that the product sutured manifold
$[-1,1]\times S$ has $r+1$ sutures. Form a new sutured manifold
$(\tilde M,\tilde\gamma)$ by gluing $r$ of the annuli from
$[-1,1]\times S$ to the annuli of $(M',\gamma')$. The resulting
sutured manifold $(\tilde M, \tilde \gamma)$ has
\[
            \rank\SHM(\tilde M,\tilde \gamma) = \rank\SHM(M',\gamma')
\]
by Proposition~\ref{prop:vertical}. Furthermore $(\tilde M,
\tilde\gamma)$ is a homology product, and its maximal product
pair is $[-1,1]\times S$ up to isotopy. The construction has been
made so that the inclusion of the maximal product pair in
$(\tilde M,\tilde\gamma)$ is not surjective on $H_{1}$, so we now
have a situation which falls into Case 1 above. It follows that
$\SHM(\tilde M,\tilde \gamma)$ has rank at least $2$; and so too
therefore does $\SHM(M,\gamma)$. This completes Ni's proof.

\subsection{More decomposition theorems}
\label{subsec:more-decomp}

In \cite{Juhasz-2}, rather general sutured manifold decompositions are
considered, and results of the following sort are obtained. Let
$(M,\gamma)$ be a balanced sutured manifold, and let $S\subset M$ be a
decomposing surface in the sense of \cite{Gabai}. There is a sutured
manifold decomposition,
\[
            (M,\gamma) \decomp{S} (M',\gamma'),
\]
and we shall suppose that $(M',\gamma')$ is also balanced (which
implies that $S$ has no closed components). Under some mild
restrictions on $S$, Juh\'asz proves in \cite{Juhasz-2} that $\SFH(M',
\gamma')$ is a direct summand of $\SFH(M,\gamma)$. An entirely similar
theorem can be proved in the context of monopole Floer homology, using
$\SHM(M,\gamma)$ in place of $\SFH(M,\gamma)$. The following is
a restatement of Theorem~1.3 of \cite{Juhasz-2}, though
with less specific information about the \spinc{} structures that are
involved behind the scenes. In the statement of the theorem, an
oriented simple
closed curve
$C$ in $R(\gamma)$ is called \emph{boundary coherent} if it either
represents a non-zero class in $H_{1}(R(\gamma))$ or it is the
oriented boundary $\partial R_{1}$ of a compact subsurface $R_{1}
\subset     R(\gamma)$ with its canonical orientation.

\begin{theorem}[{\cite[Theorem 1.3]{Juhasz-2}}]
\label{thm:Juhasz-decomp}
    Let $(M,\gamma)$ be a balanced sutured manifold and
    \[(M,\gamma)\decomp{S} (M', \gamma')\] a sutured manifold
    decomposition. Suppose that the decomposing surface $S$ has no
    closed components, and that for every component $V$ of
    $R(\gamma)$, the set of closed components of $S\cap V$ consists of
    parallel oriented boundary-coherent simple closed curves. Then the
    Heegaard Floer homology
    $\SFH(M',\gamma')$ is a direct summand of $\SFH(M,\gamma)$. \qed
\end{theorem}

We have the following result.

\begin{proposition}\label{prop:S-decomp}
    Theorem~\ref{thm:Juhasz-decomp} continues to hold with monopole
    Floer homology in place of Heegaard Floer homology. That is, with
    the same hypotheses, $\SHM(M',\gamma')$ is a direct summand of
    $\SHM(M,\gamma)$.
\end{proposition}

\begin{proof}
    By Lemma~4.5 of \cite{Juhasz-2}, Juh\'asz reduces this to the
    special case of a  ``good'' decomposing surface $S$, by which is
    meant a surface $S$ such that every component of $\partial S$
    intersects both $R_{+}(\gamma)$ and $R_{-}(\gamma)$.

    \begin{figure}
    \begin{center}
        \includegraphics[scale=0.7]{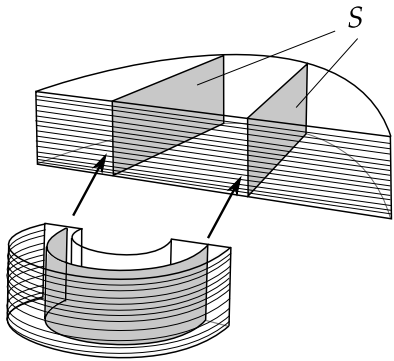}
    \end{center}
    \caption{\label{fig:Add-product-handles}
    Adding product $1$-handles to a sutured manifold containing a
    decomposing surface $S$.}
\end{figure}

    Starting from a good decomposing surface $S$, we can pass to
    another special case as follows. Let $C$ be a component of
    $\partial S$. By the definition of a good
    decomposing surface, $C$ intersects the annuli $A(\gamma)$ in
    vertical arcs. The number of these arcs counted with sign is zero.
    Pair up these arcs accordingly; and for each pair attach a product
    $1$-handle as shown in Figure~\ref{fig:Add-product-handles}.
    Repeat this with every other boundary component of $\partial S$.    
    The result of this process is a new balanced sutured manifold
    $(M_{1},\gamma_{1})$ containing a new decomposing surface $S_{1}$.
    We have $\SHM(M_{1},\gamma_{1})\cong \SHM(M,\gamma)$, because
    adding product handles has no effect.
    (The inverse operation to adding a product handle can also be
    described as removing a larger product region, by cutting along annuli
    parallel to the annuli where the handle is attached; so this
    operation is a  special case of one we have seen before.)
    Furthermore, if
    $(M'_{1},\gamma'_{1})$ is what we obtain from
    $(M_{1},\gamma_{1})$ by sutured manifold decomposition along
    $S_{1}$, then $(M'_{1},\gamma'_{1})$ is also related to
    $(M',\gamma')$ by adding adding product $1$-handles. It therefore
    suffices to prove that $\SHM(M'_{1},\gamma'_{1})$ is a direct
    summand of $\SHM(M_{1},\gamma_{1})$.

    Looking at $(M_{1},\gamma_{1})$, we now see that it is sufficient
    to prove the following lemma, which is  a priori a special case of
    the proposition.
\end{proof}

\begin{lemma}\label{lem:S-decomp-special}
    Let $(M,\gamma)$ be a balanced sutured manifold and let
    \[
            (M,\gamma) \decomp{S} (M',\gamma')
    \]
    be a sutured manifold decomposition. Suppose that $S$ has no
    closed components and that
    the oriented boundary of
    $\partial S$ consists of $n$ simple closed curves
    $C^{+}_{1},\dots, C^{+}_{n}$ in $R_{+}(\gamma)$ and $n$ simple
    closed curves $C^{-}_{1},\dots, C^{-}_{n}$ in $R_{-}(\gamma)$.
    Suppose further that the homology classes of
    $C^{+}_{1},\dots,C^{+}_{n}$ are  a collection of independent
    classes in
    $H_{1}(R_{+}(\gamma))$, and make a similar assumption for
    $R_{-}(\gamma)$. Then $\SHM(M',\gamma')$ is a direct summand of
    $\SHM(M,\gamma)$.
\end{lemma}

\begin{proof}[Proof of the lemma]
    Form the closure $Y=Y(M,\gamma)$ by attaching a product region
    $[-1,1]\times T$ as usual and then choosing the diffeomorphism $h$
    in such a way that $h(C^{+}_{i}) = C^{-}_{i}$ (with the opposite
    orientation) for all $i$. The result of this is that $Y$ contains
    two closed surfaces: first the usual surface $\bar{R}$, and second
    a surface $\bar{S}$ obtained from $S$ by identifying $C^{+}_{i}$
    with $C^{-}_{i}$ for all $i$. The intersection
    $\bar{S}\cap\bar{R}$ consists of $n$ circles, $C_{1},\dots,
    C_{n}$. Let $F$ be the oriented surface obtained from $\bar{S}\cup\bar{R}$
    by smoothing out the circles of double points, respecting
    orientations.

    The same surface $F\subset Y$ can be arrived at from a different
    direction. Start with $(M',\gamma')$. We can write $A(\gamma')$ as
    a union of components
    \[
                A(\gamma') = A(\gamma) \cup A_{1},
    \]
    where $A(\gamma)$ are the annuli of the original sutured manifold
    $(M,\gamma)$ and $A_{1}$ are the new annuli. The new annuli can be
    written as $[-1,1]\times D^{\pm}_{i}$, where the collection of curves
    $D^{\pm}_{i}$ are in natural correspondence with the curves
    $C^{\pm}_{i}$. We now form a closure $Y'$ of $(M',\gamma')$ as
    follows. We attach a product region $[-1,1]\times T'$ to
    $(M',\gamma')$, where $T'$ is a (disconnected) surface
    \[
                T' = T \cup T_{1}.
    \]
    Here $T$ is the surface used to close $Y$ and $T_{1}$ is a
    collection of $n$ annuli
    \[
                T_{1} = T_{1,1} \cup \dots \cup T_{1,n}.
    \]
    Although $T'$ breaks the rules by being disconnected, we can still
    effectively use $T'$ in constructing $\SHM(M',\gamma')$ because of
    the arguments of section~\ref{subsec:decomp}. In attaching
    $[-1,1]\times T'$ to $(M',\gamma')$ we glue $[-1,1]\times\partial
    T$ to the annuli $A(\gamma)\subset A(\gamma')$ as we did when
    closing $(M,\gamma)$, and we glue the two components $[-1,1]\times
    T_{1,i}$ to the two annuli $[-1,1]\times D^{\pm}_{i}$ belonging to
    $A_{1}$.
\begin{figure}
    \begin{center}
        \includegraphics[scale=0.7]{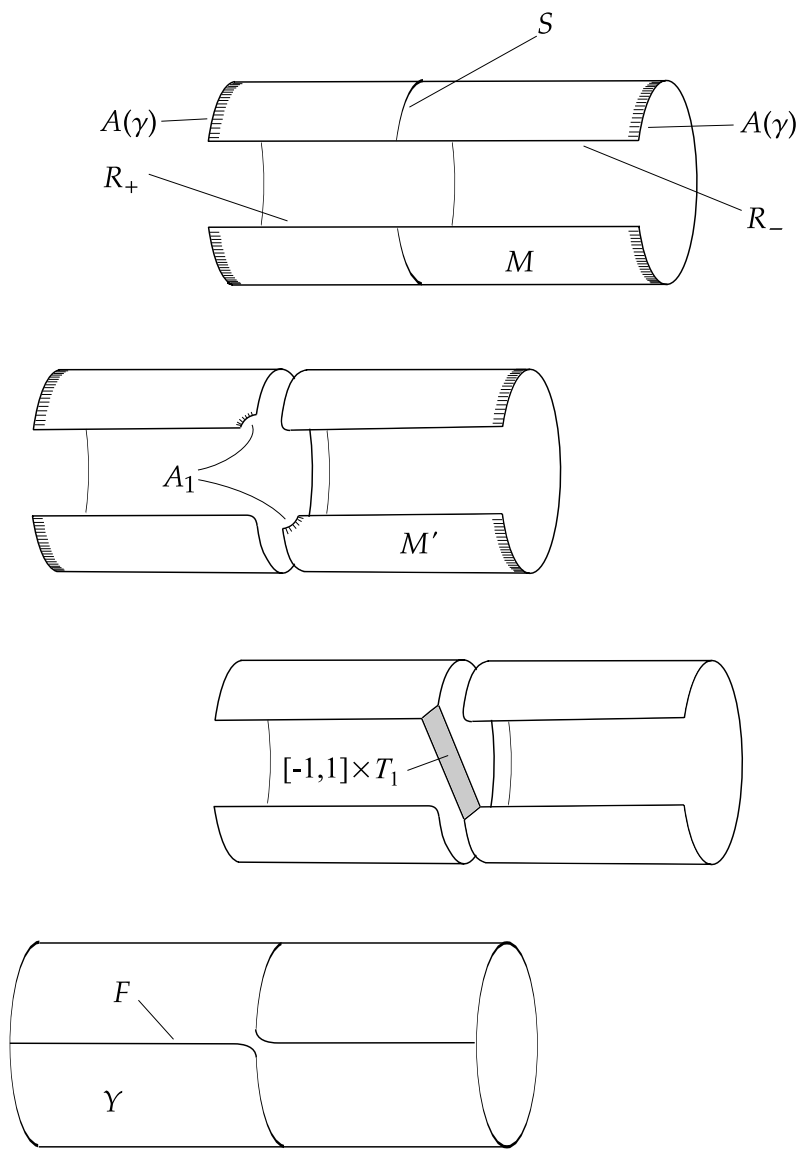}
    \end{center}
    \caption{\label{fig:Sewn-up}
    Decomposing $M$ along $S$ and then closing up to get $F$.
     The collars of $\partial M$ and $\partial
    M'$ are marked with hatching near $A(\gamma)$ and $A(\gamma')$. The
    product part $[-1,1]\times T$ is not shown in the figure, which is
    otherwise a faithful representation after multiplying by $S^{1}$.}
\end{figure}

    At this point, we have a manifold
    \[
                (M',\gamma') \cup [-1,1]\times (T\cup T_{1})
    \]
    with two boundary components $\bar{R}'_{\pm}$. The top surface
    $\bar{R'}_{+}$ can be described as a union
    \[
                \bar{R}'_{+} = \bar{R}^{\dag}_{+} \cup S_{+} \cup \{+1\}\times
                T_{1}.
    \]
    Here $\bar{R}^{\dag}_{+}$ is the surface with boundary obtained by
    cutting open $\bar{R}_{+}$ along the circles $C^{+}_{i}$, and the
    annuli $\{+1\}\times T_{1}$ are collars of half of the boundary components
    of $\bar{R}^{\dag}_{+}$. The surface $S_{+}$ is a copy of $S$.
    Up to diffeomorphism, we can forget these
    annular regions and write
    \[
\begin{aligned}
                \bar{R}'_{+} &= \bar{R}^{\dag}_{+} \cup S_{+} \\
                \bar{R}'_{-} &= \bar{R}^{\dag}_{-} \cup S_{-} .
\end{aligned}
    \]
    That is, $\bar{R}'_{+}$ is obtained from $\bar{R}_{+}$ by cutting
    open along the circles $C^{+}_{i}$ and inserting a copy of $S$.
    Finally, form the closure $Y'$ by using a diffeomorphism
    \[
                h' : \bar{R}'_{+}  \to \bar{R}'_{-} 
    \]
    which is equal to $h$ on $\bar{R}^{\dag}_{+}$ and equal to the
    identity on $S$.

    The resulting closure $Y'$ of $(M',\gamma')$ is diffeomorphic to
    $Y$; and under this diffeomorphism, the surface $\bar{R}'\subset
    Y'$ obtained from $\bar{R}'_{\pm}$ becomes the surface $F$. (See
    Figure~\ref{fig:Sewn-up}.)
    It
    follows that we can calculate $\SHM(M',\gamma')$ as
    \[
                \SHM(M',\gamma') = \HM_{\bullet}(Y|F).
    \]
    The homology class of $F$ is the sum of the classes of $\bar{R}$
    and $\bar{S}$. Furthermore, $\chi(F) = \chi(\bar{R}) +
    \chi(\bar{S})$. It follows from the adjunction inequality that the
    only \spinc{} structures in $\Sp(Y|F)$ which can have non-zero
    Floer homology are those in the intersection $\Sp(Y|\bar{R}) \cap
    \Sp(Y|\bar{S})$. So we have
    \[
                    \SHM(M',\gamma') = \bigoplus_{\s\in \Sp(\bar{R})
                    \cap \Sp(\bar{S})} \mskip -20 mu \HM_{\bullet}(Y,\s),
    \]
    while
    \[
                    \SHM(M,\gamma) = \bigoplus_{\s\in \Sp(\bar{R})}
                    \HM_{\bullet}(Y,\s) .
    \]
    This shows that $\SHM(M',\gamma')$ is a direct summand of
    $\SHM(M,\gamma)$, as the lemma asserts.
\end{proof}

As is pointed out in \cite{Juhasz-2}, one can use 
Proposition~\ref{prop:S-decomp} to give an alternative proof of the
non-vanishing of $\SHM(M,\gamma)$ when $(M,\gamma)$ is taut. One uses
a sutured manifold hierarchy, starting at $(M,\gamma)$ and ending at a
product sutured manifold, whose (monopole) Floer homology we know to
be $\Z$, so showing that $\Z$ is a summand of $\SHM(M,\gamma)$. 

\section{Instantons}
\label{sec:instantons}

Much of the contents of this paper can be adapted to the case of
(Yang-Mills) instanton homology, instead of (Seiberg-Witten) monopole
Floer homology. We present some of this material in this section. For
background on instanton homology, we refer to \cite{Donaldson-book}.

\subsection{Instanton Floer homology}

When looking at the monopole Floer homology groups $\HM_{\bullet}(Y,\s)$
of a $3$-manifold $Y$, we could avoid difficulties arising from
reducible solutions by considering situations where only non-torsion
\spinc{} structures $\s$ played a role. In instanton homology,
reducibles can be avoided by using $\SO(3)$ bundles with non-zero
$w_{2}$. We proceed as follows.

Fix a hermitian line bundle $w\to Y$
such that $c_{1}(w)$ has odd pairing with some integer homology
class. Let $E \to Y$ be a $\U(2)$ bundle with an isomorphism
$\theta : \Lambda^{2}E \to w$. Let $\cA$ be the space of $\SO(3)$
connections in $\ad(E)$ and let $\G$ be the group of determinant-1
gauge transformations of $E$ (the automorphisms of $E$ that 
respect $\theta$). The Chern-Simons functional on the space $\cB =
\cA/\G$ leads to a well-defined instanton homology  group which we
write as $I_{*}(Y)_{w}$ \cite{Donaldson-book}. It is also possible to
use a slightly larger gauge group than $\G$. Fix a surface $R \subset Y$
that has odd pairing with $c_{1}(w)$. Let $\xi = \xi_{R}$ be a real line
bundle with $w_{1}(\xi)$ dual to $R$. The map $E\mapsto E\otimes \xi$
gives rise to a map on the space of connections, \[
\iota_{R}:\cB\to\cB, \] without
fixed points, and there is a quotient $\cB/\iota_{R}$. This is the same
as the quotient of $\cA$ by a gauge group which has $\G$ as an index-2
subgroup. Let us temporarily write $I_{*}(Y)_{w,R}$ for the resulting
instanton homology group: it is the fixed space of an induced
involution on $I_{*}(Y)$. As an example, in the case $Y=T^{3}$, we
have $I_{*}(T^{3})_{w} = \Z\oplus \Z$. The involution interchanges the two
copies of $\Z$, and $I_{*}(T^{3})_{w,R}=\Z$ whenever $w\cdot [R]$ is
non-zero. In general, $I_{*}(Y)_{w}$ is $(\Z/8_{}$-graded. The involution
acts with degree $4$, and the group $I_{*}(Y)_{w,R}$ is $(Z/4)$-graded.

Although these groups are defined with $\Z$ coefficients, it will be
convenient to work with a field of characteristic zero; and in what
follows we will take that field to be $\C$. Thus we will take it that
\[
                I_{*}(T^{3})_{w,R}=\C.
\]

\subsection{The eigenspace decomposition}

The monopole Floer homology detects the Thurston norm of a
$3$-manifold (see section~\ref{subsec:adjunction}); but the formulation
of this statement requires the decomposition of the monopole Floer
homology according to the different \spinc{} structures. In order to
relate \emph{instanton} homology to the Thurston norm, one needs a
decomposition of the instanton homology. As suggested in
\cite{KM-asymptotics}, such a decomposition arises from the
eigenspaces of natural operators on the Floer groups.

Let $Y$ be again a closed $3$-manifold and $w$ a line bundle as above.
Given an oriented closed surface $R$ in $Y$, there is a
$2$-dimensional cohomology class $\mu(R)$ in $\cB$ (for which our
conventions follow \cite{Donaldson-Kronheimer}) and hence an operation
of degree $-2$ on both $I_{*}(Y)_{w}$ and $I_{*}(Y)_{w,R}$. There is
also the class $\mu(y)$, for $y$ a point in $y$, which acts with
degree $4$. The operators $\mu(R)$ and $\mu(y)$ commute, so one can
look for simultaneous eigenvalues. In the
special case that $Y=S^{1}\times \Sigma$, with $\Sigma$ a surface of
positive genus, the eigenvalues of $\mu(\Sigma)$ and $\mu(y)$ were computed by
Mu\~noz in 
\cite{Munoz}:

\begin{proposition}[{\cite[Proposition 20]{Munoz}}]
\label{prop:eigen-munoz}
    Let $w\to S^{1}\times \Sigma$ be the line bundle whose first Chern
    class is dual to the $S^{1}$ factor. Then
    the simultaneous eigenvalues of the action of  $\mu(\Sigma)$ and
    $\mu(y)$ on
    $I_{*}(S^{1}\times\Sigma)_{w}$ are the pairs of complex numbers
    \[
               ( i^{r}(2k), (-1)^{r}2 )
    \]
    for all the integers $k$ in the range $0 \le k \le g-1$ and all
    $r=0,1,2,3$. Here $i$ denotes
    $\sqrt{-1}$. 
    \qed
\end{proposition}

\begin{remark}
    In \cite{Munoz}, the $2$-dimensional class called $\alpha$
    corresponds to $2\mu(\Sigma)$ here, and the class $\beta$
    corresponds to $-4\mu(y)$. Also, the group
    $\mathit{HF}^{*}(S^{1}\times\Sigma)$ that appears in \cite{Munoz} is our
    $I_{*}(S^{1}\times\Sigma)_{w,\Sigma}$. Munoz computes the spectrum
    in the case of
    $I_{*}(S^{1}\times \Sigma)_{w,\Sigma}$, but the case of
    $I_{*}(S^{1}\times \Sigma)_{w}$ follows in a straightforward
    manner. Observe, in particular, that because $\mu(\Sigma)$ is an
    operator of degree $2$ on a $(\Z/8)$-graded vector space, the
    eigenspaces of eigenvalues $\lambda$ and $i\lambda$ will always be
    isomorphic.
\end{remark}

As a corollary of this proposition, a similar result holds for a
general $3$-manifold $Y$.

\begin{corollary}\label{cor:eigen-munoz-Y}
    Let $R\subset Y$ be closed connected surface of positive genus,
    and let $w$ have odd pairing with $R$. Then the eigenvalues of the
    action of the pair of operators $\mu(R)$ and $\mu(y)$
    on $I_{*}(Y)_{w}$ are a subset of the
    eigenvalues that occur in the case of the product manifold
    $S^{1}\times R$. That is, they are pairs complex numbers
    \[
                ( i^{r}(2k), (-1)^{r }2)
    \]
    for integers $k$ in the range $0\le k \le g-1$.
\end{corollary}

\begin{proof}
    Let $R_{0}$ be a copy of $R$ in the interior of the product
    cobordism $W=[-,1]\times Y$. The action of $\mu(R)$ on
    $I_{*}(Y)_{w}$ can be regarded as being defined by this copy of
    $R$ in the 4-dimensional cobordism. Let $W'$ be the cobordism from
    the disjoint union $S^{1}\times R$ and $Y$ at the incoming end to
    $Y$ at the outgoing end, obtained by
    removing an open tubular neighborhood of $R_{0}$ from $W$. We have
    a map defined by $W'$,
    \[
                    \psi_{W'} : I_{*}(S^{1}\times R)_{w}\otimes
                    I_{*}(Y)_{w}\to I_{*}(Y)_{w}.
    \]
    The map is surjective, because one obtains the product cobordism
    by closing off the boundary component $S^{1}\times \Sigma$.
    Furthermore, because $R_{0}$ is homologous to surfaces in each of
    the three boundary components, we have, for example
    \[
                    \psi_{W'} ( \mu(R)a \otimes b)
                    = \mu(R) \psi_{W'} (a\otimes b).
    \]
    From this relation and the surjectivity of $\psi_{W'}$, it follows
    that the eigenvalues of $\mu(R)$ on the outgoing end $Y$ are a
    subset of the eigenvalues of the action of $\mu(R)$ on
    $S^{1}\times R$. We obtain the result of the corollary by applying
    a similar argument to $\mu(y)$ and to $\mu(R)^{2} +
    \mu(y)$.
\end{proof}

We can now give a definition in instanton homology of something that
will play the role that $\HM_{\bullet}(Y|R)$ played in the
monopole theory.

\begin{definition}
    Let $Y$ be a closed, oriented $3$-manifold, $w$  a 
    hermitian line bundle on $Y$ and $R\subset Y$ a closed, connected,
    oriented surface on which $c_{1}(w)$ is odd. Let $g$ be the genus
    or $R$, which we require to be positive. We define
    \[
                I_{*}(Y|R)_{w}
    \]
    to be the simultaneous eigenspace for the operators $\mu(R)$,
    $\mu(y)$ for the pair of eigenvalues $(2g-2, 2)$. \CloseDef
\end{definition}

\begin{remark}
    Except in the case that the genus is $1$, we could define this
    more simply as just the $(2g-2)$-eigenspace of $\mu(R)$, as can be
    seen from Corollary~\ref{cor:eigen-munoz-Y}
\end{remark}

Although Mu\~noz does not calculate the dimensions of the eigenspaces
in general for $S^{1}\times\Sigma$, one can readily read off from the
proof of \cite[Proposition~20]{Munoz} that the dimension of the eigenspace
belonging to the largest eigenvalue is $1$.  That is,

\begin{proposition}\label{prop:rank-one-w}
    Let $Y=S^{1}\times R$ with $\Sigma$ of positive
    genus, and let $w$ be the line bundle dual to the $S^{1}$ factor.
    Then
    \[
                I_{*}(Y|R)_{w} = \C.
    \]
    \qed
\end{proposition}

There is a simple extension of the above definition to the case that
$R$ has more than one component, as long as $w$ is odd on each
component. If the components are $R_{m}$, then the corresponding
operators $\mu(R_{m})$ commute, and we may take the appropriate
simultaneous eigenspace. In general, the action of $\mu(R)$ on
$I_{*}(Y)_{w}$ is not diagonalizable; but one can read off from
\cite{Munoz} that the eigenspace of $\mu(R)$ belonging to the top
eigenvalue $2g-2$ is simple when one restricts to the kernel of
$\mu(y)-2$. That is,
\[
                    \ker (\mu(y) - 2) \cap \ker (\mu(R) - (2g-2))^{N}
                    = \ker (\mu(y)-2) \cap \ker (\mu(R)- (2g-2))                    
\]
for all $N\ge 1$.

\begin{proposition}
    Given any $Y$, $R$ and $w$ for which $I_{*}(Y|R)_{w}$ is defined,
    and given
    any other surface $\Sigma\subset Y$ of positive genus, the action of
    $\mu(\Sigma)$ on $I_{*}(Y|R)_{w}$ has
    eigenvalues belong to the set of even integers in the range from
    $-(2g-2)$ to $2g-2$, where $g$ is the genus of $\Sigma$. 
\end{proposition}

\begin{proof}
    The action of $\mu(\Sigma)$ on $I_{*}(Y)_{w}$ commutes with
    $\mu(R)$, so the action of $\mu(\Sigma)$ does preserve the subspace
    $I_{*}(Y|R)_{w}\subset I_{*}(Y)_{w}$.
    If $w$ is odd on $\Sigma$, then the proposition follows from
    Corollary~\ref{cor:eigen-munoz-Y} together with the fact that
    $\mu(y)-2$ is zero on this subspace. If $w$ is even on
    $\Sigma$, then one can consider a surface in the homology class of
    $R + n \Sigma$ and use the additivity of $\mu$.
\end{proof}

Because the actions of $\mu(\Sigma_{1})$ and $\mu(\Sigma_{2})$ commute
for any pair of classes $\Sigma_{1}$ and $\Sigma_{2}$,
we have a decomposition of
$I_{*}(Y|R)_{w}$ by cohomology classes (as outlined in
\cite{KM-asymptotics}):

\begin{corollary}\label{cor:simultaneous-eigenspaces}
    There is a direct sum decomposition into generalized eigenspaces
    \[
                I_{*}(Y|R)_{w} = \bigoplus_{s} I_{*}(Y|R,s)_{w}
    \]
    where the sum is over all homomorphisms
    \[
                    s : H_{2}(Y;\Z) \to 2\Z
    \]
    subject to the constraints
    \[
                  \bigl | s([S]) \bigr | \le      2 \genus (S) - 2
    \]
    for all connected surfaces $S$ with positive genus and $s([R]) = 2\genus(R)-2$.
    The summand $I_{*}(Y|R,s)_{w}$ is the simultaneous generalized
    eigenspace
     \[
            I_{*}(Y|R,s)_{w} = \bigcap_{\sigma \in H_{2}(Y)} \bigcup_{N\ge
            0} \ker \Bigl(\mu(\sigma) - s(\sigma)\Bigr)^{N}.
     \]
    \qed
\end{corollary}

It will be convenient at a later point to have a notation for the sort
of homomorphisms $s$ that arise here. Choosing a notation reminiscent
of our notation for \spinc{} structures, we write
\[
            \Sh(Y) = \Hom \bigl (  H_{2}(Y) , 2\Z  \bigr)
\]
and for an embedded surface $R\subset Y$ of genus $g$ we write
\begin{equation}\label{eq:Sh-def}
            \Sh(Y|R) = \{ \, s\in \Sh(Y) \mid s([R]) = 2g-2 \,\}.
\end{equation}

\subsection{Excision for instanton homology}

Let $Y$ be closed, oriented $3$-manifold equipped with a line bundle
$w$, and suppose $\Sigma=\Sigma_{1} \cup \Sigma_{2}$ is an oriented embedded
surface with two 
connected components of equal genus, which we require to be positive.
Suppose also that
$c_{1}(w)[\Sigma_{1}]$ and $c_{1}(w)[\Sigma_{2}]$ are equal and odd.
We allow that $Y$ has either one or two components. In the latter
case, we require one of the $\Sigma_{i}$ to be in each component. In
the former case, when $Y$ is connected, we assume that $\Sigma_{1}$
and $\Sigma_{2}$ are not homologous. Choose a diffeomorphism
$h:\Sigma_{1}\to\Sigma_{2}$, and lift it to a bundle-isomorphism
$\hat{h}$ on the restrictions of the line bundle $w$. From this data,
we form $\tilde{Y}$ by cutting along the $\Sigma_{i}$ and gluing up
using $h$ as before. The lift $\hat{h}$ can be used to glue up the
bundle also, giving us a bundle $\tilde{w}\to \tilde Y$. As before, we
write $\tilde\Sigma = \tilde\Sigma_{1}\cup \tilde\Sigma_{2}$ for the
surfaces in $\tilde Y$.

\begin{theorem}
    If $(\tilde Y,\tilde \Sigma)$ is obtained from $(Y,\Sigma)$ as
    above, then there is an isomorphism
    \[
                   I_{*}(Y|\Sigma)_{w} \cong
                   I_{*}(\tilde{Y}|\tilde{\Sigma})_{\tilde w}.
    \]
    We interpret the left-hand side as a tensor product in the case
    that $Y$ has two components.
\end{theorem}

\begin{proof}
    In the case that $\Sigma$ has genus $1$, this result is due to
    Floer \cite{Floer-Durham-paper,Braam-Donaldson}. In Floer's
    statement of the result, $Y$ had two components, but the proof
    does not require  it. It should also be said that statement the of Floer's
    theorem  in \cite{Braam-Donaldson}
    involves $I_{*}(Y)_{w}$ rather than $I_{*}(Y|\Sigma)_{w}$,
    which leads to an extra factor of two in the dimensions when $Y$
    has two components.

    The case of genus $2$ or more is essentially the same, once one
    knows that $I_{*}(S^{1}\times \Sigma_{i}| \Sigma_{i})_{w}$ has
    rank $1$. 
\end{proof}

In the case of genus $1$, note that passing from $I_{*}(Y)_{w}$ to
$I_{*}(Y|\Sigma)_{w}$ can also be achieved by taking the $+2$
eigenspace of $\mu(y)$, for one point $y$ in each component of $Y$.

Here are two particular applications of the excision theorem. They are
both variants of Proposition~\ref{prop:rank-one-w}, but involve
different line bundles.

\begin{proposition}\label{prop:rank-one-uw}
    Let $Y$ be the product $S^{1}\times\Sigma$, with $\Sigma$ a
    surface of genus $1$ or more, and let $w$ again be the line bundle
    dual to the $S^{1}$ factor. Let $u \to Y$ be a line bundle whose
    first Chern class is dual to a curve $\gamma$ lying on
    $\{\mathrm{point}\}\times \Sigma$, and write the tensor product
    line bundle as $uw$. Then we have
    \[
                    I_{*}(Y|\Sigma)_{uw}=\C.
    \]    
\end{proposition}

\begin{proof}
    Write $B$ for the vector space $I_{*}(Y|\Sigma)_{uw}$ and $A$ for
    the vector space $I_{*}(Y|\Sigma)_{w}$. We apply the
    excision theorem in a setting where the incoming manifold is two
    copies of $Y$ with the line bundle $uw$ and the outgoing manifold
    is a single copy of $Y$ with the line bundle $u^{2}w$. The latter
    gives the same Floer homology as the for line bundle $w$, so we
    learn that
    \[
                B \otimes B \cong A.
    \]
    We already know that $A$ is one-dimensional, and it follows that
    $B$ is also one-dimensional.
\end{proof}

For the second application, we can dispense with $w$:

\begin{proposition}\label{prop:rank-one-u}
    In the situation of Proposition~\ref{prop:rank-one-uw}, 
    the eigenspace of the pair of operators
    $(\mu(\Sigma), \mu(y))$ on $I_{*}(Y)_{u}$ 
    for the eigenvalues $(2g-2,2)$  is also one-dimensional.
\end{proposition}

\begin{proof}
    We can see more generally, that for any $\lambda$ the eigenspace
    for $(\lambda,2)$ on $I_{*}(Y)_{u}$ is the same as the
    corresponding eigenspace in $I_{*}(Y)_{wu}$. For this one can
    apply the excision theorem as follows. Let $c$ be a closed curve
    on $\Sigma$ so that the torus $S^{1}\times c$ intersects $u$ once.
    Let $Y_{1}$ be $S^{1}\times T^{2}$, and let $u_{1}$, $w_{1}$ and
    $c_{1}$ be similar there to $u$, $w$ and $c$. Apply the excision
    theorem with incoming manifold $Y_{1}\cup Y$ with the line bundles
    $u_{1}w_{1}$ and $uw$ respectively, cutting along the tori
    $S^{1}\times c_{1}$ and $S^{1}\times c$. The outgoing manifold is
    diffeomorphic to $Y$, with the line bundle $uw^{2}$, which gives
    the same homology as $u$. The excision theorem gives an
    isomorphism between the $+2$ eigenspaces of $\mu(y)$, which we
    denote
    \[
               \phi:     I_{*}(Y)^{(2)}_{uw} \to I_{*}(Y)^{(2)}_{u}.
    \]
    The map that gives rise to the isomorphism in the excision theorem
    intertwines (in this instance) the maps $\mu(\Sigma)$ on the outgoing
    end with
    \[
            \mu(T^{2})\otimes 1  + 1\otimes \mu(\Sigma)
    \]
    on the incoming end. Since $\mu(T^{2})$ is zero on
    $I_{*}(S^{1}\times T^{2})_{u_{1}w_{1}}$, the map $\phi$ actually
    commutes with $\mu(\Sigma)$.
\end{proof}

\begin{remark}
    The Floer homology group $I_{*}(S^{1}\times \Sigma)_{u}$ is
    something that appears to be rather simpler than the more familiar
    $I_{*}(S^{1}\times \Sigma)_{w}$. In particular, excision shows
    that the it behaves ``multiplicatively'' in $g-1$. The
    representation variety that is involved here is easy to identify:
    the critical point set of the Chern-Simons functional is two
    copies of a torus $T^{2g-2}$. The involution interchanges
    $\iota_{\Sigma}$ interchanges the two copies. It seems likely that
    the Floer group $I_{*}(S^{1}\times\Sigma)_{w,\Sigma}$ can be
    identified with the homology of this torus.
\end{remark}

\subsection{Instanton Floer homology for sutured manifolds}

Let $(M,\gamma)$ be a balanced sutured manifold. Just as we did in the
monopole case, we attach a connected product sutured manifold $[-1,1]\times T$
to $(M,\gamma)$ to obtain a manifold $Y'$ with boundary
$\bar{R}_{+}\cup\bar{R}_{-}$, a pair of diffeomorphic connected closed
surfaces. As before, we require that there be a closed curve $c$ in
$T$ such that $\{-1\}\times c$ and $\{1\}\times c$ are both
non-separating in their respective boundary components. We also pick a
marked point, $t_{0}\in T$, which we did not need before. Now we glue
$\bar{R}_{+}$ to $\bar{R}_{-}$ by a diffeomorphism. We require that
$h(t_{0}) = t_{0}$, so that the resulting closed manifold
$Y=Y(M,\gamma)$ contains a standard circle running through $t_{0}$.
This circle intersects once the closed surface $\bar{R}$ obtained by
identifying $\bar{R}_{\pm}$. We no longer require that $\bar{R}$ has
genus $2$ or more: in the instanton case, genus $1$ will suffice.

\begin{definition}
    The instanton homology of the sutured manifold $(M,\gamma)$ is the
    vector space
    \[
            \SHI(M,\gamma) :=    I_{*}(Y|\bar{R})_{w},
    \]
    where $(Y,\bar{R})$ is obtained from $(M,\gamma)$ by closing as
    just described, and $w$ is the line bundle whose first Chern class
    is dual to the standard circle through $t_{0}$. \CloseDef
\end{definition}

\begin{remark}
As an example, it follows from Proposition~\ref{prop:rank-one-w} that the
instanton homology of a product sutured manifold is $\C$.
\end{remark}

The proof that $\SHI(M,\gamma)$ is independent of the choice of
genus for $T$ and the choice of diffeomorphism $h$ can be carried over
almost verbatim from the monopole case, using the excision theorem. It
is even somewhat easier to manage, because the case of genus 1 is no
longer special. When showing that $\SHI$ is independent of the choice
of genus, we used twisted coefficients
$\HM_{\bullet}(Y|\bar{R};\Gamma_{\eta})$ as an intermediate step in
the monopole case. The counterpart of twisted coefficients in the
proof for the instanton case is the introduction of the auxiliary
line bundle $u$ that appears in Propositions~\ref{prop:rank-one-uw}
and \ref{prop:rank-one-u} above. One applies excision along tori, following the
same scheme as shown in Figure~\ref{fig:Genus-add}, to increase the
genus by $1$. On the components $S^{1}\times S$, with $S$ of genus $2$
as shown, one should take the line bundle $u$, where $u$ is the line
bundle whose first Chern class is dual to the dotted curve $d'$. This
argument shows that
\[
        I_{*}(Y|\bar{R})_{uw} = I_{*}(\tilde
        Y|\tilde{R})_{\tilde{u}\tilde{w}}
\]
where $Y$ and $\tilde Y$ are closures of $(M,\gamma)$ obtained using
auxiliary surfaces $T$ and $\tilde T$ of genus $g$ and $g+1$. Another
application of excision (cutting along copies of $\bar{R}$ and using
Proposition~\ref{prop:rank-one-uw}) shows that
\[
\begin{aligned}
I_{*}(Y|\bar{R})_{uw} &\cong  I_{*}(Y|\bar{R})_{w} \\
                       &= \SHI^{g}(M,\gamma).
\end{aligned}
\]

\subsection{Decompositions of sutured manifolds and non-vanishing}

The proofs of the decomposition results of sections~\ref{subsec:decomp} and
\ref{subsec:more-decomp} carry over without change to the instanton
setting also. In particular, Proposition~\ref{prop:S-decomp} holds in
the instanton case:

\begin{proposition}
    \label{prop:S-decomp-instanton}
        Let $(M,\gamma)$ be a balanced sutured manifold and
    \[(M,\gamma)\decomp{S} (M', \gamma')\] a sutured manifold
    decomposition satisfying the hypotheses of
    Theorem~\ref{thm:Juhasz-decomp}. Then $\SHI(M',\gamma')$ is a
    direct summand of $\SHI(M,\gamma)$. 
\end{proposition}

\begin{proof}
    The proof is the same as the proof of
    Proposition~\ref{prop:S-decomp}; but at the last step in
    Lemma~\ref{lem:S-decomp-special}, instead of using the
    decomposition into \spinc{} structures, one uses the
    generalized-eigenspace
    decomposition of Corollary~\ref{cor:simultaneous-eigenspaces}.
\end{proof}

As shown in \cite{Juhasz-2} and mentioned above at the end of
section~\ref{subsec:more-decomp}, a result such as
Proposition~\ref{prop:S-decomp-instanton} gives a non-vanishing
theorem for the case of taut sutured manifolds. We therefore have:

\begin{theorem}
    \label{thm:SHI-non-vanishing}
    If the balanced sutured manifold $(M,\gamma)$ is taut, then
    $\SHI(M,\gamma)$ is non-zero. \qed
\end{theorem}

The only alternative route known to the authors for proving
a non-vanishing theorem for instanton homology is the strategy in
\cite{KM-Witten-P}, which draws on results from symplectic and contact
topology, as well as on the partial proof of Witten's conjecture
relating Donaldson invariants and Seiberg-Witten invariants of closed
$4$-manifolds  \cite{Feehan-Leness}. We shall return to
non-vanishing theorems for instanton homology in
section~\ref{subsec:I-non-vanish}.

\subsection{Floer's instanton homology for knots}

Just as we did for the monopole case in
section~\ref{sec:knot-homology}, we can take Juh\'asz's prescription
as a definition of knot homology. Let $K\subset Z$ be again  a knot in
a closed, oriented $3$-manifold. Let $(M,\gamma)$ be the sutured
manifold obtained by taking $M$ to be the knot complement $Z\setminus
N^{\circ}(K)$ and $s(\gamma)$ a pair of oppositely oriented meridians
on $\partial K$. In the instanton case, there is no need for $\bar{R}$
to have genus $2$ or  more, so we may use the closure
$Y_{1}(M,\gamma)$ described in Definition~\ref{def:Closure-1}. This is
the closure of $(M,\gamma)$ obtained using $[-1,1]\times T$, where $T$
is an annulus. It is also described in Lemma~\ref{lem:picture-F} as
obtained from $Z\setminus N^{\circ}(K)$ by attaching $F\times S^{1}$,
where $F$ has genus one: the gluing is done so that $\{p\}\times S^{1}$
is attached to a meridian of $K$.  We summarize the construction of this
instanton knot homology in the following definition. The definition is
not new: it is the same ``instanton homology for knots'' that Floer
defined in \cite{Floer-Durham-paper}. For the purposes of this paper,
we call it $\KHI(Z,K)$:

\begin{definition}\label{def:KHI}
    The instanton knot homology $\KHI(Z,K)$ of a knot $K$ in
    $Z$ is defined to be the instanton homology of the sutured manifold
    $(M,\gamma)$ above; or equivalently, the instanton homology
    group $I_{*}(Y_{1}|\bar{R})_{w}$. Here $Y_{1}$ is obtained from
    the knot complement by attaching $F\times S^{1}$ as described, the
    surface $\bar{R}$ is the torus $\alpha\times S^{1}$ as shown in
    Figure~\ref{fig:F-alone}, and $w$ is the line bundle with
    $c_{1}(w)$ dual to $\beta\times\{p\}\subset F\times S^{1}$.
    \CloseDef
\end{definition}

The only difference between this and Floer's original definition is
that we have used $I_{*}(Y_{1}|\bar{R})_{w}$ in place of
$I_{*}(Y_{1})_{w}$. Since $\bar{R}$ has genus one, the former group
can be characterized as the $+2$ eigenspace of $\mu(y)$ acting on the
latter group. The latter group is the sum of two subspaces of equal
dimension, the eigenspaces for the eigenvalues $2$ and $-2$.

For a classical knot $K$ in $S^{3}$, we shall simply write $\KHI(K)$
for the instanton knot homology.  To get a feel for what this
invariant is, let us examine the set of critical points of the
Chern-Simons functional on $\cB$, or in other words the space of flat
connections in the appropriate $\SO(3)$ bundle, modulo the
determinant-1 gauge transformations. To do this, we start by looking
at $F\times S^{1}$, where $F$ is the genus-1 surface with one boundary
component, and the line-bundle $w$ with $c_{1}(w)$ dual to
$\beta\times \{p\}$. The appropriate representation variety can also
be viewed as the space of flat $\SU(2)$ connections on the complement
of the curve $\beta\times \{p\}$ with the property that the holonomy
around a small circle linking $\beta\times\{p\}$ is the central
element $-1$. Consider such a flat connection $A$ and let $J_{1}$ and
$J_{2}$ be the holonomies of $A$ around respectively
the curves $\alpha\times\{q\}$
and $a \times S^{1}$ in $F\times S^{1}$, where $a$ is a point on
$\alpha\setminus\beta$. The torus $\alpha\times S^{1}$ intersects the
circle $\beta\times\{p\}$ once, so we have
\[
                [J_{1}, J_{2}] = -1
\]
in $\SU(2)$. Up to a gauge transformation, we must have
\[           
            J_{1}=
            \begin{pmatrix}
                0 & -1 \\
                1 & 0
            \end{pmatrix} ,           \qquad
             J_{2}=
            \begin{pmatrix}
                i & 0 \\
                0 & -i
            \end{pmatrix} .
\]
Let $J_{3}$ be the holonomy around $\beta'\times\{q\}$, where $\beta'$
is a parallel copy of $\beta$. The elements $J_{1}$ and $J_{3}$ must
commute, so
\[
                J_{3} = \begin{pmatrix}
                e^{i\theta} & 0 \\
                0 & e^{-i\theta}
            \end{pmatrix} 
\]
for some $\theta$ in $[0,2\pi)$. The angle $\theta$ is now determined
without ambiguity from the gauge-equivalence class of the connection
$A$; and the  matrices $J_{1}$, $J_{2}$ and $J_{3}$ determine $A$
entirely. We have proved:

\begin{lemma}
    The representation variety of flat $\SO(3)$ connections on $F\times
    S^{1}$ for the given $w$, modulo the determinant-1 gauge group,
    is diffeomorphic to a circle $S^{1}$, via $J_{3}$ as above. \qed
\end{lemma}

Let us examine the restriction of these representations to the
boundary of $F\times S^{1}$. On this torus $\partial F\times S^{1}$,
the flat connections can be regarded as $\SU(2)$ connections. The
holonomy around the $S^{1}$ factor is $J_{2}$, which we have described
above. The holonomy around the $\partial F$ factor is given by the
commutator
\[
                    [J_{3}, J_{1}] = \begin{pmatrix}
                e^{2i\theta} & 0 \\
                0 & e^{-2i\theta}
\end{pmatrix}
.
\]
So for the representation variety described in the lemma, the
restriction to the boundary is a two-to-one map whose image is the
space of connections having holonomy around the $S^{1}$ factor given
by
\[
            \bi = \begin{pmatrix}
                i & 0 \\
                0 & -i
            \end{pmatrix}.
\]
Finally, we can attach $F\times S^{1}$ to the knot complement
$S^{3}\setminus N^{\circ}(K)$, and we obtain the following description
of the representation variety.

\begin{lemma}
    Let $K\subset S^{3}$ be a knot and let $Y_{1}$ and $w$ be as
    described in Definition~\ref{def:KHI}. Then 
    the representation variety given by the critical points of the
    Chern-Simons functional in the corresponding space of connections
    $\cB$ can be identified with a double cover of the
    space
    \[
            \Rep(K,\bi) = \{ \, \rho: \pi_{1}(S^{3}
            \setminus K) \to \SU(2) \mid \rho(m) =
            \bi \,\},
    \]
    where  $m$ is a chosen meridian. \qed    
\end{lemma}

Note that $\Rep(K,\bi)$ is a space of homomorphisms, not a space of
conjugacy classes of homomorphisms. The centralizer  of $\bi$ (a circle
subgroup) still acts on $\Rep(K,\bi)$ by conjugation. There is always
exactly one point of $\Rep(K,\bi)$ which is fixed by the action of this circle, namely
the homomorphism $\rho$ which factors through the abelianization
$H_{1}(S^{3}\setminus K) = \Z$. All other orbits are irreducible: they have stabilizer
$\pm 1$, so they are circles. In a generic case, $\Rep(K,\bi)$
consists of one isolated point corresponding to the abelian
(reducible) representation, and finitely many circles, one for each
conjugacy class of irreducible representations. In such a case, the
representation variety described in the lemma above is a trivial
double-cover of $\Rep(K,\bi)$. It therefore has two isolated points
corresponding to the reducible,
and two circles for each irreducible conjugacy class.

Because it comprises only the $+2$ eigenspace of $\mu(y)$, the knot
Floer homology $\KHI(K)$ has just half the dimension of
$I_{*}(Y_{1})_{w}$ in Definition~\ref{def:KHI}. Heuristically, we can
think of each irreducible conjugacy class in $\Rep(K,\bi)$ as
contributing the homology of the circle, $H_{*}(S^{1};\C)$, to the
complex that computes $\KHI(K)$, while the reducible contributes a
single $\C$. In any event, if there are only $n$ conjugacy
classes of irreducibles and the corresponding circles of critical points are
non-degenerate in the Morse-Bott sense, then it will follow that the
dimension of $\KHI(K)$ is bounded above by $2n+1$.

For a knot $K\subset Z$ supplied with a Seifert surface $\Sigma$,
there is a decomposition of the instanton knot homology $\KHI(Z,K)$ as
\[
            \KHI(Z,K) = \bigoplus_{i=-\genus(\Sigma)}^{\genus(\Sigma)} \KHI(Z,K,[\Sigma],i).
\]
The definition is the same as in the monopole case
\eqref{eq:KHM-summands}, but uses the
generalized-eigenspace decomposition of
Corollary~\ref{cor:simultaneous-eigenspaces} in place of the
decomposition by \spinc{} structures. In particular, for a classical
knot $K\subset S^{3}$, we can write
\[
            \KHI(K) = \bigoplus_{i=-g}^{g}\KHI(K,i),
\]
where $g$ is the genus of the knot. Just as in the monopole case, the
top summand $\KHI(K,g)$ can be identified with the instanton Floer
homology $\SHM(M,\gamma)$, where $(M,\gamma)$ is the sutured manifold
obtained by cutting open the knot complement along a Seifert surface
of genus $g$. (See Proposition~\ref{prop:sutured-vs-knot}.) From the
non-vanishing theorem, Theorem~\ref{thm:SHI-non-vanishing}, we
therefore deduce a non-vanishing theorem for $\KHI$.

\begin{proposition}
    Let $K$ be a classical knot of genus $g$. Then the instanton knot
    homology group $\KHI(K,g)$ is non-zero. In particular, instanton
    knot homology detects the genus of a knot. \qed
\end{proposition}

This proposition provides an alternative proof for results from
\cite{KM-Witten-P} and \cite{KM-Dehn-SU2}. In particular, we have the
following corollary:

\begin{corollary}
    If $K\subset S^{3}$ is non-trivial knot, then there exists an
    irreducible homomorphism $\rho : \pi_{1}(S^{3}\setminus K)$ which
    maps a chosen meridian $m$ to the element $\bi\in \SU(2)$.
\end{corollary}

\begin{proof}
    If there is no such homomorphism, then $\Rep(K,\bi)$ consists only
    of the reducible, which is always non-degenerate. The critical
    point set in $\cB$ then consists of two irreducible critical
    points, so the rank of $I_{*}(Y_{1})_{w}$ is at most $2$, and the
    rank of $\KHI(K)$ is therefore at most $1$. This is inconsistent
    with non-vanishing of $\KHM(K,g)$, since $\KHM(K,g)$ is isomorphic
    to $\KHM(K,-g)$.
\end{proof}

\subsection{Instanton homology and fibered knots}

Instanton knot homology detects fibered knots, just as the other
versions do. We state and prove this here. We need, however, an extra
hypothesis on the Alexander polynomial. For Heegaard knot homology,
and also in the monopole case, we know the Alexander polynomial is
determined by the knot homology, and the extra hypothesis is not
needed. It seems likely that the same holds in the instanton case, but
we have not proved it.

We begin with a version of Theorem~\ref{thm:fibered-M} for the
instanton case. 

\begin{theorem}\label{thm:fibered-M-instanton}
        Suppose that the balanced sutured manifold $(M,\gamma)$ is
    taut and a homology
    product.
    Then $(M,\gamma)$ is a product sutured manifold if and only if
    $\SHI(M,\gamma)=\C$.
\end{theorem}

\begin{proof}
    Ni's argument, as presented for monopole knot homology in the
    proof of Theorem~\ref{thm:fibered-M}, works just as well for $\SHI$
    as it does for $\SHM$, with one slight change (a change which is
    in the spirit of \cite{Juhasz-2}). The key point occurs in
    Case 1 in the proof of Theorem~\ref{thm:fibered-M}
    (section~\ref{subsec:proof-fibered}), where it is already assumed
    that $(M,\gamma)$ has just one suture. We described this step using
    \spinc{} structures, but we can argue using homology instead.

    Let $N$ be obtained from $(M,\gamma)$ by adding a product region
    $[-1,1]\times T$ to the single suture. The boundary of $N$ is
    $\bar{R}_{+}\cup \bar{R}_{-}$. In
    \cite{Ni-A}, Ni shows that if $E\times I$ does not carry all the
    homology of $(M,\gamma)$, then one can find two decomposing
    surfaces $S_{1}$
    and $S_{2}$ in $N$ with the following properties. First,
    the boundaries of $S_{1}$ and $S_{2}$ are the same and consist of
    a pair of circles $\omega_{+}$ and $\omega_{-}$ which represent
    non-zero homology classes in $\bar{R}_{+}$ and $\bar{R}_{-}$.
    Second, the sutured manifolds $(M_{1}', \gamma_{1}')$ and
    $(M_{2}', \gamma_{2}')$ obtained by decomposition of $N$
    along $S_{1}$ and $S_{2}$ respectively are both taut. Third, if
    $Y$ is obtained from $N$ by gluing $\bar{R}_{+}$ to $\bar{R}_{-}$
    by a diffeomorphism $h$ with $h(\omega_{+}) = \omega_{-}$, then
    the resulting closed surface $\bar{S}_{1}$, $\bar{S}_{2}$ and
    $\bar{R}$ in $Y$ satisfy the following conditions,
    for some $m>0$ and some closed surface
    $\bar{S}_{0}$ with $\chi(\bar{S}_{0})$ non-zero,
    \begin{align*}
                [\bar{S}_{1}] &= m
                            [\bar{R}] + [\bar{S}_{0}] \\
                            [\bar{S}_{2}] &= m
                            [\bar{R}] - [\bar{S}_{0}] \\
    \intertext{ and }
                \chi (\bar{S}_{1}) &= \chi (\bar{S}_{2}) \\
                            &= m
                            \chi(\bar{R}) + \chi(\bar{S}_{0}) .
   \end{align*}
   These last conditions imply that $\Sh(Y|\bar{R}) \cap
   \Sh(Y|\bar{S}_{1})$ is disjoint from $\Sh(Y|\bar{R}) \cap
   \Sh(Y|\bar{S}_{2})$. (The notation $\Sh$ is introduced at
   \eqref{eq:Sh-def}.)

   For $i=1,2$, let $F_{i}$ be the surface in $Y$ obtained by
   smoothing out the intersection of $\bar{R}$ and $\bar{S}_{i}$ (a
   single circle in both cases). The proof of
   Lemma~\ref{lem:S-decomp-special} shows that
   \[
   \begin{aligned}
   \SHI(M'_{i},\gamma'_{i}) &=
                    \bigoplus_{\;\;\quad s\in \Sh(Y|F_{i})\qquad\;\;} I_{*}(Y|\bar{R},s)_{w} \\
                     &=
                    \bigoplus_{s\in \Sh(Y|\bar{R}) \cap
                    \Sh(Y|\bar{S}_{i})} I_{*}(Y|\bar{R},s)_{w}\\
                     &\subset \qquad I_{*}(Y|\bar{R})_{w} \\
                     &= \qquad \SHI(M,\gamma).
                    \end{aligned}
   \]
   The disjointness of the two indexing sets for $s$ means that we
   have
   \[
   \SHI(M'_{1},\gamma'_{1}) \oplus
              \SHI(M'_{2},\gamma'_{2})
              \subset \SHI(M,\gamma).
   \]
   Finally, both summands on the right are non-zero because these
   sutured manifolds are taut.
\end{proof}

\begin{corollary}\label{cor:fibered-knots-instanton}
    Let $K$ be a non-trivial knot in $S^{3}$. Suppose that the symmetrized
    Alexander polynomial $\Delta_{K}(T)$ is monic and that its degree
    (by which we mean the highest power of $T$ that appears) is $g$.
    Then $K$ is fibered if and only if $\KHI(K,g)$ is one-dimensional.
\end{corollary}

\begin{proof}
    The proof given for Corollary~\ref{cor:fibered-knots} (the
    monopole case) needs no alteration, except that the hypothesis on
    the Alexander polynomial has been explicitly included, rather than
    being deduced from Lemma~\ref{lem:Alexander}.
\end{proof}

\begin{corollary}
    Let $K\subset S^{3}$ be a  knot whose Alexander polynomial is
    monic of degree equal to the genus of the knot. Consider the
    irreducible homomorphisms $\rho : \pi_{1}(S^{3}\setminus
    K)\to\SU(2)$ which
    map a chosen meridian $m$ to the element $\bi\in \SU(2)$. If there
    is only one conjugacy class of such homomorphisms, and if these
    homomorphisms are non-degenerate, then $K$ is fibered. \qed
\end{corollary}

\subsection{Non-vanishing theorems in the closed case}
\label{subsec:I-non-vanish}

Theorem~\ref{thm:SHI-non-vanishing} asserts the non-vanishing of
instanton Floer homology for balanced sutured manifolds; but the
theorem does not say anything directly about closed $3$-manifolds $Y$.
Nevertheless, with a little extra input, we obtain the following
result as a corollary.

\begin{theorem}
    \label{thm:I-non-vanishing}
    Let $Y$ be a closed irreducible $3$-manifold containing a closed,
    connected, oriented surface $\bar{R}$ representing a non-zero class in
    second homology. Let $w$ be a hermitian line bundle whose first
    Chern class has odd evaluation on $[R]$. Then $I_{*}(Y|\bar{R})_{w}$ is
    non-zero.
\end{theorem}

\begin{proof}
    Let $M$ be the manifold obtained by cutting $Y$ open along $R$,
    and write the boundary of $M$ as $R_{+} \cup R_{-}$. We regard $M$
    as a sutured manifold, with an empty set of sutures. (The absence
    of sutures means that $M$ fails to be balanced.) Let $N$ be the
    double of $M$. We can regard $R=R_{+}\cup R_{-}$ as a surface in
    the closed manifold $N$. We can ``double'' the line bundle also;
    so we have a line bundle, also denoted by $w$, on $N$. By the
    excision theorem, it will be sufficient to show that
    $I_{*}(N|R)_{w}$ is non-zero. Since $R_{-}$ and $R_{+}$ are
    homologous in $N$ and of equal genus, we have
    \[
            I_{*}(N|R)_{w} = I_{*}(N|R_{+})_{w},
    \]
    so we could equally well deal with $I_{*}(N|R_{+})_{w}$ instead.

    From the proof of Theorem~3.13 of \cite{Gabai}, we have a closed,
    oriented
    surface $T\subset N$ with the following properties. The
    surface $T$ meets $R$ in a non-empty set of circles, and we let
    $T'$ be the surface obtained from $T$ and $R$ by smoothing these
    circles of double points. This $T'$ has the property that by
    cutting $N$ open along $T$ and then decomposing further along a
    non-empty collection of annuli $J$, we arrive at a taut,
    sutured manifold $(N'',\delta'')$. 

    If $T$ intersects both $R_{+}$ and $R_{-}$, then $(N'',\delta'')$
    is balanced. If $T$ intersects only $R_{+}$, say, then
    $(N'',\delta'')$ fails to be balanced, because its boundary
    contains two copies of $R_{-}$: these are components of $\partial
    N''$ which fail to meet $A(\delta'')$, contrary to the definition
    of balanced. If this is what happens, we
    re-attach these two copies of $R_{-}$. We rename the resulting
    manifold as our new $N''$ and proceed. At this point,
    $(N'',\delta'')$ is a balanced sutured manifold.
    
    By Theorem~\ref{thm:SHI-non-vanishing}, we know that
    $\SHM(N'',\delta'')$ is non-zero. We can regard the manifold $N$
    as a closure of $(N'',\delta'')$, but with an auxiliary surface
    that fails to be connected: the auxiliary surface is the
    collection of annuli $J$. But as we argued in the proof of
    Lemma~\ref{lem:S-decomp-special}, a disconnected auxiliary surface
    is as good as a connected one here. We can therefore compute
    $\SHM(N'',\delta'')$ as $I_{*}(N|F)_{w}$, where $F$ is the surface
    in $N$ formed from $R_{\pm}(\delta'')$ when making the closure.
    Thus
    \[
    I_{*}(N|F)_{w} \ne 0.                    
    \]
    This surface $F$ can be identified with $T'$ in the case that $T$
    meets both $R_{+}$ and $R_{-}$. In the case that $T$ meets only
    $R_{+}$, then $F$ is $T'\setminus R_{-}$. In other words, $F$ is
    obtained by smoothing the circles of double points of either
    $T\cup R$ or $T\cup R_{+}$. As in the proof
    Lemma~\ref{lem:S-decomp-special}, the Floer homology
    $I_{*}(N|F)_{w}$ is a direct summand of $I_{*}(N|R_{+})_{w}$. So
    the latter is non-zero, and we are done.
\end{proof}

\begin{corollary}
    If $Y$ is obtained from zero-surgery on a non-trivial knot
    $K\subset S^{3}$, then $I_{*}(Y)_{w}$ is non-zero for an odd line
    bundle $w$. \qed
\end{corollary}

Essentially the same theorem and corollary are proved in
\cite{KM-Witten-P}. But the present proof requires considerably less
geometry and analysis. From Floer's surgery exact triangle, one
obtains, as in \cite{KM-Witten-P},

\begin{corollary}
    If $Y_{1}$ is obtained as $+1$ surgery on $K\subset S^{3}$, then
    $\pi_{1}(Y_{1})$ admits a non-trivial homomorphism to $\SU(2)$.
    In particular, $Y_{1}$ is not a homotopy sphere. \qed
\end{corollary}

This provides a proof of the Property P conjecture that is independent
of the work of Feehan and Leness in \cite{Feehan-Leness} and
independent also of Perelman's proof of the Poincar\'e conjecture.

\subsection{Questions and conjectures}

There are various questions and conjectures which naturally arise. The most obvious of
these is:

\begin{conjecture}
       For balanced sutured manifolds $(M,\gamma)$, the monopole
       and Heegaard groups
       $\SHM(M,\gamma)$ and $\SFH(M,\gamma)$ are isomorphic. When
       tensored with $\C$,  they  are
      both isomorphic to the instanton version, $\SHI(M,\gamma)$.
\end{conjecture}

\noindent
As a special case, we have:

\begin{conjecture}
    With complex coefficients, the knot homologies defined by  Ozsv\'ath-Szab\'o
    and Rasmussen are isomorphic to Floer's instanton homology for
    knots, $\KHI(K)$, as defined here and in
    \cite{Floer-Durham-paper}.
\end{conjecture}

There are various more
modest questions one should ask. We have not shown that the Alexander
polynomial can be recovered from the instanton knot homology groups
$\KHI(K,i)$; but it is natural to conjecture that this is so, just as
in the monopole and Heegaard theories. This may be only a matter of
repeating \cite{Fintushel-Stern-knot} in the instanton
context:

\begin{conjecture}
    The Euler characteristics of the instanton knot homology groups
    $\KHI(K,i)$, for $i=-g,\dots,g$, are the coefficients of the
    symmetrized Alexander polynomial of $K$.
\end{conjecture}

If this conjecture is proved, then the hypothesis on the Alexander
polynomial could be dropped from
Corollary~\ref{cor:fibered-knots-instanton}.

A loose end in our development of $\SHM(M,\gamma)$
is the lack of a complete accounting of \spinc{} structures. The
material of section~\ref{subsec:spinc-decomp} is a step in the right
direction. In \cite{Juhasz-2}, Juh\'asz proves that his Heegaard Floer
homology of sutured manifolds can be decomposed as a direct sum
indexed by the set of
relative \spinc{} structures $\Sp(M,\gamma)$,
and it would be desirable to have a
similar statement for the monopole and instanton cases.

Juh\'asz \cite{Juhasz-3} has considered an extension of the fibering
theorem,  which prompts naturally a conjecture in the instanton
context. Motivated by this, we have:

\begin{conjecture}[\emph{cf}. {\cite{Juhasz-3}}]
    Let $K\subset S^{3}$ be a  knot, and consider the
    irreducible homomorphisms $\rho : \pi_{1}(S^{3}\setminus
    K)\to\SU(2)$ which
    map a chosen meridian $m$ to the element $\bi\in \SU(2)$. Suppose
    that these homomorphisms are non-degenerate and that  the number of
     conjugacy classes of such homomorphisms is less then $2^{k+1}$.
     Then the knot complement $S^{3}\setminus N^{\circ}(K)$ admits a
     foliation of depth at most $2k$, transverse to the torus
     boundary.
\end{conjecture}

The fact that $I_{*}(Y|\Sigma)_{w}$ is of rank $1$ in the case that
$Y$ is a surface bundle of $S^{1}$ with fiber $\Sigma$ is something
that has other applications. For example, combined with Donaldson's
theorem on the existence of Lefschetz pencils
\cite{Donaldson-pencils},  it yields a fairly
direct proof that symplectic $4$-manifolds have non-zero Donaldson
invariants. Essentially the same strategy was used by Ozsv\'ath and
Szab\'o in the Heegaard context. What the argument shows specifically
is that if $X\to S^{2}$ is a symplectic Lefschetz fibration whose
fiber $F$ has genus $2$ or more, and if $w$ is the line bundle
dual to a section, then the Donaldson invariant $D^{w}(F^{n})$ is
non-zero for all large enough $n$ in the appropriate residue class mod
$4$.

Another matter is whether one can relate either the monopole or instanton
knot homologies to the corresponding Floer homologies of the $3$-manifolds obtained
by surgery on the knot, particularly for large integer surgeries. This
is how Heegaard knot homology arose 
in \cite{Rasmussen-thesis}.

In a previous paper \cite{KM-knot-singular}, the authors described
another knot-homology constructed using instantons. The definition
there is distinctly different from the definition of $\KHI(K)$ given
in this paper, because instantons with singularities in codimension-2
were involved. Nevertheless, both theories involve the same
representation variety $\Rep(K,\bi)$. Various versions are defined in
\cite{KM-knot-singular}, but the one most closely related to $\KHI(K)$
is the ``reduced'' variant, called $\RI_{*}(K)$ in
\cite{KM-knot-singular}. Like $\KHI(K)$, the group $\RI_{*}(K)$ is a
Floer homology group, constructed from a Chern-Simons functional whose
set of critical points can be identified with $\Rep(K,\bi)$.
The paper \cite{KM-knot-singular} develops its theory for the gauge
group $\SU(N)$, not just $\SU(2)$, and it would be interesting to
pursue a similar direction with $\SHI(M,\gamma)$ and $\KHI(K)$.

The ``hat'' version of Heegaard Floer homology, for a closed
$3$-manifold $Y$, can also be recovered as a special case of
Juh\'asz's $\SFH$, as shown in \cite{Juhasz-1}. The appropriate
manifold $M$ is the complement of a ball in $Y$, and one takes a
single annular suture on the result $2$-sphere boundary. One can take
this as a definition of a ``hat'' version of monopole Floer homology.
In the instanton case, this leads to essentially the same construction that was
used in \cite{KM-knot-singular} to avoid reducibles: one replaces $Y$
by $Y\# T^{3}$ and takes $w$ to be a line bundle that is trivial on
$Y$ and of degree $1$ on a $T^{2}$ in the $T^{3}$.

Finally, as we mentioned in the introduction,
it is worth asking whether, in the Heegaard theory, the
Floer homology of a balanced sutured manifold $(M,\gamma)$, as defined in
\cite{Juhasz-1}, can
also be recovered as the Heegaard Floer homology of a closed manifold
$Y = Y(M,\gamma)$, of the sort that we have used here. If so, it would
be interesting to know whether the existing proofs of the
decomposition theorems in \cite{Ni-A} and \cite{Juhasz-2}, for
example, can be adapted to prove Floer's excision theorem in the
context of Heegaard Floer theory.

\bibliographystyle{abbrv}
\bibliography{categorify}

\end{document}